\documentclass[12pt,a4paper,english,reqno,a4paper]{amsart}
\usepackage[T1]{fontenc}
\usepackage[utf8]{inputenc}
\usepackage{color}
\usepackage{verbatim}
\usepackage{amstext}
\usepackage{amsthm, amssymb}
\usepackage{stmaryrd}
\usepackage{setspace}
\makeatletter

\numberwithin{equation}{section}
\numberwithin{figure}{section}
\theoremstyle{plain}
\newtheorem{thm}{\protect\theoremname}[section]
  \theoremstyle{remark}
  \newtheorem{rem}[thm]{\protect\remarkname}
  \theoremstyle{plain}
  \newtheorem{lem}[thm]{\protect\lemmaname}
  \theoremstyle{definition}

\usepackage{amsfonts}\usepackage{amsthm}
\usepackage{epsfig}\usepackage{mathrsfs}
\usepackage{amstext}
\usepackage{esint}
\usepackage[T1]{fontenc}
\usepackage[utf8]{inputenc}

\@ifundefined{definecolor}
 {\usepackage{color}}{}

\makeatletter
\newcommand{\Rmnum}[1]{\expandafter\@slowromancap\romannumeral#1@}\makeatother%%%----------------------------------------------------------------------------

\numberwithin{equation}{section}

\allowdisplaybreaks

\newcommand{\set}[1]{\left\{#1\right\}}

\newcommand{\pr}[1]{\left(#1\right)}

\newcommand{\defs}{:=}

\newcommand{\dif}{\mathrm{d}}

\DeclareSymbolFont{lettersA}{U}{pxmia}{m}{it}
\DeclareMathSymbol{\piup}{\mathord}{lettersA}{"19}

\newcommand{\mr}[1]{\mathrm{#1}}

\newcommand{\mcc}{\mathcal{C}}

\newcommand{\mcd}{\mathcal{D}}

\makeatother

\makeatother

\usepackage{babel}
  \providecommand{\definitionname}{Definition}
  \providecommand{\lemmaname}{Lemma}
  \providecommand{\remarkname}{Remark}
\providecommand{\theoremname}{Theorem}

\begin{document}
\title[]
{On Admissible Positions of Transonic Shocks for Steady Isothermal Euler Flows in a Horizontal Flat Nozzle under Vertical Gravity}

\author{Beixiang Fang}

\author{Xin Gao}

\address{B.X. Fang: School of Mathematical Sciences, MOE-LSC, and SHL-MAC, Shanghai
Jiao Tong University, Shanghai 200240, China }

%\address{B.X. Fang: Key Laboratory of Scientific and Engineering Computing
%(Ministry of Education), School of Mathematical Sciences, Shanghai
%Jiao Tong University, Shanghai 200240, China}

\email{\texttt{bxfang@sjtu.edu.cn}}

\address{X. Gao: School of Mathematical Sciences, Shanghai
	Jiao Tong University, Shanghai 200240, China }

\email{\texttt{sjtu2015gx@sjtu.edu.cn}}

% Please provide minimum  5 keywords.
%\subjclass[2010]{35A01, 35A02, 35B20, 35B35, 35B65, 35J56, 35L65, 35L67, 35M30, 35M32, 35Q31, 35R35, 76L05, 76N10}

\date{\today}

% Email address of each of all authors is required.
% You may list email addresses of all other authors, separately.
 \email{}
 \email{}
 \email{}
\begin{abstract}
In this paper we are concerned with the existence of transonic shocks for 2-D steady isothermal Euler flows in a horizontal flat nozzle under vertical gravity. In particular, we focus on the contribution of the vertical gravity in determining the position of the shock front. For steady horizontal flows, the existence of normal shocks with the position of the shock front being arbitrary in the nozzle can be easily established. This paper will try to determine the position of the shock front as the state of the flow at the entrance of the nozzle and the pressure at the exit are slightly perturbed.
	Mathematically, it can be formulated as a free boundary problem for the steady Euler system with vertical gravity, and the position of the shock front is the very free boundary that need to be determined.
	Since the unperturbed normal shock solutions give no information on the position of the shock front, one of the key difficulties is to find where the shock front may appear.
	To overcome this difficulty, this paper proposes a free boundary problem of the linearized Euler system with vertical gravity, whose solution could be an initial approximation for the shock solution with the free boundary being the approximation for the shock front.
	Due to the existence of the vertical gravity, difficulties arise in solving the boundary value problem in the approximate subsonic domain behind the shock front.
	The linearized Euler system is elliptic-hyperbolic composite for subsonic flows, and the elliptic part and the hyperbolic part are coupled in the $0$-order terms depending on the acceleration of gravity $g$.
	Moreover, the coefficients are not constants since the unperturbed shock solution depends on the vertical variable.
	New ideas and techniques are developed to deal with these difficulties and, under certain sufficient conditions on the perturbation, the existence of the solution to the proposed free boundary problem is established as the acceleration of gravity $g>0$ and the perturbation are sufficiently small.
	Then, with the obtained initial approximation of the shock solution, a nonlinear iteration scheme can be constructed which leads to a transonic shock solution with the position of the shock front being close to the initial approximating position.
\end{abstract}

\maketitle

\tableofcontents

\section{Introduction}

This paper concerns the existence of transonic shocks for steady 2-D Euler flows of isothermal gases in a horizontal flat nozzle under the vertical gravity(see Figure \ref{fig:1}). Assume the flow enters the nozzle with a supersonic state and leaves it with a relatively high pressure, then it is expected that a shock front occurs in the nozzle such that the flow pressure rises to coincide with the pressure at the exit.
Then the position of the shock front is one of the most desirable information one would like to know.
This paper is devoted to determine the admissible position of the shock front with a given supersonic state at the entry and the receiver pressure at the exit, under the assumption that the fluid cannot penetrate the nozzle walls, as proposed by Courant and Friedrichs in \cite{CR} for supersonic flows with shocks in a nozzle.
In particular, this paper is going to investigate whether or not the vertical gravity helps to determine the position of the shock front and to show the mechanism if the answer is ``yes''.

\begin{figure}[!h]
	\centering
	\includegraphics[width=0.65\textwidth]{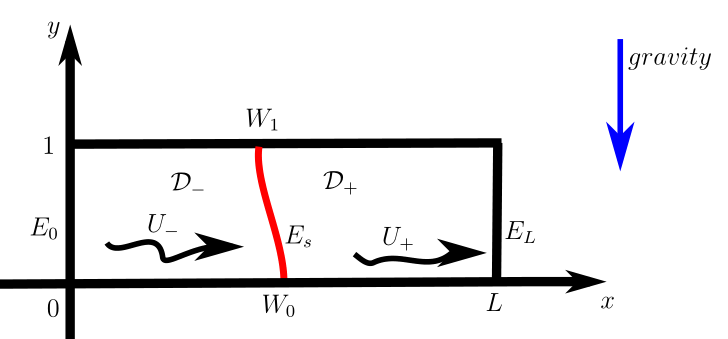}
	\caption{The transonic shock flows in a flat horizontal nozzle under vertical gravity.\label{fig:1}}
\end{figure}

Let $ (x,y) $ be the space variables with $ x $-axis standing for the horizontal direction and $ y $-axis the vertical direction.
Then the motion of the inviscid isothermal gas under vertical gravity is governed by the following system
\begin{align}
&\partial_x (\rho u) + \partial_y (\rho v)=0,\label{eq1}\\
&\partial_x (\rho u^2 + p) + \partial_y (\rho u v)=0,\\
&\partial_x (\rho u v ) + \partial_y (\rho v^2 + p) = - \rho g,\label{a}
\end{align}
where $\rho$ is the density, $p$ is the pressure, $(u,v)^\top $ are the horizontal component and vertical component of the velocity, and $g$ is the acceleration of gravity. For isothermal gases, its state equation is assumed to be $ p(\rho)=\rho $ in this paper. Then the sonic speed $ c^2\defs p'(\rho)\equiv1$.

Then for a shock front occurs in the flow field whose position is  $x= \varphi(y)$,  the following Rankine-Hugoniot conditions (which will be abbreviated as R-H conditions) should be satisfied
\begin{align}
&[\rho u] - \varphi^{'}[\rho v]=0,\label{eq2}\\
&[\rho u^2 + p] - \varphi^{'}[\rho u v]=0,\label{eq3}\\
&[\rho u v ] - \varphi^{'}[\rho v^2 + p]=0,\label{eq4}
\end{align}
where $[\cdot]$ stands for the jump of the corresponding quantity across the shock front.

Let
\begin{align}\label{eq5}
\mathcal{D} :  = \{ (x, y)\in \mathbb{R}^2 : 0 < x < L,\, 0 < y < 1 \}
\end{align}
be the domain bounded by a flat horizontal nozzle with the entrance $E_0$, the exit $E_L$, as well as the walls $W_0$ and $W_1$(see Figure \ref{fig:1}):
\begin{align*}
\begin{split}
	E_0\defs&\set{(x, y)\in \mathbb{R}^2 :  x = 0,\, 0 < y < 1 },\\
	E_L\defs&\set{(x, y)\in \mathbb{R}^2 :  x = L,\, 0 < y < 1 },\\
	W_0\defs&\set{(x, y)\in \mathbb{R}^2 :  0 < x < L,\,  y = 0 },\\
	W_1\defs&\set{(x, y)\in \mathbb{R}^2 :  0 < x < L,\,  y = 1 }.
\end{split}
\end{align*}
The assumption that the fluid cannot penetrate the nozzle boundary yields the following slip boundary condition on  $W_0$ and $W_1$:
\begin{equation}\label{eq:slip_bdry_cond}
	v=0.
\end{equation}
Then the existence problem of the transonic shocks could be formulated as follows.

\vskip 0.5cm

\textbf{The Free Boundary Problem {\bf {$\llbracket \textit{SP}\rrbracket $}}. }

\vskip 0.3cm

Let the independent flow parameters be denoted by $U\defs(p,\theta, q)^\top$, where $\theta = \arctan \displaystyle\frac{v}{u}$ is the flow angle, and $q= \sqrt{u^{2} + v^{2}}$ is the magnitude of the flow velocity.
Given a supersonic state $ U=U_{\mr{in}}(y) $ at the entrance $ E_{0} $, and a relatively high pressure $ p=P_{\mr{out}}(y) $ at the exit $ E_{L} $, whether or not there exists a shock solution $ U=U(x,y) $ in $ \mathcal{D} $ to the 2-D steady Euler system \eqref{eq1}-\eqref{a}, with the position of the shock front being
\begin{equation*}
	E_s\defs\set{(x, y)\in \mathbb{R}^2 :  x = \varphi(y),\, 0 < y < 1},
\end{equation*}
 such that the R-H conditions \eqref{eq2}-\eqref{eq4} are satisfied on $ E_s $, and the boundary condition \eqref{eq:slip_bdry_cond} holds on $W_0$ and $W_1$ (see Figure \ref{fig:1}).

%\vskip 0.5cm
\subsection{Steady normal shock solutions in a flat nozzle}

We first show the existence of special shock solutions to the problem {\bf {$\llbracket \textit{SP}\rrbracket $}} for horizontal flows.
The special solutions can be established under the following assumptions:
\begin{enumerate}
	\item[(H1)] The velocity directions are horizontal for the flows both ahead of and behind the shock front, namely, $ v\equiv 0 $ in $ \mcd $. Then the shock front is a vertical straight line such that $ \varphi'(y)\equiv 0 $ (see Figure \ref{fig:2}).
	\item[(H2)] The states for the flows both ahead of and behind the shock front depend only on the vertical variable $ y $, and is independent of the horizontal variable $ x $. That is, $ U=U(y) $.
\end{enumerate}
Let $p_0$, $ q_0$ are positive constants and $q_0 > 1$. Then it can be easily verified that
\begin{equation}\label{eq:bgsolu-super}
\bar{U}_-(y)=(\bar{p}_-(y),\bar{\theta}_-(y), \bar{q}_-(y))^\top : = ( p_0 \exp(-gy), 0, q_0)^\top,
\end{equation}
satisfies the Euler system \eqref{eq1}-\eqref{a}, which describe a horizontal supersonic flow.
Then, under the assumption (H1) that  $ \varphi'(y)\equiv 0 $, the R-H conditions $\eqref{eq2}$-$\eqref{eq3}$ become, with $ \bar{U}_+(y)=(\bar{p}_+(y), 0, \bar{q}_+(y))^\top $ being the state behind the shock front,
\begin{align}
&[\bar{\rho}\bar{q}]= \bar{\rho}_{+} \bar{q}_{+} - \bar{\rho}_{-} \bar{q}_{-}=0,\label{eq7}\\
&[\bar{p}+\bar{\rho}\bar{q}^2]  =  (\bar{p}_{+}+\bar{\rho}_{+} \bar{q}_{+}^2) - (\bar{p}_{-}+\bar{\rho}_{-} \bar{q}_{-}^2)=0.\label{eq7s}
\end{align}
Then it follows that
\begin{equation}\label{10}
\begin{split}
	\bar{p}_+ (y) &= \bar{\rho}_+ (y) = \bar{p}_-(y) \bar{q}_-^2(y) = \bar{p}_-(y) {q}_0^2,\\
	\bar{q}_+(y) &= \frac{1}{\bar{q}_-(y)} = \frac{1}{{q}_0} < 1.
\end{split}
\end{equation}
It can also be verified that $ \bar{U}_+(y) $ also satisfies the Euler system \eqref{eq1}-\eqref{a}, which describe a horizontal subsonic flow.

Thus, for any  $\bar{x}_s \in (0,L)$ such that the position of the shock front being
\[
	\bar{E}_s \defs \set{(x, y)\in \mathbb{R}^2 :  x = \bar\varphi(y)\equiv\bar{x}_s,\ 0<y<1},
\]
$ \pr{\bar{U}_-(y);\ \bar{U}_+(y);\ \bar\varphi(y)} $ consists a transonic normal shock solution to the problem {\bf {$\llbracket \textit{SP}\rrbracket $}}( see Figure \ref{fig:2}), with $ U_{\mr{in}}(y)\defs \bar{U}_-(y)$ and $ P_{\mr{out}}(y)\defs \bar{p}_+ (y) $, in the sense that
\begin{equation}\label{eq:special-solu}
	\bar{U}(x,y)\defs
	\begin{cases}
		\bar{U}_-(y), & \text{ for }0 < x < \bar{x}_s,\, 0 < y < 1,\\
		\bar{U}_+(y), & \text{ for }\bar{x}_s < x < L,\, 0 < y < 1.
	\end{cases}
\end{equation}
In this paper, the subscript ``$ - $'' will represent the parameters of the flow ahead of the shock front and the subscript ``$ + $'' behind of the shock front.
\begin{figure}[!h]
	\centering
	\includegraphics[width=0.6\textwidth]{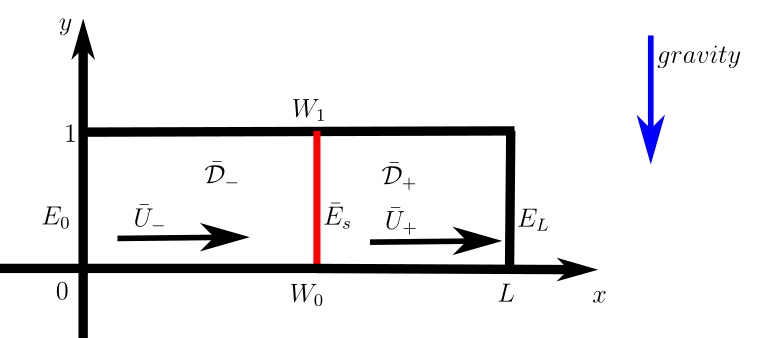}
	\caption{The transonic shock flows in the flat nozzle.\label{fig:2}}
\end{figure}

\begin{rem}
	Analogous to the transonic planar normal shocks for steady Euler flows in a flat nozzle without gravity, namely, $ g=0 $, the position $ \bar{E}_s $ of the shock front could be arbitrary in $ \mcd $ since $\bar{x}_s \in (0,L)$ could be arbitrary, and the subsonic state $ \bar{U}_+(y) $ behind the shock front is uniquely determined by the supersonic state $ \bar{U}_-(y) $ given by \eqref{eq:bgsolu-super}.
%	Obviously, for each $\bar{x}_s \in (0,L)$, $ (\bar{U}_-;\ \bar{U}_+;\ \bar{x}_s)$ gives a plane transonic shock solution to the steady Euler system \eqref{eq1}-\eqref{a}. Since the state $ \bar{U}_- $ in the supersonic region $D_-$ is given, then the state of the flow behind the shock front $ \bar{U}_+$ in the subsonic region $D_+$ is uniquely determined by \eqref{10}, but the position of the plane shock front could be arbitrary in the flat nozzle ${D} $.
\end{rem}

\begin{rem}
	In case $ \bar{q}_-(y) $, with $ \inf\limits_{0<y<1}\bar{q}_-(y)>1 $, is a function depending only on $ y $, special shock solutions  $ \pr{\bar{U}_-(y);\ \bar{U}_+(y);\ \bar\varphi(y)} $ to the problem {\bf {$\llbracket \textit{SP}\rrbracket $}} could also be established with $ \bar{U}_+(y) $ being determined by \eqref{10}, and the position of the shock front could also be arbitrary. This fact means that in general the perturbation of the horizontal component of the velocity does not help to determine the position of the shock front.
%	In the special shock solutions $ \pr{\bar{U}_-(y);\ \bar{U}_+(y);\ \bar\varphi(y)} $, it is worth of pointing out that  $ \bar{q}_-(y)>1 $ could also be any functions depending only on $ y $.
\end{rem}

\begin{rem}
	In case the flat nozzle boundary is slightly perturbed, by applying the ideas and techniques in \cite{FB63}, it turns out that, as the acceleration of gravity $ g>0 $ is small, the primary ingredient that helps to determine the position of the shock front is contributed by the perturbation of the nozzle boundary as well as the receiver pressure at the exit. That is, the contribution of the vertical gravity is covered and could not be observed clearly. Hence, in order to show the contribution of the vertical gravity, the flat nozzle boundary will not be perturbed in this paper.
\end{rem}

\begin{rem}
It is worth of pointing out that, for polytropic gases $ p=A(S)\rho^{\gamma} $ with the entropy $ S $ and the adiabatic exponent $\gamma>1$, there is no shock solutions to the problem {\bf {$\llbracket \textit{SP}\rrbracket $}} satisfying the assumptions (H1) and (H2).
\end{rem}

\subsection{The small perturbation problem}

Based on the established special solutions defined by \eqref{eq:special-solu}, this paper is going to investigate the mechanism how the vertical gravity contributes to determine the position of the shock front by slightly perturbed the pressure at the entrance and the exit of the flat nozzle.
Then the problem {\bf {$\llbracket \textit{SP}\rrbracket $}} is further described as the small perturbation problem  {\bf {$\llbracket \textit{FBP}\rrbracket $}} below with more detailed boundary data.

\vskip 0.5cm

\textbf{The small perturbation problem {\bf {$\llbracket \textit{FBP}\rrbracket $}}. }

\vskip 0.3cm
Let
\begin{equation}
	\begin{split}
		U_{\mr{in}}(y)&\defs \bar{U}_{-}(y) + \sigma( P_I(y),0,0)^\top, \\
		P_{\mr{out}}(y)&\defs \bar{p}_+(y) +{P}_e(y;\ g,\sigma),
	\end{split}
\end{equation}
where $ \sigma>0 $ and $ g>0 $ are sufficiently small constants, $ P_I(y)\in\mcc^{2,\alpha}(\bar{\mathbb{R}}_+) $ is a given function with $\alpha\in(0,1)$, and $P_e(y;\ g,\sigma)\in\mcc^{2,\alpha}(\bar{\mathbb{R}}_+)$ is a given function of $ y $ with parameters $ g>0 $ and $ \sigma>0 $.
Then try to determine a transonic shock solution $ \pr{ U_-(x,y);\ {U}_+(x,y);\ \varphi(y)} $ (see Figure \ref{fig:1}) to the problem {\bf {$\llbracket \textit{SP}\rrbracket $}} in the sense that:
%the states of the fluid $ U $ in the nozzle with a single shock front  $ \{x= \varphi(y)\}$  such that:
\begin{enumerate}
	\item The position of the shock front is
	\begin{equation}
		E_s\defs\set{(x, y)\in \mathbb{R}^2 :  x = \varphi(y),\, 0 < y < 1},
	\end{equation}
	and the domain $\mathcal{D}$ is divided into two parts by $E_s$:
	\begin{equation}\label{eq10}
	\begin{aligned}
	& \mathcal{D}_- = \{ (x, y)\in \mathbb{R}^2 : 0 < x < \varphi(y),\, 0 < y < 1\},\\
	& \mathcal{D}_+ = \{ (x, y)\in \mathbb{R}^2 : \varphi(y) < x < L,\, 0 < y < 1 \},
	\end{aligned}
	\end{equation}
	where $ \mathcal{D}_- $ is the region of the supersonic flow ahead of the shock front, while $ \mathcal{D}_+ $ is the region of the subsonic flow behind it.
	
	\item $ U(x,y) =  U_-(x,y)$ satisfies the Euler system $\eqref{eq1}$-$\eqref{a}$ in $ \mathcal{D}_- $,  the boundary conditions at the entry of the nozzle
	\begin{align}\label{eq11}
	U_-=U_{\mr{in}}(y),\quad \text{on}  \quad  E_0,
	\end{align}
	and the slip boundary condition on the walls of the nozzle
	\begin{align}\label{eq11-}
	\theta_- = 0, \quad \text{on} \quad (W_0\cup W_1)\cap \overline{\mathcal{D}_-};
	\end{align}
	
	\item $ U(x,y)= U_+(x,y) $ satisfies the Euler system $\eqref{eq1}$-$\eqref{a}$ in $ \mathcal{D}_+$, the slip boundary condition on the walls of the nozzle
	\begin{align}\label{eq14}
	\theta_+ = 0, \quad \text{on} \quad (W_0\cup W_1)\cap \overline{\mathcal{D}_+},
	\end{align}
	and the given pressure at the exit of the nozzle
	\begin{align}\label{eq14+}
	p_+ = P_{\mr{out}}(y),\quad \text{on} \quad E_L;
	\end{align}
	
	\item $(U_- , U_+)$ satisfies the R-H conditions $\eqref{eq2}$-$\eqref{eq4}$ across the shock front $ E_s $. %  we denote above free boundary problems as {\bf {$\textit{FBP}$}}.
\end{enumerate}

%\vskip 2cm

This paper is going to show the existence of a transonic shock solution to the small perturbation problem {\bf {$\llbracket \textit{FBP}\rrbracket $}} for certain given functions $P_I$ and $P_e$.

Let $\alpha\in(0,1)$. Suppose that $P_I\in\mcc^{2,\alpha}(\bar{\mathbb{R}}_+)$ is a given function satisfying
\begin{enumerate}
\item
For some constant $C_I >0$ independent of $g$,
\begin{align}\label{PCI}
  \|P_I\|_{\mcc^{2,\alpha}(\bar{\mathbb{R}}_+)}\leq C_I;
\end{align}
	\item
For some constant $C_{I0} >0$ independent of $g$,
\begin{align}\label{P1.21}
  P_I(y) \geq C_{I0}, \quad  \text{for any} \quad y\in [0,1];
\end{align}

	\item
Let $\frac{\partial_y  P_I(y)}{P_I(y)} = \wp(y)$. Then $\wp(y)$ satisfies
\begin{align}
 \wp(0) =& \wp(1)= -g,\label{P_I}\\
 \wp(y) <& -g,\quad \text{for} \quad y\in(0,1).\label{P_I1}
\end{align}
Moreover, for any $[\tau_1, \tau_2] \subsetneqq (0,1)$ and
the constants $\tau_1, \tau_2>0$ independent of $g$,
  there exist uniform constants $C_{I1}$ and $C_{I2}$ independent of $g$, such that
\begin{align}\label{P1.23}
  - C_{I1}\leq  \wp(y)\leq -C_{I2},\quad \text{for}\quad y\in[\tau_1, \tau_2].
\end{align}

\end{enumerate}

Let $P_e$ be a function with the following form
\begin{align}\label{PE}
	{P}_e(y;g,\sigma) \defs \sigma\cdot q_0^2 P_I( y) + g\sigma\cdot  q_0^2 P_E(y),
\end{align}
where $P_E\in\mcc^{2,\alpha}(\bar{\mathbb{R}}_+)$ is a given function at the exit of the nozzle.
\begin{rem}
There exist functions $P_I$ satisfying \eqref{PCI}-\eqref{P1.23}. For example, let
\begin{align*}
	\wp(y) =  \begin{cases} \displaystyle\frac{g-1}{\tau_1}\Big(\displaystyle\frac{1}{\tau_1^2} y^3 - \displaystyle\frac{3}{\tau_1}y^2 +3y \Big) -g, &y \in [0,\tau_1]\\
	-1, & y \in [\tau_1, \tau_2]\\ \displaystyle\frac{1-g}{(1-\tau_2)^3}\Big( y^3 - {3\tau_2}y^2 + {3\tau_2^2}y -{\tau_2^3}\Big) -1. & y \in [\tau_2, 1]
	\end{cases}
	\end{align*}
Then it is easy to find functions for $P_I$ that satisfies the conditions \eqref{PCI}-\eqref{P1.23}.
	%\begin{align*}
%	\wp(y) =  \begin{cases}
%	10(g-1)\Big(100 y^3 -30y^2 +3y \Big) -g, &y \in [0,\frac{1}{10}]\\
%	-1, & y \in [\frac{1}{10}, \frac{9}{10}]\\
%	10(1-g)\Big(100 y^3 -270y^2 +243y \Big) -730+729g. & y \in [\frac{9}{10}, 1]
%	\end{cases}
%	\end{align*}
%Then it is easy to find functions for $P_I$ that satisfies the conditions \eqref{PCI}-\eqref{P1.23}.

\begin{rem}
	\eqref{PCI}-\eqref{PE} are sufficient conditions on $ P_I(y) $ and $ P_e(y;\ g,\sigma) $, under which the position of the shock front can be determined and the existence of a transonic shock solution to the problem {\bf {$\llbracket \textit{FBP}\rrbracket $}} can be established.
These sufficient conditions show a mechanism how the vertical gravity contributes to determine the position of the shock front. In particular, the conditions \eqref{PCI}-\eqref{P1.23} yield that the break-down of the balance between the pressure and the vertical gravity, such that the velocity direction is deflected and the flow is no longer horizontal. Then it will be observed that the action of the vertical gravity is not canceled, which contributes to determine the position of the shock front.
\end{rem}
	%with $h_i(y) < 0(i=1,2,3)$ and
%	\begin{align}
%	&h_1(0) = -g,\quad h_3(1) = -g,\\
%	&\partial_y^m h_2\Big(\frac{1}{10}\Big) = \partial_y^m h_1\Big(\frac{1}{10}\Big),\\
%	&\partial_y^m h_2\Big(\frac{9}{10}\Big) =\partial_y^m h_3 \Big(\frac{9}{10}\Big),\quad m=0,1,2.\label{hi}
%	\end{align}
%	\item
%	\begin{align}\label{PE}
%	{P}_e(L,y;g,\sigma) = q_0^2 P_I(y)\sigma +q_0^2 P_E(L,y)g\sigma,
%	\end{align}
%	where $\sigma >0 $ is a sufficiently small constant and $\sigma\neq g$, $P_E$ is a given function and $P_E\in\mcc^{2,\alpha}(\bar{\mathbb{R}}_+)$.

\end{rem}

Define
\begin{align*}
  {\Omega }^\aleph:  = \{ (\xi, \eta)\in \mathbb{R}^2 : 0 < \xi < \xi_1, \,0 < \eta < 1\},
\end{align*}
with $\xi_1 = \displaystyle\frac{\sqrt{\bar{M}_-^2 -1}}{p_0 q_0}$.

Let $({p}^\aleph, {\theta}^\aleph)^\top$ be the solution to the following problem in the domain ${\Omega}^\aleph$:
\begin{align}
  &\partial_\eta {p}^\aleph + q_0 \partial_\xi {\theta}^\aleph = 0,\label{eq896*}\\
 &\partial_\eta {\theta}^\aleph - \frac{1 - \bar{M}_-^2}{p_0^2 q_0^3} \partial_\xi {p}^\aleph = 0,\label{eq897*}
\end{align}
with the initial-boundary conditions
\begin{align}
{\theta}^\aleph(0,\eta) =& 0, \quad {p}^\aleph(0,\eta) = P_I(\eta),\\
{\theta}^\aleph(\xi,0) =& 0,  \quad {\theta}^\aleph(\xi,1) = 0. \label{eq899*}
\end{align}
Let
\begin{align}\label{Rg*}
  \mathcal{R}_{g\sigma}^\natural(\xi)\defs & {K}\cdot \int_0^1 \int_0^\xi {\theta}^\aleph(\tau,\eta)\dif \tau \dif \eta,
  \end{align}
  where
  \begin{align}\label{defineK}
    {K}: = - \displaystyle\frac{(q_0^2 -1)^2}{p_0 q_0} < 0.
  \end{align}
\begin{figure}[!h]
	\centering
	\includegraphics[width=0.35\textwidth]{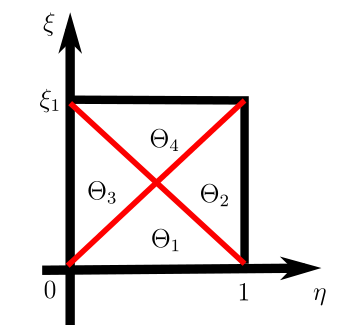}
	\caption{The characteristics for the equations \eqref{eq896*} and \eqref{eq897*}.\label{fig:5}}
\end{figure}

\begin{rem}
For the problem \eqref{eq896*}-\eqref{eq899*}, direct calculations yield that
\begin{align}\label{thealeph}
	{\theta}^\aleph(\xi,\eta) =  \begin{cases}
	 \Theta_1(\xi,\eta), &0\leq \xi \leq \displaystyle\frac{\xi_1}{2}, \, \frac{\xi}{\xi_1}\leq \eta \leq 1-\frac{\xi}{\xi_1}\\
 \Theta_2(\xi,\eta), & \xi_1 (1-\eta)\leq \xi \leq \xi_1 \eta, \,\frac12\leq \eta \leq 1\\
  \Theta_3(\xi,\eta), & \xi_1 \eta \leq \xi \leq \xi_1 (1-\eta),\, 0\leq \eta \leq \frac12\\
	 \Theta_4(\xi,\eta), &\displaystyle\frac{\xi_1}{2} \leq \xi \leq \xi_1,\, 1-\frac{\xi}{\xi_1} \leq \eta \leq \frac{\xi}{\xi_1}
	\end{cases}
	\end{align}
where $\mathcal{K}=\frac{\sqrt{\bar{M}_-^2 - 1}}{2p_0 q_0^2}$,
\begin{align*}
  \Theta_1(\xi,\eta) =&
\mathcal{K} \Big({P}_I\big(\eta -\frac{\xi}{\xi_1}\big) - {P}_I\big(\eta +\frac{\xi}{\xi_1}\big)\Big),\\
 \Theta_2 (\xi,\eta)
    =&\mathcal{K} \Big({P}_I\big(\eta -\frac{\xi}{\xi_1}\big) -  {P}_I\big(2-\eta -\frac{\xi}{\xi_1}\big)\Big),\\
 \Theta_3(\xi,\eta)=&
 \mathcal{K} \Big( {P}_I\big(\frac{\xi}{\xi_1}-\eta\big) - {P}_I\big(\eta +\frac{\xi}{\xi_1}\big)\Big),\\
  \Theta_4(\xi,\eta)=&
 \mathcal{K} \Big( {P}_I\big(\frac{\xi}{\xi_1}-\eta\big)- {P}_I\big(2-\eta -\frac{\xi}{\xi_1}\big)\Big).
\end{align*}
Thus, \eqref{Rg*} implies that
  \begin{equation}\label{expressedR}
	\mathcal{R}_{g\sigma}^\natural(\xi)\defs
	\begin{cases}
		&\flat_1(\xi) , \quad \text{ for}\quad 0 \leq  \xi \leq \frac{\xi_1}{2},\\
&\flat_2 (\xi),\quad \text{ for}\quad   \frac{\xi_1}{2}< \xi \leq {\xi_1},
	\end{cases}
\end{equation}
where
\begin{align}
  \flat_1(\xi)\defs& {K}\Big(\int_0^\xi\int_{0}^{\frac{\tau}{\xi_1}}\Theta_3(\tau,\eta)\dif \eta \dif \tau+ \int_0^\xi\int_{\frac{\tau}{\xi_1}}^{1-\frac{\tau}{\xi_1}}\Theta_1(\tau,\eta)\dif \eta \dif \tau\notag\\
& \qquad \qquad \qquad \qquad \qquad \qquad +\int_0^\xi\int_{1-\frac{\tau}{\xi_1}}^1\Theta_2(\tau,\eta)\dif \eta \dif \tau\Big),\\
\flat_2(\xi)\defs& {K}\Big(\flat_1(\frac{\xi_1}{2}) + \int_{\frac{\xi_1}{2}}^{\xi} \int_{\frac{\tau}{\xi_1}}^{1} \Theta_2(\tau,\eta)
 \dif \eta\dif \tau +  \int_{\frac{\xi_1}{2}}^{\xi} \int_0^{1 - \frac{\tau}{\xi_1}} \Theta_3(\tau,\eta)
 \dif \eta\dif \tau\notag\\
 &\qquad \qquad\qquad \qquad \qquad\qquad  +  \int_{\frac{\xi_1}{2}}^{\xi} \int_{1-\frac{\tau}{\xi_1}}^{\frac{\tau}{\xi_1}} \Theta_4(\tau,\eta)
 \dif \eta\dif \tau
 \Big).
\end{align}

\end{rem}
Then we are going to prove the following Theorem in this paper.
\begin{thm}\label{thm:ExistenceNozzleShocks}
Suppose that \eqref{PCI}-\eqref{PE} hold. $\mathcal{R}_{g\sigma}^\natural $ be the function defined in \eqref{expressedR}, and
\begin{align}
\mathcal{P}_{g\sigma}^\natural\defs  \displaystyle\frac{q_0^2 -1}{p_0^2 q_0}\int_0^1 {P}_E(\eta)\dif \eta.
\end{align}
Assume that
\begin{align}\label{thmR}
		\mathcal{R}_{g\sigma}^\natural(\xi_1 )< \mathcal{P}_{g\sigma}^\natural< \mathcal{R}_{g\sigma}^\natural(0),
	\end{align}
then there exist a sufficiently small constant $ g_0>0 $ such that for any $0<g<g_0$ and
\begin{align}\label{gsigma}
  0< \sigma \leq g^3,
\end{align}
 there exists a transonic shock solution $ \pr{ U_-(x,y);\ {U}_+(x,y);\ \varphi(y)} $ to the problem {\bf {$\llbracket \textit{FBP}\rrbracket $}}.
\end{thm}

\begin{rem}
The condition \eqref{gsigma} can be replaced by
\begin{align}
  0< \sigma \leq g^{2+\varepsilon},
\end{align}
where the constant $\varepsilon>0$. Then $g_0$ will depend on $\varepsilon$.
\end{rem}

\begin{rem}
Under the condition of \eqref{thmR}, there exists a solution $\bar{\xi}_*\in (0, \xi_1)$ such that
\begin{align}\label{inipo}
  \mathcal{R}_{g\sigma}^\natural (\bar{\xi}_*) = \mathcal{P}_{g\sigma}^\natural.
  \end{align}
  It turns out that this solution $\bar{\xi}_*$ helps to obtain the initial approximating position of the shock front. See Theorem \ref{thm:initial_approx_existence} for details.

\end{rem}

%Suppose that the supersonic incoming flow at the entrance $E_0$ is given by
%\begin{equation}
%U_- = \bar{U}_- + (\sigma P_I(y),0,0).
%\end{equation}
%Moreover, at the exit $E_L$, assume that
%\begin{align}
%p_e(L,y;g,\sigma) = \bar{p}_+(y) +{P}_e(L,y;g,\sigma).
%\end{align}

%\begin{rem}
%	By applying the definition of $P_I$, it is easy to see that $\bar{p}_- + \sigma P_I$ is separated from the background solution $\bar{p}_-$. Moreover, choosing such $\sigma P_I$ guarantees that the compatibility conditions hold at the corners of the entrance up to second order. Furthermore, the $\mcc^{2,\alpha}$ regularity of the supersonic solution can be established. In addition, $P_I$ is a strictly decreasing function, which plays a crucial role in determining the position of the shock front.
%\end{rem}

In order to establish the existence of transonic shock solutions to the problem {\bf {$\llbracket \textit{FBP}\rrbracket $}}, one of the key difficulties is to obtain information on the position of the shock front since it can be arbitrary for the special solution \eqref{eq:special-solu}. Motivated by the ideas introduced by Fang and Xin in \cite{FB63}, a free boundary problem for the linearized Euler system with vertical gravity will be proposed to obtain an approximating position of the shock front.
Analogous to the problem without the gravity in \cite{FB63}, the linearized equations of the Euler system for the subsonic flow behind the shock front are elliptic-hyperbolic composite and a solvability condition should be satisfied in order that the boundary value problem for the elliptic sub-system has a solution, which is employed to determine the approximating position of the shock front.
However, there are differences in the linearized system that bring new difficulties.
The first difference is the existence of the 0-order terms that depend on the acceleration of gravity $ g $
 in the linearized Euler system. It leads to the coupling between the elliptic part and the hyperbolic part.
The second difference is the variable coefficients of the linearized system since the unperturbed shock solution depends on the vertical variable $ y $.
Both differences are brought by the existence of the vertical gravity.
They couple together and make it more difficult to deduce the solvability condition for the elliptic sub-problem and to further analyse its relation with the approximation position of the shock front.
To overcome the difficulties brought by these differences, an auxiliary system is introduced to help deduce the solvability condition for the linearized problem of elliptic-hyperbolic coupled type for the subsonic
flow, and establish the existence of its solution.
 %to deduce the solvability condition, which turns out to be also sufficient to ensure the existence of solutions to the elliptic sub-problem, as well as the linearized problem for the subsonic flow behind the shock front.
Then further careful analysis on the solvability condition will be carried out to show the existence of the approximating position of the shock front under the prescribed conditions on $ P_I(y) $ and $ P_e(y;\ g,\sigma) $.
Since the nozzle boundary is not perturbed, it turns out that the leading terms in the solvability condition are $g\cdot\sigma$-terms among the higher order terms.
%That is, it is
These $g\cdot\sigma$-terms show the contribution of the vertical gravity in determining the position of the shock front.
Once the initial approximation is obtained, a further nonlinear iteration could be constructed and proved to lead to a transonic shock solution to the problem {\bf {$\llbracket \textit{FBP}\rrbracket $}} if the acceleration of gravity $g>0$ is sufficiently small and $\sigma$ is of order $g^3$.

The flow pattern of gas flows involving a single shock front in a nozzle, which enter the nozzle with a supersonic state and leave with a subsonic state, is one of the fundamental phenomena for nozzle flows.
In the mathematical analysis for it, how the position of the shock front can be determined is one of the key issues.
In \cite{CR}, Courant and Friedrichs first gave a systematic analysis from the viewpoint of nonlinear partial differential equations.
They point out that, the position of the shock front cannot be determined unless additional conditions are imposed at the exit and the pressure condition is suggested and preferred(see \cite[Page 373-374]{CR}).
Since then, in order to establish a rigorous mathematical analysis for the flow pattern, various nonlinear PDE models and different boundary conditions have been proposed, fruitful ideas and methods had been developed, and substantial progresses had been made.
In 1980s, for the unsteady transonic gas flows governed by the quasi-one-dimensional models, in \cite{Liu1982ARMA,Liu1982CMP},
T.P. Liu proved the existence of shocks solutions for certain given Cauchy data, and established a stability theory for them.
In \cite{EmbidGoodmanMajda1984}, Embid-Goodman-Majda showed that, in general, there exist more than one shock solutions for the steady quasi-one-dimensional model.
%In \cite{EmbidGoodmanMajda1984}, Embid-Goodman-Majda showed that there may exist multiple shock solutions for steady quasi-one-dimensional flows in a nozzle with both expanding and contracting portions.
See also, for instance, \cite{ChenGengZhang2009SIAMJMA,RauchXieXin2013JMPA} and references therein for literatures on quasi-one-dimensional nozzle flows.
As to the steady multi-dimensional models such as potential equations or the Euler system, thanks to continuous efforts of many mathematicians, there have been substantial progresses in the past two decades, for instance, see \cite{BF,GCM2007,GK2006,GM2003,GM2004,
GM2007,SC2005,Chen_S2008TAMS,SC2009,CY2008,FG,FL,FB63,LXY2009,
LiXY2009,LiXY2010,LiXinYin2010PJM,LiXinYin2011PJM,LXY2013,
LiuYuan2009SIAMJMA,LiuXuYuan2016AdM,ParkRyu_2019arxiv,Yong,WS94,
XinYanYin2011ARMA,XinYin2005CPAM,XinYin2008JDE,XinYin2008PJM,
HY2008,Yuan2012NA,HZ1}.
Two typical kinds of nozzles are studied.
One is an expanding nozzle of an angular sector or a diverging cone.
In \cite{CR}, Courant and Friedrichs established the unique existence of a transonic shock solution in such a nozzle with given constant pressure at the exit.
Based on this shock solution, in \cite{SC2009} by Chen and in a series of papers \cite{LXY2009,LiXY2009,LXY2013} by Li-Xin-Yin, the well-posedness of shock solutions in an expanding nozzle has been established, with prescribed pressure at the exit as suggested by Courant and Friedrichs.
See also \cite{LiXY2010,LiXinYin2010PJM,LiXinYin2011PJM,Yong,WS94} for related studies on transonic shocks in a 3-D axisymmetric conic nozzle.
The other is a flat nozzle with two parallel walls.
In this case, the existence of planar normal shock solutions can be easily established. However, the position of the shock front cannot be determined since it can be arbitrary in the flat nozzle.
Thus, as the state of the incoming flow or the nozzle boundary is perturbed, since no information is available in advance, catching the position of the shock front is one of the key difficulties.
An idea to deal with this difficulty is presuming that the shock front goes through a fixed point which is given in advance artificially, and spontaneously replacing the pressure condition at the exit by other conditions, for instance, see \cite{GK2006,GM2003,XinYanYin2011ARMA,XinYin2005CPAM,
XinYin2008JDE,XinYin2008PJM}.
%In particular, in \cite{GM2003}, Chen-Feldman applied this idea to prove the existence of transonic shock solutions for multidimensional potential flows with given potential value at the exit. For Euler system, in \cite{GK2006}, Chen-Chen-Song established the existence of transonic shock solutions with given vertical component of the velocity.
%See also \cite{XinYanYin2011ARMA,XinYin2005CPAM,XinYin2008JDE,XinYin2008PJM}.
Recently, in \cite{FB63}, Fang-Xin proposed another idea to determine the position of the shock front with the pressure condition at the exit. They proposed a free boundary problem for the linearized Euler system whose solution could be taken as an initial approximation for the transonic shock solution, including the approximating position of the shock front. Then a nonlinear iteration scheme starting from the initial approximation could be designed, which leads us to a shock solution. In the above literatures, the exterior forces are neglected.
%To the best of our knowledge, there is no mathematical analysis on the influence of the gravity.
In this paper, the force of gravity will be taken into account and it will be investigated whether and how the vertical gravity contributes to determine the position of the shock front for the flow in a horizontal nozzle.

\subsection{Outline of the paper.}

The rest of the paper is organized as follows. In Section 2, the problem {\bf {$\llbracket \textit{FBP}\rrbracket $}} is
reformulated via the Lagrangian transformation. Then a free boundary problem of the linearized
Euler system based on the background normal shock solution
is proposed in order to obtain an initial approximation of the shock solution. Moreover, the main theorems are stated. In Section 3,
a preliminary solving boundary value problems of a typical elliptic-hyperbolic composite system is given, solvability conditions will be described, and the existence as well as the a prior estimates will be established. They will be employed later in solving the linearized problem for the subsonic flows.
 In Section 4, with the help of the preliminary in Section 3, the existence of the initial approximation for the shock solution can be established, and, the approximate position of the shock front can also be determined by applying the solvability condition given in Section 3. Based on the initial approximation, in Section 5, a nonlinear iteration scheme will be described. Finally, in Section 6, the nonlinear iteration scheme will be verified to
be well-defined and contractive, which concludes the proof for the main theorem.

\section{The Lagrange transformation and the main results}
In this section, the Lagrange
transformation will be introduced to straighten the streamline and reformulate  the problem {\bf {$\llbracket \textit{FBP}\rrbracket $}}. Then, the free boundary problem for the linearized Euler system will be introduced, which will be used to determine an initial approximation of the shock solution. Finally, the main theorems, describing the existence of the initial approximation and the transonic shock solution, are presented.

\subsection{Reformulation by the Lagrange transformation.}
For steady flows, the streamlines coincide with the characteristics associating to the linearly degenerate eigenvalue of the Euler system. It is useful to employ the Lagrange transformation to straighten the streamlines which turns out to be crucial for the regularity analysis of the solution in the subsonic region.
We describe formally the Lagrange transformation below and refer the readers to, for instance, \cite{SC2005,LXY2009,LXY2013} and references therein for more details.

Let
\begin{equation}\label{eq17}
\left\{
\begin{array}{l}
\xi = x,\\
\eta = \int_{(0,0)}^{(x,y)}\rho u(s,t) \dif t - \rho v(s,t) \dif s.
\end{array}
\right.
\end{equation}
%Direct calculations yield that the Jacobian of the transformation is
%\begin{equation}\label{eq19}
% \frac{\partial(\xi,\eta)}{\partial(x,y)} = \rho u,
%\end{equation}
%obviously, it is invertible if and only if $\rho u \neq 0 $.
Under this transformation, the equations $\eqref{eq1}$-$\eqref{a}$ become
\begin{align}
&\partial_\xi \Big(\displaystyle\frac{1}{\rho u}\Big) - \partial_\eta \Big(\displaystyle\frac{v}{u}\Big) = 0,\label{eq22}\\
&\partial_\xi \Big(u+ \displaystyle\frac{p}{\rho u}\Big) - \partial_\eta \Big(\displaystyle\frac{pv}{u}\Big) = 0,\label{eq23}\\
&\partial_\xi v + \partial_\eta p + \displaystyle\frac{g}{u} = 0.\label{eq24}
\end{align}
Under the Lagrange transformation, the upper boundary $W_1$ becomes $ \set{\eta = \eta_0} $ with
\begin{align*}
  \eta_0 = q_0 \int_0^1 \Big(\bar{p}_- (y)+ \sigma P_I (y)\Big) \dif y,
\end{align*}
which depends on the quantities $g$ and $\sigma$. Obviously, the value of $\eta_0$ changes as $g$ and $\sigma$ change.
Therefore, it would be better to further introduce the following transformation such that $W_1$ becomes a fixed boundary $ \set{\tilde{\eta} = p_0 q_0} $ independent of $g$ and $\sigma$:
%Thus one may first convert it into a fixed boundary. Furthermore, the following transformation will be employed
\begin{equation}\label{eq250000}
\left\{
\begin{array}{l}
\tilde{\xi} = \xi,\\
\tilde{\eta} = \displaystyle\frac{p_0 q_0}{\eta_0} \eta,
\end{array}
\right.
\end{equation}
where
\begin{equation}
\begin{aligned}
  \displaystyle\frac{p_0 q_0}{\eta_0}  \defs  1 + H_1(g,\sigma),
\end{aligned}
\end{equation}
with
\begin{align}\label{q4}
  H_1(g,\sigma) \defs \frac{p_0 \Big(1+ \displaystyle\frac{1}{g}(e^{-g}-1)\Big) - \sigma\cdot \int_{0}^{1}P_I(y)\dif y}{-p_0\displaystyle\frac{1}{g}\Big(e^{-g}-1\Big) + \sigma\cdot \int_{0}^{1}P_I(y)\dif y}.
\end{align}
%Direct calculations yield that the Jacobian of the transformation is
%\begin{equation}\label{eq19}
% \frac{\partial(\xi,\eta)}{\partial(x,y)} = \rho u,
%\end{equation}
%obviously, it is invertible if and only if $\rho u \neq 0 $.
Then the equations \eqref{eq22}-\eqref{eq24} are reformulated as
 \begin{align}
&\partial_{\tilde{\xi}} \Big(\displaystyle\frac{1}{\tilde{\rho} \tilde{u}}\Big) - \Big(1+H_1(g,\sigma)\Big)\partial_{\tilde{\eta}} \Big(\displaystyle\frac{\tilde{v}}{\tilde{u}}\Big) = 0,\label{eq6667}\\
&\partial_{\tilde{\xi}} \Big(\tilde{u}+ \displaystyle\frac{\tilde{p}}{\tilde{\rho} \tilde{u}}\Big) - \Big(1+H_1(g,\sigma)\Big)\partial_{\tilde{\eta}} \Big(\displaystyle\frac{\tilde{p}\tilde{v}}{\tilde{u}}\Big) = 0,\label{eq6668}\\
&\partial_{\tilde{\xi}} \tilde{v} + \Big(1+H_1(g,\sigma)\Big)\partial_{\tilde{\eta}} \tilde{p} + \displaystyle\frac{g}{\tilde{u}} = 0.\label{eq06669}
\end{align}
For simplicity of the notations, we shall drop `` $ \tilde{} $ '' in the sequel arguments. In addition, without loss of generality, we may assume $p_0 q_0 = 1$.

Further computations yield that the equations \eqref{eq6667}-\eqref{eq06669} can be rewritten as the following form:
\begin{align}
&\Big(1+H_1(g,\sigma)\Big)\partial_{{\eta}} {p} - \displaystyle\frac{\sin {\theta}}{{\rho} {q}}\partial_{{\xi}} {p} + {q} \cos {\theta} \partial_{{\xi}} {\theta} + \displaystyle\frac{\cos {\theta}}{{q}} g =0,\label{eq260000}\\
&\Big(1+H_1(g,\sigma)\Big)\partial_{{\eta}} {\theta} - \displaystyle\frac{\sin {\theta}}{{\rho} {q}} \partial_{{\xi}} {\theta} - \displaystyle\frac{\cos {\theta}}{{\rho} {q}}\displaystyle\frac{1 - {{M}}^2}{{\rho} {{q}}^2 } \partial_{{\xi}}{ p} - \displaystyle\frac{\sin {\theta}}{{\rho} {{q}}^3} g =0,\label{b}\\
&\partial_\xi \Big(\frac12 q^2\Big) + \frac{1}{\rho} \partial_\xi p + g \tan {\theta}= 0.\label{b1}
\end{align}
The equation $\eqref{b1}$ can be replaced by the following form:
\begin{align}
  \partial_{{\xi}} B + g \tan {\theta}=0,
\end{align}
where the Bernoulli constant $B = \displaystyle\frac12 {q}^2 + {i}$ and $i=\ln \rho$ being the enthalpy.

\begin{rem}
	It is easy to see that \eqref{b1} is a transport equation and is hyperbolic. Moreover, the equations \eqref{eq260000} and \eqref{b} can be rewritten in the matrix form as below:
	\begin{equation}\label{b2}
	A_1(U)\partial_{\xi} ( p,\theta)^\top+ A_2\partial_{\eta} ( p, \theta )^\top + a(U)= 0,
	\end{equation}
	where $a(U) = \Big( \displaystyle\frac{\cos\theta}{q}g, \,  - \displaystyle\frac{\sin\theta}{\rho q^3}g\Big)^\top$,
	
	\[ A_1(U) = \displaystyle\frac{1}{\rho q}\begin{pmatrix}
	-\sin\theta & \rho q^2\cos\theta \\
	\displaystyle\frac{M^2 -1}{\rho q^2}\cos\theta & -\sin\theta
	\end{pmatrix},\quad 	 A_2 = \begin{pmatrix}
	1+H_1(g,\sigma) & 0 \\
	0 & 1+H_1(g,\sigma)
	\end{pmatrix}.\]
	Direct calculations follow that the eigenvalues of \eqref{b2} are
%$A_1(U)-\lambda A_2$ are
	\begin{equation}
	\lambda_\pm = \frac{- \sin\theta \pm \sqrt{M^2 -1}\cos\theta}{\rho q(1+H_1(g,\sigma))}.
	\end{equation}
For supersonic flows, $\lambda_\pm$ are real since the Mach number $M>1$, which implies that the system $\eqref{b2}$ is hyperbolic, while for subsonic flows, $\lambda_\pm $ are a pair of conjugate complex number since the Mach number $M<1$, which implies that the system $\eqref{b2}$ is elliptic. Therefore, the system \eqref{eq260000}-\eqref{b1} is hyperbolic as $ M>1 $, while it is elliptic-hyperbolic composite as $ M<1 $.
\end{rem}

Let
$${\Gamma}_{s} \defs \{ ({{\xi}}, {{\eta}})\in \mathbb{R}^2 : {\xi} = {\psi}({\eta}),\,\,\, 0 <{\eta} <1 \}$$
be the position of a shock front under the transformations \eqref{eq17} and \eqref{eq250000}, then the R-H conditions \eqref{eq2}-\eqref{eq4}
 across the shock front are reformulated as
\begin{align}
&\Big[\displaystyle\frac{1}{\rho u}\Big] + \Big(1+H_1(g,\sigma)\Big)\psi' \Big[\displaystyle \frac{v}{u}\Big] = 0,\label{eq26}\\
&\Big[u + \displaystyle \frac{p}{\rho u}\Big] + \Big(1+H_1(g,\sigma)\Big)\psi'\Big[\displaystyle\frac{p v}{u}\Big]= 0,\label{eq27}\\
&[v] - \Big(1+H_1(g,\sigma)\Big)\psi'[p]=0.\label{eq28}
\end{align}
Applying the equation $\eqref{eq28}$, one can eliminate the quantity $\psi'$ in the equations $\eqref{eq26}$ and $\eqref{eq27}$ respectively, which yields that
\begin{align}
G_1 (U_+, U_-) \defs & \Big[\displaystyle\frac{1}{\rho u}\Big][p] + \Big[\displaystyle \frac{v}{u}\Big] [v]= 0,\label{eq30}\\
G_2 (U_+, U_-) \defs & \Big[u + \displaystyle \frac{p}{\rho u}\Big][p]+ \Big[\displaystyle\frac{p v}{u}\Big][v]= 0,\label{eq31}\\
G_3 (U_+, U_-; \psi') \defs & [v] - \Big(1+H_1(g,\sigma)\Big)\psi'[p]=0.\label{eq33}
\end{align}

\begin{figure}[!h]
	\centering
	\includegraphics[width=0.6\textwidth]{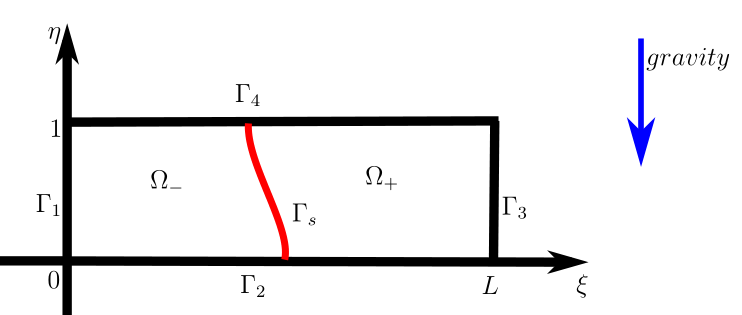}
	\caption{The transonic shock flows in the Lagrangian  coordinate.\label{fig:3}}
\end{figure}

Under the transformations \eqref{eq17} and \eqref{eq250000}, the domain $\mathcal{D}$ becomes
\begin{equation}\label{eq39}
{\Omega }:  = \{ (\xi, \eta)\in \mathbb{R}^2 : 0 < \xi < L, \,0 < \eta < 1\}.
\end{equation}
It is separated by the shock front ${\Gamma}_{s} $
into two parts:
the supersonic region and subsonic region respectively, denoted by,
\begin{equation}\label{eq:supersonic_region_Lagrange}
{\Omega}_- = \{ ({\xi}, {\eta})\in \mathbb{R}^2 : 0 < {\xi} < \psi({\eta}), \,0 < {\eta} < 1\},
\end{equation}
\begin{equation}
{\Omega}_+ = \{ ({\xi}, {\eta})\in \mathbb{R}^2 : \psi({\eta}) < {\xi} < L, \, 0 < {\eta} < 1\}.
\end{equation}
Moreover, the boundaries $E_0, W_0, E_L, W_1$ become
\begin{align}\label{eq40}
&\Gamma_1 = \{ (\xi, \eta)\in \mathbb{R}^2 :  \xi = 0 , \,0 < \eta < 1\},\\
&{\Gamma}_2= \{ (\xi, \eta)\in \mathbb{R}^2 : 0 < \xi < L, \, \eta = 0\},\\
&\Gamma_3 = \{ (\xi, \eta)\in \mathbb{R}^2 :  \xi = L, \,0 < \eta < 1\},\\
&{\Gamma}_4 = \{ (\xi, \eta)\in \mathbb{R}^2 : 0 < \xi < L, \, \eta = 1\},
\end{align}
respectively( see Figure \ref{fig:3}).
%\begin{figure}[H]
%\centering
%\includegraphics[width=8cm]{222}
%\caption{\label{fig:test}}
%\end{figure}\par

%\subsection{}

Then, under the transformations \eqref{eq17} and \eqref{eq250000}, the small perturbation problem {\bf {$\llbracket \textit{FBP}\rrbracket $}} is reformulated as the problem below.
\vskip 0.5cm

\textbf{The free boundary problem {\bf {$\llbracket\textit{FBPL}\rrbracket$}}}

\vskip 0.3cm
Try to determine a transonic shock solution $ \pr{ U_-(\xi,\eta);\ {U}_+(\xi,\eta);\ \psi(\eta)} $ such that %( see Figure \ref{fig:3}) such that:

\begin{enumerate}
	\item $ U_{-}(\xi,\eta) $ satisfies the equations  \eqref{eq260000}-\eqref{b1} in $ \Omega_- $ and the following initial-boundary conditions
\begin{align}
&U_-= U_{\mr{in}}(Y_0(\eta;g,\sigma)),&\quad&\text{on} \quad \Gamma_1,\label{eq44}\\ %\bar{U}_- + ( \sigma P_I(Y_-(0,\eta;g,\sigma)), 0, 0)
&\theta_- = 0 , &\quad &\text{on} \quad (\Gamma_2\cup \Gamma_4)\cap\overline{\Omega_-}\label{eq45},   %\cap\overline{\Omega_-}
\end{align}
where
\begin{align}
  Y_0(\eta;g,\sigma) = \displaystyle\frac{1}{1+H_1(g,\sigma)}\cdot \int_0^\eta \displaystyle\frac{1}{q_0 \rho_-(0,s) } \dif s,
\end{align}
and,
\begin{align}
{ P}_{\mr{in}} (Y_0(\eta;g,\sigma)) \defs & \bar{p}_-(\eta) +P_I^\sharp(Y_0(\eta;g,\sigma)),
\end{align}
with
\begin{align}
\bar{p}_-(\eta) \defs & p_0 - g\cdot \frac{1}{(1+H_1(g,0))q_0}\eta,\\
P_I^\sharp(Y_0(\eta;g,\sigma)) \defs& \sigma \cdot P_I (Y_0(\eta;g,\sigma))\notag\\
 &+g\cdot \frac{H_1(g,\sigma) - H_1(g,0) }{\big(1+H_1(g,0)\big)\big(1+H_1(g,\sigma)\big)q_0}\eta \notag\\
  &+ g\sigma\cdot\int_0^{Y_0(\eta;g,\sigma)} P_I(\tau)\dif \tau;\label{PIsharp}
\end{align}

\item  $ U_{+}(\xi,\eta) $ satisfies the equations \eqref{eq260000}-\eqref{b1} in $ \Omega_+ $ and the following boundary conditions:
\begin{align}
&\theta_+ = 0 , &\quad &\text{on} \quad (\Gamma_2\cup \Gamma_4)\cap\overline{\Omega_+},\label{eq47}\\ %\cap\overline{\Omega_+}
&p_+ = P_{\mr{out}}(Y_L(\eta;g,\sigma);g,\sigma) ,&\quad &\text{on} \quad \Gamma_3,\label{eq48}  %\bar{p}_+(\eta) +  P_e (Y_+(L,\eta;g,\sigma);g,\sigma)
\end{align}
where
\begin{align}
  Y_L(\eta;g,\sigma) = \displaystyle\frac{1}{1+H_1(g,\sigma)}\cdot \int_0^\eta \displaystyle\frac{1}{(\rho_+ q_+ \cos \theta_+)(L,s)} \dif s,
\end{align}
and,
\begin{align}
 P_{\mr{out}}(Y_L(\eta;g,\sigma);g,\sigma)=& \bar{p}_+(\eta) + P_e^\sharp(Y_L(\eta;g,\sigma); g,\sigma),
\end{align}
with
\begin{align}
\bar{p}_+(\eta) \defs & p_0 q_0^2 - g\cdot\frac{ q_0}{1+H_1(g,0)}\eta,\\
P_e^\sharp(Y_L(\eta;g,\sigma); g,\sigma)\defs&   P_e (Y_L(\eta;g,\sigma); g,\sigma)\notag\\
&+g\cdot \frac{\big(H_1(g,\sigma) - H_1(g,0)\big)q_0 }{\big(1+H_1(g,0)\big)\big(1+H_1(g,\sigma)\big)}\eta \notag\\
&+ g\cdot
  \int_0^{Y_L(\eta;g,\sigma)} P_e(\tau;g,\sigma)\dif \tau;\label{Pesharp}
\end{align}

\item On the shock front $ \Gamma_{s} $, $(U_-(\xi,\eta) , U_+(\xi, \eta))$ satisfies the R-H conditions  $\eqref{eq26}$-$\eqref{eq28}$.%, we denote above free boundary problem as {\bf {$\textit{FBPL}$}}.
\end{enumerate}

\begin{rem}
  Under the transformations \eqref{eq17} and \eqref{eq250000}, the states $\bar{U}_\pm (y)$ for the background solution become:
   \begin{align}
    \bar{U}_-(\eta) &= (\bar{p}_-(\eta),\bar{\theta}_-(\eta), \bar{q}_-(\eta))^\top \notag\\
    & \defs \pr{\bar{p}_-(\eta),\ 0,\ q_0}^\top, \label{rho-}\\
%    \bar{p}_-(\eta) =\bar{\rho}_-(\eta) =  p_0 - \displaystyle\frac{g}{(1+ H_1(g,0))q_0}\eta, \quad \bar{\theta}_- = 0, \quad \bar{q}_- = q_0,
   \bar{U}_+(\eta) &= (\bar{p}_+(\eta),\bar{\theta}_+(\eta), \bar{q}_+(\eta))^\top \notag\\
   & \defs \pr{\bar{p}_-(\eta) {q}_0^2,\ 0,\ \frac{1}{{q}_0}}^\top. \label{2.35}
%   \bar{p}_ + (\eta)= \bar{p}_-(\eta) {q}_0^2,\quad \bar{\theta}_+ = 0,\quad\bar{q}_+ = \frac{1}{{q}_0}
\end{align}
Moreover,
\begin{align}\label{q2}
\partial_\eta P_I^\sharp(0)= \partial_\eta P_I^\sharp(1) =g\cdot \frac{H_1(g,\sigma) - H_1(g,0) }{\big(1+H_1(g,0)\big)\big(1+H_1(g,\sigma)\big)q_0}.
\end{align}
It shows that the compatibility conditions hold for the hyperbolic system in $\Omega_-$.

\end{rem}

\subsection{The free boundary problem for the initial approximation.}\

\begin{figure}[!h]
	\centering
	\includegraphics[width=0.6\textwidth]{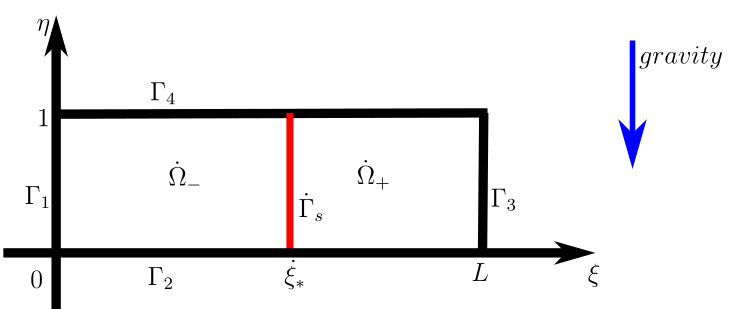}
	\caption{The domain for the initial linearized problem.\label{fig:4}}
\end{figure}

To solve the free boundary problem {\bf {$\llbracket\textit{FBPL}\rrbracket$}}, one of the key issue is to determine the position of the shock front $ \Gamma_{s} $.
However, there is no information one can get from the unperturbed shock solution.
Motivated by the ideas introduced in \cite{FB63} to determine an approximating position of the shock front, we shall propose a free boundary problem for the linearized Euler system as below, whose solution gives an initial approximation of the shock solution.

Assume that the initial approximating position of the shock front is
\begin{equation}\label{eq79}
\dot{\Gamma}_s= \{(\xi,\eta) : \xi = \dot{\xi}_*, \,\,\,0< \eta <1\},
\end{equation}
where $0<\dot{\xi}_*<L$ is unknown and will be determined later (see Figure \ref{fig:4}).
Then the whole domain $ \Omega $ is divided by $ \dot{\Gamma}_s $ into two parts: the supersonic region $\dot{\Omega}_-$ and subsonic region $\dot{\Omega}_+$, denoted respectively
by
\begin{align}\label{eq80}
&\dot{\Omega}_- = \{ (\xi, \eta)\in \mathbb{R}^2 : 0 < \xi < \dot{\xi}_*, \,\,\,0 < \eta < 1\},\\
&\dot{\Omega}_+ = \{ (\xi, \eta)\in \mathbb{R}^2 : \dot{\xi}_* <\xi < L,\,\,\, 0 < \eta < 1\}.
\end{align}
Let $ \dot{U}_- =  (\dot{p}_-, \dot{\theta}_-, \dot{q}_-)^\top$ be the
initial approximate supersonic flow ahead of the shock front governed by the following linearized Euler system at the supersonic state $ \bar{U}_- $ in $ \dot{\Omega}_- $:
\begin{align}
  &\Big(1+H_1(g,0)\Big)\partial_\eta \dot{p}_- + \bar{q}_- \partial_\xi \dot{\theta}_- - \frac{g}{\bar{q}_-^2} \dot{q}_- = H_{2_-}(g,\sigma),\label{eq63}\\
 &\Big(1+H_1(g,0)\Big)\partial_\eta \dot{\theta}_- - \frac{1 - \bar{M}_-^2}{\bar{\rho}_-^2{\bar{q}_-^3}} \partial_\xi \dot{p}_- - \frac{g}{\bar{\rho}_- \bar{q}_-^3}\dot{\theta}_- = 0,\label{eq64}\\
   &\bar{q}_-\partial_\xi \dot{q}_- + \frac{1}{\bar{\rho}_-}\partial_\xi \dot{p}_- + g \dot{\theta}_- = 0,\label{eq65}
\end{align}
where
\begin{align*}
  H_{2_-}(g,\sigma) = g\cdot \displaystyle\frac{H_1(g,\sigma)- H_1(g,0)}{q_0(1+H_1(g,0))}.
\end{align*}
Let $ \dot{U}_+ =  ( \dot{p}_+, \dot{\theta}_+,\dot{q}_+ )^\top$ be the initial approximate subsonic flow behind of the shock front governed by the following linearized Euler system at the subsonic state $ \bar{U}_+ $ in $ \dot{\Omega}_+ $:
\begin{align}
   &\Big(1+H_1(g,0)\Big)\partial_\eta \dot{p}_+ + \bar{q}_+ \partial_\xi \dot{\theta}_+ - \frac{g}{\bar{q}_+^2} \dot{q}_+  = H_{2_+}(g,\sigma),\label{eq66}\\
 &\Big(1+H_1(g,0)\Big)\partial_\eta \dot{\theta}_+ - \frac{1 - {\bar{M}_+}^2}{\bar{\rho}_+^2{{\bar{q}_+}^3}} \partial_\xi \dot{p}_+ - \frac{g}{\bar{\rho}_+ {\bar{q}_+}^3}\dot{\theta}_+ = 0,\label{eq67}\\
  &\bar{q}_+ \partial_\xi \dot{q}_+ + \frac{1}{\bar{\rho}_+}\partial_\xi \dot{p}_+ + g \dot{\theta}_+ = 0,\label{eq68}
\end{align}
where
 \begin{align}
   H_{2_+}(g,\sigma) = g\cdot\displaystyle\frac{\big(H_1(g,\sigma)- H_1(g,0)\big) q_0}{1+H_1(g,0)}.
 \end{align}

Then the following free boundary problem will be employed to determine the initial approximation $(\dot{U}_-$, $\dot{U}_+$, $\dot{\xi}_*)$, and together with the updated approximating shock profile $\dot{\psi}'$.

\textbf{The free boundary problem {\bf {$\llbracket\textit{IFBPL}\rrbracket$}} for the initial approximation}

\vskip 0.3cm

Try to determine $ (\dot{U}_{-}(\xi,\eta),\ \dot{U}_{+}(\xi,\eta),\ \dot{\xi}_*;\ \dot{\psi}'(\eta)) $ in $ \Omega $ such that:
\begin{enumerate}
	\item $ \dot{U}_{-}(\xi,\eta) $ satisfies the linearized equations \eqref{eq63}-\eqref{eq65} in $ \dot\Omega_- $, with the initial-boundary conditions
\begin{align}
   &\dot{U}_- = ( P_I^{*}({Y}_0(\eta;g,\sigma)) , 0,0)^\top,&\quad &\text{on}\quad \Gamma_1,\label{eq82}\\
   &\dot{\theta}_- = 0 ,&\quad &\text{on}\quad (\Gamma_{2}\cup\Gamma_{4})\cap\overline{\dot\Omega_-},\label{eq82B}
  \end{align}
where
\begin{align}\label{qc2}
 P_I^{*}({Y}_0(\eta;g,\sigma))= \frac{1+H_1(g,\sigma)}{1+H_1(g,0)} P_I^\sharp({Y}_0(\eta;g,\sigma))\defs P_I^{*} (\eta),
\end{align}
with $P_I^\sharp$ being defined in \eqref{PIsharp};
	\item   $\dot{U}_+$ satisfies the linearized equations $\eqref{eq66}$-$\eqref{eq68}$ in $ \dot\Omega_+ $, with the boundary conditions
\begin{align}
   &\dot{\theta}_+ = 0 ,&\quad &\text{on} \quad  (\Gamma_{2}\cup\Gamma_{4})\cap\overline{\dot\Omega_+},\label{eq83}\\
   &\dot{p}_+ = P_e^\sharp({Y}_L(\eta;g,0);g,\sigma),&\quad &\text{on} \quad \Gamma_3,\label{eq83a}
   \end{align}
where $P_e^\sharp$ is defined in \eqref{Pesharp} and
\begin{equation}\label{eq141}
  {Y}_L(\eta;g,0) = \displaystyle\frac{1}{1+H_1(g,0)}\cdot \int_0^\eta \displaystyle\frac{1}{\bar{q}_+\bar{\rho}_+(s)} \dif s;
\end{equation}

\item On the shock front $ \dot{\Gamma}_s $, $(\dot{U}_{-},\dot{U}_{+})$ satisfies the following linearized R-H conditions: 	
	\begin{align}
&{\mathbf{\alpha}}_{j+} \cdot {\dot{U}}_+ + {\mathbf{\alpha}}_{j-} \cdot {\dot{U}}_- = 0,  \quad j = 1,2,\label{eq69}\\
  &{\mathbf{\alpha}}_{3+} \cdot {\dot{U}}_+ + {\mathbf{\alpha}}_{3-} \cdot {\dot{U}}_- - (1+H_1(g,0))[\bar{p}] {\dot{\psi}}' = 0,\label{eq74}
\end{align}
where
\begin{align}\label{eq70}
{\mathbf{\alpha}}_{j\pm} = {{\nabla_{{U}_\pm}}}G_i(\bar{U}_+,\bar{U}_-), \quad {\mathbf{\alpha}}_{3\pm} = {{\nabla_{{U}_\pm}}}G_3(\bar{U}_+,\bar{U}_-;0).
\end{align}
\end{enumerate}

\begin{rem}
 Direct computation yields that
\begin{align}
&\mathbf{\alpha}_{1\pm} = \pm\frac{[\bar{p}]}{\bar{\rho}_\pm \bar{q}_\pm}\cdot\Big(-\frac{1}{\bar{\rho}_\pm {\bar{c}}_\pm^2},\, 0, \, -\frac{1}{ \bar{q}_\pm}\Big)^\top,\label{x62}\\
&\mathbf{\alpha}_{2\pm} = \pm\frac{[\bar{p}]}{\bar{\rho}_\pm \bar{q}_\pm}\cdot\Big(1- \frac{\bar{p}_\pm}{\bar{\rho}_\pm \bar{c}_\pm^2},\, 0 ,\, \bar{\rho}_\pm \bar{q}_\pm - \frac{\bar{p}_\pm}{\bar{q}_\pm}\Big)^\top,\label{x61}\\
&\mathbf{\alpha}_{3\pm} = \pm \Big(0,\, \bar{q}_\pm,\, 0\Big)^\top.\label{x60}
\end{align}
\end{rem}

\begin{rem}
By applying \eqref{rho-}-\eqref{2.35} and $p_0 q_0 = 1$, it follows that
  \begin{align}
    {Y}_L(\eta;g,0) ={Y}_0(\eta;g,0), \quad  {Y}_L(\eta;0,0) ={Y}_0(\eta;0,0)=\eta.
  \end{align}
Denote
\begin{align*}
\dot{P}_e^\sharp(\eta)\defs P_e^\sharp({Y}_L(\eta;g,0);g,\sigma).
\end{align*}
Then applying \eqref{PE}, \eqref{PIsharp} and \eqref{Pesharp}, one has
\begin{align*}
   \dot{P}_e^\sharp(\eta) =&q_0^2  {P}_I^\sharp ({Y}_0(\eta;g,0)) + g\sigma\cdot q_0^2 {P}_E ({Y}_L(\eta;g,0)) \notag\\
   &+ g\cdot \int_{0}^{{Y}_0(\eta;g,0)}\Big( P_e(\tau;g,\sigma)- \sigma\cdot q_0^2 P_I(\tau)  \Big)\dif \tau\notag\\
   =&q_0^2  {P}_I^\sharp ({Y}_0(\eta;g,0)) + g\sigma \cdot q_0^2 {P}_E ({Y}_L(\eta;g,0)) + g^2\sigma \cdot q_0^2  \int_{0}^{{Y}_0(\eta;g,0)}P_E(\tau)  \dif \tau.
\end{align*}
Thus, one can deduce that
 \begin{align}\label{iniPE}
   \dot{P}_e^\sharp(\eta) = q_0^2 {P}_I^{*}(\eta) + g\sigma\cdot q_0^2 {P}_E(\eta)+ O(1)g^2 \sigma + O(1)\sigma^2,
 \end{align}
 where $P_I^*$ is defined in \eqref{qc2} and $O(1)$ is a bounded function and depends on $p_0, q_0$, $P_I$, $P_E$, $P_I^{'}$ and $P_E^{'}$.

\end{rem}

\subsection{Main results}\
Before we state the main result, some function spaces will be first introduced.

In the supersonic region, it is natural to introduce the classical H\"{o}lder spaces. For any bounded domain $\Omega \subset \mathbb{R}^n$, $m > 0 $ be an integer, and $0< \alpha < 1$, $\mcc^{m,\alpha}(\Omega)$ denotes the classical H\"{o}lder spaces with the index $(m, \alpha)$ for functions with continuous derivatives up to $m$-th order, equipped with the classical $\mcc^{m,\alpha}(\Omega)$ norm:
\begin{equation}
   \|u\|_{\mcc^{m,\alpha}(\Omega)} : = \sum_{|\mathbf{m}|\leq m} \sup\limits_{\mathbf{x}\in\Omega}|D^{\mathbf{m}} u(\mathbf{x})|+ \sum_{|\mathbf{m}| =  m} \sup\limits_{\mathbf{x},\mathbf{y}\in\Omega; \mathbf{x}\neq \mathbf{y}}\frac{|D^{\mathbf{m}} u(\mathbf{x}) - D^{\mathbf{m}} u(\mathbf{y})|}{|\mathbf{x}-\mathbf{y}|^\alpha},
\end{equation}
where $D^{\mathbf{m}} = \partial_{x_1}^{m_1}\partial_{x_2}^{m_2}\cdots\partial_{x_n}^{m_n}$,
$\mathbf{m}=(m_1,m_2,\ldots, m_n)$ is a multi-index with $m_i\geq 0$ be an integer and $|\mathbf{m}| = \sum\limits_{i =1}^n m_i$.

In the subsonic region, since the boundary of the domain has corner singularities, the Sobolev spaces $W_\beta^{s}(\Omega)$ with $1 \leq \beta < \infty$ will be employed. The index $s$ will take real value, as defined in \cite{PG}, for the trace function on the boundary. Let $s = m+\alpha$, where $m$ is a nonnegative integer and $0<\alpha<1$.
Define
\begin{align}
  \|u\|_{W_\beta^s(\Omega)}: = \left( \|u\|_{W_\beta^m(\Omega)}^\beta + \sum_{|\mathbf{m}| = m}  \int\int_{\Omega\times\Omega}\frac{|D^{\mathbf{m}}u(x) - D^{\mathbf{m}}u(y)|^\beta}{|x - y|^{n+ \alpha \beta}}\dif x\dif y    \right)^{\frac{1}{\beta}}.
\end{align}
It should be pointed that for any $u\in W_\beta^1(\Omega)$, its trace on the boundary belongs to $W_\beta^{1-\frac{1}{\beta}}(\partial\Omega)$.

Moreover, since the Euler system for subsonic flows is elliptic-hyperbolic composite, for the flow state $ U=(p,\theta,q)^\top $, the function spaces for $(p,\theta)^\top$ are different from $q $.
Define
\begin{align}\label{eq099}
 \|{U} \|_{(\dot{\Omega}_+;\dot{\Gamma}_s)}: = \| {p} \|_{W_\beta^1(\dot{\Omega}_+)}+ \|{\theta} \|_{W_\beta^1(\dot{\Omega}_+)}
+ \| {q} \|_{\mcc^0(\dot{\Omega}_+)}+ \|{q} \|_{W_\beta^{1-\frac{1}{\beta}}(\dot{\Gamma}_s)}.
\end{align}
Since the shock front $\Gamma_{{s}}\defs \{{\xi} = {\psi}({\eta})\}$ is a free boundary, then the following coordinate transformation will be employed
\begin{align*}
\mathcal{T} : \begin{cases}
\tilde{\xi} = L + \displaystyle\frac{L - \dot{\xi}_*}{L - {\psi}(\eta)}(\xi - L),\\
\tilde{\eta} = \eta,
\end{cases}
\end{align*}
with the inverse
\begin{align*}
\mathcal{T}^{-1} : \begin{cases}
\xi = L + \displaystyle\frac{L -  {\psi}(\tilde{\eta})}{L - \dot{\xi}_*}(\tilde{\xi} - L),\\
\eta = \tilde{\eta}.
\end{cases}
\end{align*}
Obviously, under this transformation, the free boundary $\Gamma_{{s}}$ is changed into the fixed boundary $\dot{\Gamma}_{{s}}$. Correspondingly, the domain ${\Omega}_+$ (see Figure \ref{fig:3}) is transformed into the fixed domain $\dot{\Omega}_+$.

Therefore, we define the norm of $U$ in the domain ${\Omega}_+$ as below:
\begin{equation}
\|{U} \|_{({\Omega}_+;{\Gamma}_s)}: = \| {U}\circ \mathcal{T}^{-1}\|_{(\dot{\Omega}_+;\dot{\Gamma}_s)}.
\end{equation}

In this paper, we will establish the existence of the transonic shock in the flat nozzle by proving the following theorems.

\begin{thm}\label{thm:initial_approx_existence}
	 Let $\beta > 2$. Suppose that \eqref{PCI}-\eqref{PE} hold and $\bar{\xi}_*$ satisfies \eqref{inipo}. There exist a sufficiently small constant $ g_0>0 $ such that for any $0<g<g_0$ and $0< \sigma \leq g^3$,
there exists a unique solution $ (\dot{U}_{-}(\xi,\eta),\ \dot{U}_{+}(\xi,\eta),\ \dot{\xi}_*;\ \dot{\psi}'(\eta)) $ to the free boundary problem {\bf {$\llbracket\textit{IFBPL}\rrbracket$}}, with the unknown constant $\dot{\xi}_*\in (0, \xi_1)$. Moreover, the following estimates hold:
	\begin{align}
&|\dot{\xi}_*  - \bar{\xi}_*|  \leq \dot{C}_* g,\label{relationp}\\
&\|\dot{U}_-\|_{\mcc^{2,\alpha}(\dot{\Omega}_-)} \leq \dot{C}_{-}\sigma,\label{eq061}\\
&\| \dot{U}_+ \|_{(\dot{\Omega}_+;\dot{\Gamma}_{{s}})} +\| \dot{\psi}'\|_{W_\beta^{1-\frac{1}{\beta}}(\dot{\Gamma}_s)}
	 \leq \dot{C}_+ \sigma,\label{eq089}
	\end{align}
 where the constant $\dot{C}_*$ depends on $p_0, q_0$, $P_I$, $P_E$ and $\|{\theta}^\aleph\|_{\mathcal{C}^0((0,\bar{\xi}_*)\times (0,1))}$, the constant
   $\dot{C}_-$ depends on $p_0, q_0$, $P_I$ and $L$, the constant
    $\dot{C}_+$ depends on $p_0, q_0$, $P_I$, $P_E$, $L$, $\beta$ and $\dot{\xi}_*$.
\end{thm}

With above preparations, the main result can be stated as follows:
\begin{thm}\label{ppo}
Under the assumptions of Theorem \ref{thm:initial_approx_existence}, there exists a unique solution $(U_-, U_+, \xi_*; \psi')$ to the free boundary problem {\bf {$\llbracket\textit{FBPL}\rrbracket$}} and the following estimates hold:
\begin{align}
  &|\psi(1) - \dot{\xi}_*|\leq C_s \frac{\sigma}{g},\quad \| {\psi}'\|_{W_\beta^{1-\frac{1}{\beta}}({\Gamma}_s)}\leq C_{{s}} \sigma,\\
	&\| U_- - \bar{U}_- \|_{\mcc^{2,\alpha}(\Omega_-)}  \leq C_-\sigma,\\
	&\| {U}_+ - \bar{U}_+ \|_{({\Omega}_+;{\Gamma}_s)} \leq C_+ \sigma,\\
	&\| U_- - (\bar{U}_- + \dot{U}_-)\|_{{\mcc}^{1,\alpha}({\Omega})} \leq \frac12\sigma g^{\frac{3}{2}},\\
	&\| U_+\circ \mathcal{T}^{-1}- (\bar{U}_+ + \dot{U}_+)\|_{(\dot{\Omega}_+;\dot{\Gamma}_s)}  \leq \frac12\sigma g^{\frac{3}{2}},\\
&\|{\psi}' - \dot{\psi}'\|_{W_\beta^{1-\frac{1}{\beta}}(\dot{\Gamma}_s)}\leq \frac12\sigma g^{\frac{3}{2}},
\end{align}
where the constants $C_s$ and $C_\pm$ depend on $p_0, q_0$, $P_I$, $P_E$, $L$, $\beta$ and $\|{\theta}^\aleph\|_{\mathcal{C}^0((0,\dot{\xi}_*)\times (0,1))}$.

\end{thm}

 \begin{rem}
   It should be noted that, under the conditions \eqref{PCI}-\eqref{PE}, when the length of the nozzle $L \leq \xi_1$, $\theta^{\aleph}$ has the sign-preserving property such that the  function $\mathcal{R}_{g\sigma}^\natural$ is strictly decreasing. Then \eqref{inipo} has a unique solution. That is, the approximate position of the shock front in the nozzle can be determined uniquely. Theorem \ref{ppo} shows that there exists a transonic shock solution with the position of the shock front close to it. However, under the given boundary conditions, it is still an open problem whether the obtained shock solution is unique or not.

   Moreover, when $L > \xi_1$, the solution $\theta^{\aleph}$ no longer enjoys the sign-preserving property such that the function $\mathcal{R}_{g\sigma}^\natural$ will not be monotone. Therefore, there may exists more than one solutions $\bar{\xi}_*$ to the equation \eqref{inipo}. Thus, there may exist more than one initial approximating shock solutions, and each approximation will lead to a transonic shock solution to the free boundary problem {\bf {$\llbracket\textit{FBPL}\rrbracket$}}.
   Hence, there may exist more than one shock solutions to the problem {\bf {$\llbracket\textit{FBPL}\rrbracket$}}.
 \end{rem}

 %\begin{rem}
%   Pay attention to the state equation, for the isothermal gases, there is a background solution for the Euler system $\eqref{eq1}$, but for the polytropic gas, there not exist a background solution only depend on the variable $y$, it is some difficult to find the background solution and furthermore we cannot obtain a good estimate for the hyperbolic system in the supersonic part, and eventually the estimate in the subsonic part also failed.
%
% \end{rem}
%

\section{A preliminary: boundary value problems of a typical elliptic-hyperbolic composite system}
To establish the existence of the shock solution and prove Theorem \ref{ppo}, one of the key steps is to solve the boundary value problem of the linearized Euler system for the subsonic flow behind the shock front, which is elliptic-hyperbolic composite. As a preliminary, in this section, we are going to establish theorems on the existence of the solutions to such problems, which will be employed later in proving Theorem \ref{thm:initial_approx_existence} and Theorem \ref{ppo}.

We remark that the notations used in this section are independent and have no relations to the ones in other parts of the paper.

Let $\xi_0$ and $L$ be two positive constants, and
\begin{equation}
  {\Omega} = \{ (\xi, \eta)\in \mathbb{R}^2 : \xi_0 < \xi < L, \,\,\,0 < \eta< 1\},
\end{equation}
be a rectangle with the boundaries
\begin{align*}
&\Gamma_s = \{ (\xi,\eta)\in \mathbb{R}^2 : \xi =\xi_0,\,\,\, 0 <\eta <1 \},\\
&\Gamma_2 = \{ (\xi,\eta)\in \mathbb{R}^2 :  \xi_0 <{\xi} < L, \,\,\,\eta =0 \},\\
&\Gamma_3 = \{ (\xi,\eta)\in \mathbb{R}^2 : \xi = L,\,\,\, 0< \eta <1 \},\\
&\Gamma_4 = \{ (\xi,\eta)\in \mathbb{R}^2 :  \xi_0 < {\xi} < L , \,\,\,\eta =1 \}.
\end{align*}
Consider the following boundary value problem for the unknowns $(\mathcal{U},\mathcal{V}, \mathcal{W} )^\top$:
\begin{align}
&\partial_\eta \mathcal{U} + \mathcal{A}_1\partial_\xi  \mathcal{V} - \mathcal{A}_2 \mathcal{W} = \mathcal{F}_1,&\quad &\text{in}\quad \Omega\label{x}\\
&\partial_\eta \mathcal{V} - \mathcal{A}_3(\eta)\partial_\xi \mathcal{U} - \mathcal{A}_4(\eta)\mathcal{V}  = \mathcal{F}_2, &\quad &\text{in}\quad \Omega \label{xx}\\
&\partial_\xi \mathcal{W} + \mathcal{A}_5(\eta)\partial_\xi \mathcal{U} + \mathcal{A}_6\mathcal{V}  = \partial_\xi \mathcal{F}_3 + \mathcal{F}_4,&\quad &\text{in}\quad \Omega\label{x26} \\
&\mathcal{U} = \mathcal{U}_s,\mathcal{W} = \mathcal{W}_s, &\quad &\text{on} \quad {\Gamma}_s\\
&\mathcal{V} = 0, &\quad &\text{on}\quad\Gamma_{2} \cup\Gamma_{4}\\
&\mathcal{U}= \mathcal{U}_3, &\quad &\text{on}\quad\Gamma_3\label{x1}
\end{align}
where $\mathcal{A}_1$, $\mathcal{A}_2$ and $\mathcal{A}_6$ are constants, $\mathcal{A}_i> 0 (i=1, \cdots ,6)$, and $\inf \limits_ {\eta\in (0,1)} A_3(\eta) >0$.

Since $ \mathcal{A}_1 $ and $ \mathcal{A}_3 $ are positive, the equations \eqref{x} and \eqref{xx} form an elliptic system of first order for $(\mathcal{U},\mathcal{V})^\top$ in its principle part.
Moreover, it is obvious that the equation \eqref{x26} is a transport equation.
Hence, the equations \eqref{x}-\eqref{x26} form an elliptic-hyperbolic composite system with the unknowns being coupled in the $ 0 $-order terms.
The appearance of the coupled terms brings difficulties in deducing the solvability condition on the non-homogeneous terms and the boundary data, which is observed to be required for the elliptic sub-problem of the boundary value problem of the linearized Euler system without the gravity for subsonic flows.
To deal with the difficulties, an auxiliary problem for a modified elliptic system of first order for the equations \eqref{x}-\eqref{xx} will be introduced.
It turns out that the solvability condition can be easily reduced with the help of the introduced problem.
Then the existence of the solution and its a priori estimates can be established by employing classical theory for the elliptic equations and the hyperbolic equations.

In this section, we are going to prove the following theorem for the boundary value problem  $\eqref{x}$-$\eqref{x1}$.

\begin{thm}\label{cv0}
 Let $\beta>2$. Suppose $\mathcal{F}_i\in L^\beta({\Omega})$,$(i =1,2)$, $\mathcal{F}_j\in \mcc^0({\Omega})$,$(j = 3,4)$ and $\mathcal{U}_{s},\mathcal{W}_{s} \in W_\beta^{1 - \frac{1}{\beta}}(\Gamma_s)$, $\mathcal{U}_{3}\in W_\beta^{1 - \frac{1}{\beta}}(\Gamma_3)$, then for the boundary value problem $\eqref{x}$-$\eqref{x1}$, there exists a unique solution $(\mathcal{U},\mathcal{V},\mathcal{W})^\top$ if and only if
\begin{align}\label{x30}
  \int\int_{\Omega}  \mathcal{A}_+(\eta) \mathcal{F}_2(\xi,\eta)\dif\xi \dif \eta = \int_0^1 \mathcal{A}_3(\eta)\mathcal{A}_+(\eta) (\mathcal{U}_s - \mathcal{U}_3)(\eta)\dif\eta,
\end{align}
where
$$\mathcal{A}_+(\eta) \defs \exp \Big(-\int_0^\eta \mathcal{A}_4(\tau)\dif \tau\Big).$$
Moreover, $(\mathcal{U},\mathcal{V}, \mathcal{W} )^\top$ satisfies the following estimate:
  \begin{equation}\label{x20}
  \begin{split}
   & \| \mathcal{U} \|_{W_\beta^1(\Omega)} +  \| \mathcal{V}\|_{W_\beta^1(\Omega)}+  \| \mathcal{W}\|_{\mcc^0(\Omega)} +  \| \mathcal{W}\|_{W_\beta^{1-\frac{1}{\beta}}(\Gamma_s)}\\
     \leq &C\Big( \sum_{i=1}^2\| \mathcal{F}_i \|_{L^\beta(\Omega)} + \| \mathcal{F}_3 - \mathcal{F}_3(\xi_0,\eta) \|_{\mcc^0(\Omega)}+ \| \mathcal{F}_4 \|_{\mcc^0(\Omega)}\Big) \\
     &+C\Big( \|\mathcal{U}_s\|_{W_\beta^{1-\frac{1}{\beta}}(\Gamma_s)}+ \|\mathcal{W}_s \|_{W_\beta^{1-\frac{1}{\beta}}(\Gamma_s)}+ \|\mathcal{U}_3\|_{W_\beta^{1-\frac{1}{\beta}}(\Gamma_3)} \Big),
     \end{split}
     \end{equation}
    where the constant $C$ depends on $\xi_0$, $L$, $\beta$ and the coefficients $\mathcal{A}_i$, $(i=1, \cdots ,6)$.
\end{thm}

\begin{proof}
We divide our proof into three steps.

\textbf{Step 1}: In this step, the following auxiliary problem will be solved
\begin{align}
  &\partial_\eta \mathcal{U}^{(1)} + \mathcal{A}_1\partial_\xi  \mathcal{V}^{(1)} =0,&\quad &\text{in}\quad \Omega\label{x2}\\
&\partial_\eta \Big(\mathcal{A}_+(\eta)\mathcal{V}^{(1)}\Big) - \partial_\xi \Big(\mathcal{A}_+(\eta) \mathcal{A}_3(\eta)\mathcal{U}^{(1)}\Big)  = \mathcal{A}_+ (\eta) \mathcal{F}_2, &\quad &\text{in}\quad \Omega\\
&\mathcal{U}^{(1)} = \mathcal{U}_s, &\quad &\text{on} \quad {\Gamma}_s\\
&\mathcal{U}^{(1)} = \mathcal{U}_3, &\quad &\text{on}\quad\Gamma_{3} \\
&\mathcal{V}^{(1)} = 0, &\quad &\text{on}\quad\Gamma_{2} \cup\Gamma_{4}\label{x3}
\end{align}
where
$$\mathcal{A}_+(\eta) = \exp \big(-\int_0^\eta \mathcal{A}_4(\tau)\dif \tau\big).$$
 The equation \eqref{x2} implies that there exists a potential function $\Psi$ such that
  \begin{equation}\label{x10}
    \nabla \Psi = (\partial_\xi \Psi, \partial_\eta \Psi) = (-\mathcal{U}^{(1)}, \mathcal{A}_1 \mathcal{V}^{(1)}),
  \end{equation}
then the problem $\eqref{x2}$-$\eqref{x3}$ can be formulated as
 \begin{align}
   &\partial_\eta \Big(\frac{\mathcal{A}_+(\eta)}{\mathcal{A}_1}\partial_\eta \Psi\Big) +\partial_\xi \Big(\mathcal{A}_+(\eta) \mathcal{A}_3(\eta)\partial_\xi \Psi\Big)  = \mathcal{A}_+ (\eta) \mathcal{F}_2, &\quad &\text{in}\quad \Omega\label{x6}\\
   &-\partial_\xi \Psi = \mathcal{U}_s, &\quad &\text{on} \quad {\Gamma}_s\\
     &-\partial_\xi \Psi = \mathcal{U}_3, &\quad &\text{on} \quad {\Gamma}_3\\
&\quad \, \partial_\eta \Psi = 0. &\quad &\text{on}\quad\Gamma_{2} \cup\Gamma_{4}\label{x7}
 \end{align}
By applying Corollary 4.4.3.8 in the book \cite{PG}, there exists a unique solution ${\Psi}\in W_\beta^2(\Omega)$, up to an additive constant, to the problem \eqref{x6}-\eqref{x7}, if and only if
\begin{align}\label{sol}
  \int\int_{\Omega}  \mathcal{A}_+(\eta) \mathcal{F}_2(\xi,\eta)\dif\xi \dif \eta = \int_0^1 \mathcal{A}_3(\eta)\mathcal{A}_+(\eta) (\mathcal{U}_s - \mathcal{U}_3)(\eta)\dif\eta,
\end{align}
which is exactly the solvability condition \eqref{x30}. Moreover, without loss of generality, we may assume $\int\int_{\Omega} \Psi(\xi,\eta)\dif\xi \dif \eta =0$. Then, by applying Poincar\'{e} inequality, there exists a
constant $C_{(\Omega)}$ such that
 \begin{equation}\label{t1}
   \|\Psi\|_{L^2(\Omega)} \leq C_{(\Omega)} \|\nabla \Psi \|_{L^2(\Omega)}.
 \end{equation}
Then, multiplying $\Psi$ on both sides of the
equation \eqref{x6}, integrating over $\Omega$ and then employing the formula of integration by
parts, one has
\begin{equation}\label{x8}
\begin{split}
  \|\nabla \Psi\|_{L^2(\Omega)}^2 \leq& C \| \mathcal{F}_2 \|_{L^2(\Omega)}  \| \Psi\|_{L^2(\Omega)} \\
  &+ C\Big( \|\mathcal{U}_s\|_{L^\infty(\Gamma_s)}+ \|\mathcal{U}_3\|_{L^\infty(\Gamma_3)}\Big) \| \Psi\|_{L^2(\partial\Omega)}.
\end{split}
\end{equation}
Therefore, by employing Trace theorem and \eqref{t1}, $\mathcal{F}_2\in L^\beta({\Omega})$, $\mathcal{U}_{s}\in W_\beta^{1 - \frac{1}{\beta}}(\Gamma_s)$, $\mathcal{U}_{3}\in W_\beta^{1 - \frac{1}{\beta}}(\Gamma_3)$, it follows that
 \begin{align}\label{x9}
  \| \Psi\|_{H^1(\Omega)}\leq C \mathcal{F}_{\Psi},
\end{align}
where
$$\mathcal{F}_{\Psi}\defs\| \mathcal{F}_2 \|_{L^\beta(\Omega)} + \|\mathcal{U}_s\|_{W_\beta^{1-\frac{1}{\beta}}(\Gamma_s)}+ \|\mathcal{U}_3\|_{W_\beta^{1-\frac{1}{\beta}}(\Gamma_3)} .$$
Employing the embedding theorem, and Theorem 4.3.2.4 as well as Remark 4.3.2.5 in the book \cite{PG}, one has
\begin{align}
   \| \Psi\|_{W_\beta^2(\Omega)}\leq C \mathcal{F}_{\Psi}.
 \end{align}
By the definition of $\Psi$ in \eqref{x10}, it is easy to see that
  \begin{align}\label{x19}
   \|\mathcal{U}^{(1)}\|_{W_\beta^1(\Omega)} + \| \mathcal{V}^{(1)}\|_{W_\beta^1(\Omega)} \leq C \mathcal{F}_{\Psi}.
 \end{align}

\textbf{Step 2}: Let $ (\mathcal{U}^{(2)}, \mathcal{V}^{(2)})^\top =(\mathcal{U}, \mathcal{V})^\top -  (\mathcal{U}^{(1)}, \mathcal{V}^{(1)})^\top $, by employing \eqref{x}-\eqref{x1},
then $(\mathcal{U}^{(2)}, \mathcal{V}^{(2)})^\top$ satisfies the following problem in $\Omega$:
 \begin{align}
   &\partial_\eta \mathcal{U}^{(2)} + \mathcal{A}_1\partial_\xi  \mathcal{V}^{(2)} - \mathcal{A}_2 \mathcal{W} = \mathcal{F}_1,\label{x4}\\
& \partial_\eta \Big(\mathcal{A}_+(\eta)\mathcal{V}^{(2)}\Big) - \partial_\xi \Big(\mathcal{A}_+(\eta) \mathcal{A}_3(\eta)\mathcal{U}^{(2)}\Big) = 0, \label{x14}\\
&\partial_\xi \mathcal{W} + \mathcal{A}_5(\eta)\partial_\xi \mathcal{U}^{(2)} + \mathcal{A}_6\mathcal{V}^{(2)}
  =\partial_\xi \mathcal{F}_3 + \mathcal{F}_4 -\mathcal{A}_5(\eta)\partial_\xi \mathcal{U}^{(1)} - \mathcal{A}_6\mathcal{V}^{(1)}, \label{x11}
 \end{align}
with the boundary conditions
\begin{align}
& \mathcal{W} = \mathcal{W}_s, &\quad &\text{on} \quad {\Gamma}_s\label{mathcalW}\\
  &\mathcal{U}^{(2)} =0, &\quad &\text{on} \quad {\Gamma}_s\cup\Gamma_{3}  \label{x17}\\
&\mathcal{V}^{(2)} = 0. &\quad &\text{on}\quad\Gamma_{2} \cup\Gamma_{4}\label{x5}
\end{align}
Applying \eqref{x11} and \eqref{mathcalW}, one can deduce that
 \begin{equation}\label{x12}
 \begin{split}
   \mathcal{W} =&\int_{\xi_0}^{\xi}\Big(\partial_\tau \mathcal{F}_3 + \mathcal{F}_4 -\mathcal{A}_5(\eta)\partial_\tau \mathcal{U}^{(1)} - \mathcal{A}_6\mathcal{V}^{(1)}\Big)(\tau,\eta)\dif\tau\\
   &-  \int_{\xi_0}^{\xi}\Big(\mathcal{A}_5(\eta)\partial_\tau \mathcal{U}^{(2)}+ \mathcal{A}_6\mathcal{V}^{(2)}\Big)(\tau,\eta)\dif\tau + \mathcal{W}_s.
 \end{split}
 \end{equation}
Substituting $\eqref{x12}$ into \eqref{x4}, one has
\begin{equation}\label{x13}
\begin{split}
    &\partial_\eta \mathcal{U}^{(2)} + \mathcal{A}_1 \partial_\xi \mathcal{V}^{(2)} + \mathcal{A}_2 \mathcal{A}_6\int_{\xi_0}^{\xi} \mathcal{V}^{(2)}(\tau,\eta) \dif \tau + \mathcal{A}_2 \mathcal{A}_5(\eta)\mathcal{U}^{(2)}\\
   =& \mathcal{F}_1 + \mathcal{A}_2\Big( \mathcal{F}_3- \mathcal{F}_3(\xi_0,\eta)\Big)
+\mathcal{A}_2\int_{\xi_0}^{\xi}\Big(\mathcal{F}_4 -\mathcal{A}_5(\eta)\partial_\tau \mathcal{U}^{(1)} - \mathcal{A}_6\mathcal{V}^{(1)}\Big)(\tau,\eta)\dif \tau\\
 &+ \mathcal{A}_2 \mathcal{W}_s\\
   \defs & \tilde{\mathcal{F}}_1.
  \end{split}
  \end{equation}
Let
$$\mathcal{B}_+(\eta) : = \exp\Big(\int_0^\eta \mathcal{A}_2 \mathcal{A}_5 (\tau)\dif \tau\Big ),$$
 then the equation \eqref{x13} can be rewritten as
\begin{equation}\label{x15}
\begin{split}
  &\partial_\eta \Big(\mathcal{B}_+(\eta) \mathcal{U}^{(2)}\Big) + \partial_\xi \Big(\mathcal{A}_1 \mathcal{B}_+(\eta) \mathcal{V}^{(2)}\Big)  + \mathcal{A}_2 \mathcal{A}_6 \mathcal{B}_+ (\eta) \int_{\xi_0}^{\xi}  \mathcal{V}^{(2)}(\tau,\eta) \dif\tau\\
   =& \mathcal{B}_+ (\eta)\tilde{\mathcal{F}}_1.
\end{split}
\end{equation}
 By applying \eqref{x14}, there exists a potential function $\Phi$ such that
\begin{equation}
  \nabla {\Phi} = (\partial_\xi {\Phi}, \partial_\eta {\Phi}) = (\mathcal{A}_+ (\eta)\mathcal{V}^{(2)}, \mathcal{A}_+(\eta) \mathcal{A}_3(\eta)\mathcal{U}^{(2)}).
\end{equation}
Then \eqref{x15} becomes
\begin{equation}\label{x16}
\begin{split}
  &\partial_\eta \Big(\frac{\mathcal{B}_+(\eta)}{\mathcal{A}_+(\eta) \mathcal{A}_3(\eta)}\partial_\eta \Phi\Big) + \partial_\xi \Big(\mathcal{A}_1\frac{ \mathcal{B}_+(\eta)}{\mathcal{A}_+(\eta)} \partial_\xi \Phi\Big)  + \mathcal{A}_2 \mathcal{A}_6 \frac{ \mathcal{B}_+(\eta)}{\mathcal{A}_+(\eta)}\Big(\Phi - \Phi(\xi_0,\eta)\Big)\\
   = &\mathcal{B}_+(\eta) \tilde{\mathcal{F}}_1.
\end{split}
\end{equation}
In addition, the boundary conditions \eqref{x17}-\eqref{x5} are changed into
\begin{align}
    \partial_\eta \Phi =& 0, &\quad  &\text{on} \quad \Gamma_s\cup\Gamma_3 \\
  \partial_\xi \Phi =& 0. &\quad  &\text{on} \quad \Gamma_2 \cup  \Gamma_4
\end{align}
Without loss of generality, one may assume that $\Phi({\xi}_0,0)=0$, then one has
\begin{align}
  \Phi =0. \quad \text{on}\quad \partial\Omega
\end{align}
By employing standard elliptic theory (cf.\cite{Evans, PG}), it follows that
\begin{align}
   \| \Phi\|_{W_\beta^2(\Omega)}\leq C \| \tilde{\mathcal{F}}_1 \|_{L^\beta(\Omega)}.
 \end{align}
 Then by applying the definition of $\Phi$, it holds that
 \begin{equation}\label{x18}
 \begin{split}
   \|\mathcal{U}^{(2)}\|_{W_\beta^1(\Omega)} + \| \mathcal{V}^{(2)}\|_{W_\beta^1(\Omega)}
   \leq C \| \tilde{\mathcal{F}}_1 \|_{L^\beta(\Omega)}.
 \end{split}
 \end{equation}

\textbf{Step 3}: Finally, it remains to solve $\mathcal{W}$. Recalling the equation \eqref{x26}, one has
\begin{equation}\label{x27}
\begin{split}
  \mathcal{W}(\xi,\eta) =& \mathcal{F}_3(\xi,\eta) - \mathcal{F}_3(\xi_0,\eta) + \mathcal{W}_s(\xi_0,\eta) + \mathcal{A}_5(\eta) (\mathcal{U}(\xi,\eta) - \mathcal{U}_s(\xi_0,\eta))\\
   &+ \int_{\xi_0}^{\xi} (\mathcal{F}_4 - \mathcal{A}_6 \mathcal{V})(\tau,\eta)\dif\tau,
\end{split}
\end{equation}
from which one can deduce that
 \begin{align}\label{3.46}
    \| \mathcal{W}\|_{\mcc^0(\Omega)} +
     \| \mathcal{W}\|_{W_\beta^{1-\frac{1}{\beta}}(\Gamma_s)} \leq& C \Big( \| \mathcal{F}_3 - \mathcal{F}_3(\xi_0,\eta) \|_{\mcc^0(\Omega)}+ \| \mathcal{F}_4 \|_{\mcc^0(\Omega)} + \|\mathcal{U}_s\|_{W_\beta^{1-\frac{1}{\beta}}(\Gamma_s)}\notag\\
     &\qquad + \|\mathcal{W}_s \|_{W_\beta^{1-\frac{1}{\beta}}(\Gamma_s)} +  \| \mathcal{U} \|_{\mcc^0(\Omega)} + \| \mathcal{V} \|_{\mcc^0(\Omega)}  \Big).
  \end{align}

Therefore, by applying \eqref{x19},\eqref{x18}, \eqref{3.46} and the definition of $\tilde{\mathcal{F}}_1$, the estimate \eqref{x20} can be obtained immediately.
\end{proof}

\section{The initial approximation}

In this section, we are going to establish the existence of the solution to the initial linearized free boundary problem {\bf {$\llbracket\textit{IFBPL}\rrbracket$}} with the help of Theorem \ref{cv0}. That is, the Theorem \ref{thm:initial_approx_existence} will be proved.

\subsection{The solution $ \dot{U}_{-} $ in $ \Omega $}
For the linearized equations $\eqref{eq63}$-$\eqref{eq65}$ in the domain ${\Omega}$ with the initial-boundary conditions \eqref{eq82}-\eqref{eq82B},
the existence of the unique solution $\dot{U}_-\in {\mcc}^{2,\alpha}({\Omega})$ can be easily obtained by applying the theory in the book \cite{LY}. In order to clearly analyze the solvability condition of the subsonic solution in the following arguments,
the linear equations \eqref{eq63}-\eqref{eq64} with the initial-boundary conditions \eqref{eq82}-\eqref{eq82B} will be divided into the following two parts.

Let
$(\dot{p}_-^{(1)}, \dot{\theta}_-^{(1)})^\top$ satisfies the following problem $\textbf{(a)}$:
\begin{align}
  &\partial_\eta \dot{p}_-^{(1)} + q_0 \partial_\xi \dot{\theta}_-^{(1)}= 0,\label{eq896}\\
 &\partial_\eta \dot{\theta}_-^{(1)} - \frac{1 - \bar{M}_-^2}{p_0^2 q_0^3} \partial_\xi \dot{p}_-^{(1)}= 0,\label{eq897}
\end{align}
with the initial-boundary conditions
\begin{align}
 \dot{\theta}_-^{(1)} =& 0, \quad \dot{p}_-^{(1)} = \sigma P_I(\eta),&\quad  &\text{on}&\quad  &\Gamma_1\\
 \quad\dot{\theta}_-^{(1)} =& 0. & \quad  &\text{on}&\quad &\Gamma_2\cup \Gamma_4\label{eq899}
\end{align}
Recalling the problem \eqref{eq896*}-\eqref{eq899*}, one can find that $( \dot{p}_-^{(1)}, \dot{\theta}_-^{(1)}) = \sigma( {p}^\aleph, {\theta}^\aleph )$.
Let $(\dot{p}_-^{(2)}, \dot{\theta}_-^{(2)})^\top$ be the solution to the following problem $\textbf{(b)}$:
\begin{align}
  &\partial_\eta \dot{p}_-^{(2)} + q_0\partial_\xi \dot{\theta}_-^{(2)}= -H_1(g,0)\partial_\eta \dot{p}_- + \frac{g}{q_0^2} \dot{q}_- + H_{2_-}(g,\sigma),\label{eq920}\\
 &\partial_\eta \dot{\theta}_-^{(2)} - \frac{1 - \bar{M}_-^2}{p_0^2{q_0^3}} \partial_\xi \dot{p}_-^{(2)}
 =\displaystyle\frac{1-\bar{M}_-^2}{q_0^3}\left(\displaystyle\frac{1}{\bar{\rho}_-^2} - \displaystyle\frac{1}{p_0^2}\right)\partial_\xi\dot{p}_- - H_1(g,0)\partial_\eta \dot{\theta}_- +  \frac{g\dot{\theta}_-}{\bar{\rho}_- q_0^3},\label{eq921}
\end{align}
with the initial-boundary conditions
\begin{align}
 \dot{\theta}_-^{(2)} =& 0,\quad \dot{p}_-^{(2)} = {P}_I^*(\eta) -\sigma P_I(\eta),&\quad &\text{on} &\quad  &\Gamma_1\\
 \quad\dot{\theta}_-^{(2)} =& 0. &\quad  &\text{on} &\quad  &\Gamma_2\cup \Gamma_4\label{p3}
\end{align}
Then it is obvious that
\begin{align*}
(\dot{p}_-,\dot{\theta}_-) = (\dot{p}_-^{(1)}, \dot{\theta}_-^{(1)}) + (\dot{p}_-^{(2)}, \dot{\theta}_-^{(2)}).
 \end{align*}
 This decomposition will only be employed in the analysis for the solvability condition later. Moreover, it turns out that $(\dot{p}_-^{(1)}, \dot{\theta}_-^{(1)})$ contributes the principle part in analyzing the solvability condition, and $(\dot{p}_-^{(2)}, \dot{\theta}_-^{(2)})$ contributes in the higher order term.

 \begin{rem}
   For the problems $\textbf{(a)}$ and $\textbf{(b)}$, the initial data only ensures that the zero order compatibility conditions hold,
  then applying the method of characteristic, one can only obtain global $\mcc^0$ estimate and the piecewise $\mcc^1$ regularity of the solution $(\dot{p}_-^{(i)}, \dot{\theta}_-^{(i)})$, $(i=1,2)$. Fortunately, the global $\mcc^0$ estimate is sufficient
to analyze the solvability condition.
 \end{rem}

Consequently, the following lemma holds:

 \begin{lem}\label{w0}
  Suppose that \eqref{PCI}-\eqref{P1.23} hold, then there exists a unique solution $\dot{U}_-$ to the linearized equations $\eqref{eq63}$-$\eqref{eq65}$ in the domain ${\Omega}$ with the initial-boundary conditions \eqref{eq82}-\eqref{eq82B}, and satisfies the following estimate:
\begin{align}\label{eq84}
   \|\dot{U}_- \|_{\mcc^{2,\alpha}(\Omega)} \leq& C \| {P}_I^*\|_{\mcc^{2,\alpha}(\Gamma_1)} \leq \dot{C}_- \sigma,
\end{align}
where the constant $\dot{C}_-$ depends on $p_0$, $q_0$, ${P}_I$ and $L$.
Moreover, for the problems $\textbf{(a)}$ and $\textbf{(b)}$, there exist solutions $(\dot{p}_-^{(i)}, \dot{\theta}_-^{(i)})^\top$, $(i=1,2)$
satisfying
\begin{align}
 &\|\dot{p}_-^{(1)}\|_{\mcc^0(\Omega)} + \|\dot{\theta}_-^{(1)}\|_{\mcc^0(\Omega)} \leq \dot{C}_-^{\natural(1)}
   \sigma,\label{eq9723(1)}\\
 &\|\dot{p}_-^{(2)}\|_{\mcc^0(\Omega)} + \|\dot{\theta}_-^{(2)}\|_{\mcc^0(\Omega)} \leq \dot{C}_-^{\natural(2)}
   \big(g\sigma + \sigma^2\big),\label{eq9723}
\end{align}
where the constants $\dot{C}_-^{\natural(i)}$, $(i=1,2)$ depend on $p_0$, $q_0$, ${P}_I$ and $L$.
In particular,
\begin{align}
   &C_0 \sigma \leq \dot{\theta}_-^{(1)}(\xi, \eta)\leq  C_1\sigma ,\quad \text{for any }\quad \xi\in (0, \xi_1),\quad \eta\in (0,1),\label{sign}
\end{align}
where $\xi_1 = \displaystyle\frac{\sqrt{\bar{M}_-^2 -1}}{p_0 q_0}$ and the constants $C_0, C_1 >0$ depend on $p_0$, $q_0$ and $P_I$.

 \end{lem}

\begin{proof}
By applying \eqref{q2}, the compatibility conditions hold at the corners up to second order. Then the existence of the unique solution $\dot{U}_-\in {\mcc}^{2,\alpha}({\Omega})$ can be obtained by employing the theory in the book \cite{LY}.
Moreover, the existence of the unique solutions $(\dot{p}_-^{(i)}, \dot{\theta}_-^{(i)})^\top\in {\mcc}^{0}({\Omega})$ can be easily obtained. In addition, it is easy to see \eqref{eq9723(1)} holds. Besides, \eqref{eq9723} can be derived immediately by applying the estimate $\eqref{eq84}$.
Finally, employing \eqref{PCI}-\eqref{P1.23} and \eqref{thealeph}, one can deduce that \eqref{sign} holds.

\end{proof}

%\begin{rem}\label{mb1}
%  \begin{equation}\label{eq3000}
%    \xi_1 =  \displaystyle\frac{\sqrt{\bar{M}_-^2 -1}}{p_0 q_0},\quad \xi_{*2} = \xi + \displaystyle\frac{\sqrt{\bar{M}_-^2 -1}}{p_0 q_0}(\eta -1) ,\quad \xi_{*3} = \xi - \displaystyle\frac{\sqrt{\bar{M}_-^2 -1}}{p_0 q_0}\eta.
%    \end{equation}
%\end{rem}

%\begin{rem}
%  When $\xi>\xi_1$, $\dot{\theta}_-^{(1)}$ no longer enjoys the sign-preserving property.
%\end{rem}

\subsection{Reformulation of the linearized R-H conditions \eqref{eq69}-\eqref{eq74}.}
%The equations \eqref{eq69}-\eqref{eq74} form a closed linear algebraic equations for $ (\dot{p}_+, \dot{q}_+) $ such that they can be expressed by $ \dot{U}_{-} $ on the free boundary  $ \dot{\Gamma}_s$.

The equation \eqref{eq69} can be rewritten as the following form:
\begin{equation}\label{x55}
  A_* (\dot{p}_+, \dot{q}_+)^\top = (\dot{J}_1 , \dot{J}_2)^\top,
\end{equation}
where $\dot{J}_i : = -{\mathbf{\alpha}}_{j-} \cdot {\dot{U}}_-$, $(i =1,2)$,
\[ A_* = \frac{[\bar{p}]}{\bar{\rho}_+\bar{q}_+}\begin{pmatrix}
   -\displaystyle\frac{1}{\bar{\rho}_+ \bar{c}_+^2}& -\displaystyle\frac{1}{\bar{q}_+}\\
  1 - \displaystyle\frac{\bar{p}_+}{\bar{\rho}_+ \bar{c}_+^2}& {\bar{\rho}_+\bar{q}_+} - \displaystyle\frac{\bar{p}_+}{ \bar{q}_+}\\
 \end{pmatrix}. \]
Then the following lemma holds on the shock front $\dot{\Gamma}_s$.
\begin{lem}\label{e1}
  On the shock front $\dot{\Gamma}_s$, it holds that
\begin{align}
\det(A_*) =& \displaystyle\frac{[\bar{p}]^2}{\bar{\rho}_+^2\bar{q}_+^3}(1-\bar{M}_+^2)\neq 0,\quad \text{as}\quad \bar{M}_+\neq 1,\label{x34}\\
 \dot{p}_+ \defs& \dot{\pounds}_1,\label{eq109}\\
  \dot{q}_+\defs& \dot{\pounds}_2,\label{eq110}\\
  \dot{\psi}' \defs& \dot{\pounds}_3,\label{eq111}
\end{align}
with
\begin{align}
\dot{\pounds}_1=&  \displaystyle\frac{\bar{\rho}_+\bar{q}_+^2}{\bar{M}_+^2-1}
\Big(\displaystyle\frac{1-q_0^2}{\bar{\rho}_-} {P}_I^*(\eta)
  + g \cdot\int_0^\xi \dot{\theta}_-(\tau,\eta)\dif \tau\notag\\
  &\qquad \quad \quad \quad - (1+ H_1(g,0))(2- q_0^2)\bar{\rho}_- q_0\cdot \partial_\eta\int_0^\xi \dot{\theta}_-(\tau,\eta)\dif \tau\Big),\label{sf}\\
\dot{\pounds}_2=&\frac{\bar{q}_+}{1-\bar{M}_+^2}\Big(\frac{1}{q_0^2}g\cdot \int_0^\xi \dot{\theta}_-(\tau,\eta)\dif \tau - (1+H_1(g,0))\frac{\bar{\rho}_-}{q_0}\partial_\eta \int_0^\xi \dot{\theta}_-(\tau,\eta)\dif \tau \Big),\\
  \dot{\pounds}_3=&\displaystyle\frac{\bar{q}_+\dot{\theta}_+ - \bar{q}_-\dot{\theta}_-}{(1+ H_1(g,0))[\bar{p}]}.
\end{align}

\end{lem}

\begin{proof}
 By applying the definition of $A_*$, $\eqref{x34}$ can be obtained immediately.

 By employing the fact of $\bar{\rho}_+\bar{q}_+ = \bar{\rho}_-\bar{q}_-$, then \eqref{x55} implies that
\begin{align}
  &\frac{1}{\bar{\rho}_+ \bar{c}_+^2}\dot{p}_+ + \frac{1}{\bar{q}_+}\dot{q}_+ = \frac{1}{\bar{\rho}_- \bar{c}_-^2}\dot{p}_- + \frac{1}{\bar{q}_-}\dot{q}_-,\label{eq103}\\
  & \Big(\dot{p}_+ + \bar{\rho}_+\bar{q}_+\dot{q}_+\Big) - \bar{p}_+ \Big(\frac{1}{\bar{\rho}_+ \bar{c}_+^2}\dot{p}_+ + \frac{1}{\bar{q}_+}\dot{q}_+ \Big)\notag\\
   =& \Big(\dot{p}_- + \bar{\rho}_-\bar{q}_-\dot{q}_- \Big)- \bar{p}_-\Big( \frac{1}{\bar{\rho}_-\bar{c}_-^2}\dot{p}_- + \frac{1}{\bar{q}_-}\dot{q}_- \Big).\label{eq104}
\end{align}
Substituting $\eqref{eq103}$ into $\eqref{eq104}$, one has
\begin{equation}\label{eq00}
  \dot{p}_+ + \bar{\rho}_+\bar{q}_+\dot{q}_+ = \dot{p}_- + \bar{\rho}_-\bar{q}_-\dot{q}_- + [\bar{p}]\Big( \frac{1}{\bar{\rho}_-\bar{c}_-^2}\dot{p}_- +  \frac{1}{\bar{q}_-}\dot{q}_-  \Big).
\end{equation}
By employing the equations $\eqref{eq103}$ and $\eqref{eq00}$, it follows that
\begin{align}
  \displaystyle\frac{\bar{M}_+^2-1}{\bar{\rho}_+\bar{q}_+^2}\dot{p}_+ = & -\displaystyle\frac{1}{\bar{\rho}_+\bar{q}_+^2}\Big(\dot{p}_- + \bar{\rho}_-\bar{q}_-\dot{q}_-\Big) \notag\\
  &\,\, + \Big(1 - \displaystyle\frac{[\bar{p}]}{\bar{\rho}_+\bar{q}_+^2}\Big)
  \displaystyle\frac{1}{\bar{\rho}_-\bar{q}_-^2}
  \Big(\bar{M}_-^2\dot{p}_-
  +\bar{\rho}_-\bar{q}_-\dot{q}_-\Big),\label{x59}\\
  \dot{q}_+ = & -\frac{\bar{q}_+}{\bar{\rho}_+\bar{c}_+^2(1-\bar{M}_+^2)}\Big(\dot{p}_- + \bar{\rho}_-\bar{q}_-\dot{q}_-\Big)\notag\\
   &+ \displaystyle\frac{\bar{q}_+}{\bar{\rho}_-\bar{q}_-^2(1-\bar{M}_+^2)} \Big(1 - \displaystyle\frac{[\bar{p}]}{\bar{\rho}_+\bar{c}_+^2}\Big)
  \Big(\bar{M}_-^2\dot{p}_-
  +\bar{\rho}_-\bar{q}_-\dot{q}_-\Big)\label{eq001}.
\end{align}
Moreover, employing the equations $\eqref{eq64}$ and $\eqref{eq65}$, one can obtain
%\begin{align}\label{eq107}
% & \partial_\eta \dot{\theta}_- - \displaystyle\frac{1}{1+ H_1(g,0)}\displaystyle\frac{1}{\bar{\rho}_-^2\bar{q}_-^3}
%  \Big(-\bar{M}_-^2\partial_\xi \dot{p}_- - \bar{\rho}_-\bar{q}_-\partial_\xi \dot{q}_- - g \bar{\rho}_-\dot{\theta}_-\Big)\notag\\
%  &\quad - \displaystyle\frac{1}{1+ H_1(g,0)}\displaystyle\frac{g}{\bar{\rho}_- \bar{q}_-^3}\dot{\theta}_- = 0,
%\end{align}
%that is
\begin{equation}\label{eq1011}
  \partial_\eta \dot{\theta}_- + \displaystyle\frac{1}{1+ H_1(g,0)} \displaystyle\frac{1}{\bar{\rho}_-^2\bar{q}_-^3}
  \partial_\xi\Big(\bar{M}_-^2\dot{p}_- + \bar{\rho}_-\bar{q}_-\dot{q}_-\Big) = 0.
\end{equation}
Furthermore, one has
\begin{equation}\label{eq1012}
  \bar{M}_-^2\dot{p}_- + \bar{\rho}_-\bar{q}_-\dot{q}_-
  = \bar{M}_-^2 {P}_I^*(\eta) - (1+ H_1(g,0)){\bar{\rho}_-^2\bar{q}_-^3}
  \partial_\eta \int_0^\xi \dot{\theta}_-(\tau, \eta)\dif \tau.
\end{equation}
In addition, $\eqref{eq65}$ implies that
\begin{equation}\label{eq1013}
  \dot{p}_- + \bar{\rho}_-\bar{q}_-\dot{q}_-
  =  {P}_I^*(\eta)- g \bar{\rho}_-\int_0^\xi \dot{\theta}_-(\tau,\eta)\dif \tau.
\end{equation}
Substituting the expressions $\eqref{eq1012}$ and $\eqref{eq1013}$ into the equations $\eqref{x59}$ and \eqref{eq001}, then employing the equation \eqref{2.35}, one can obtain the equations $\eqref{eq109}$ and \eqref{eq110}.

Finally, substituting the coefficients $\alpha_{3\pm}$ in \eqref{x60} into the equation $\eqref{eq74}$, one can obtain the equation $\eqref{eq111}$ immediately.
\end{proof}

\subsection{Determine $ \dot{\xi}_{*} $ and $ \dot{U}_{+} $.}

With the help of Theorem \ref{cv0}, Lemma \ref{w0} and Lemma \ref{e1}, one can now determine the approximating position of the shock front.

By employing Theorem \ref{cv0} and taking
\begin{align*}
  &\mathcal{U}\defs \dot{p}_+, \,\,\mathcal{V} \defs\dot{\theta}_+,\,\, \mathcal{W} \defs \dot{q}_+, \,\,\mathcal{F}_1 \defs \displaystyle\frac{H_{2+}(g,\sigma)}{1+ H_1(g,0)}, \,\,\mathcal{F}_2 = \mathcal{F}_3 = \mathcal{F}_4 = 0,\\
  &\mathcal{U}_s\defs \dot{\pounds}_1,\,\, \mathcal{W}_s\defs \dot{\pounds}_2,\,\, \mathcal{U}_3\defs \dot{P}_e^{\sharp},\\
  &\mathcal{A}_1 \defs \displaystyle\frac{\bar{q}_+}{1+ H_1(g,0)}, \,\,\mathcal{A}_2 \defs \displaystyle\frac{1}{1+ H_1(g,0)}\frac{g}{\bar{q}_+^2},\,\, \mathcal{A}_3 \defs \displaystyle\frac{1}{1+ H_1(g,0)}\frac{1 - \bar{M}_+^2}{\bar{\rho}_+^2 \bar{q}_+^3},\\
  &\mathcal{A}_4 \defs \displaystyle\frac{1}{1+ H_1(g,0)}\frac{g}{\bar{\rho}_+\bar{q}_+^3},\,\,
  \mathcal{A}_5 \defs \frac{1}{\bar{\rho}_+\bar{q}_+},\,\,\mathcal{A}_6 \defs\frac{g}{\bar{q}_+},
\end{align*}
then \eqref{x30} yields that
\begin{align}\label{eq117}
  0 = \int_0^{1} \mathcal{A}_+ \mathcal{A}_3 \Big( \dot{\pounds}_1 (\dot{\xi}_*,\eta) - \dot{P}_e^{\sharp}(\eta)\Big)\dif \eta,
 \end{align}
where
\begin{align}\label{eq87}
 \mathcal{A}_+ =& \exp \Big(-\int_0^\eta \mathcal{A}_4(\tau)\dif \tau\Big)\notag\\
  = & \exp\Big(-\displaystyle\frac{g}{1+ H_1(g,0)}\int_0^{\eta} \frac{1}{\bar{q}_+^3\bar{\rho}_+ (\tau)}\dif \tau\Big) \\
 = & \exp\Big(-\displaystyle\frac{g q_0}{1+ H_1(g,0)}\int_0^{\eta} \frac{1}{\bar{\rho}_-(\tau)}\dif \tau\Big)\notag\\
 % =& \exp\left(-\displaystyle\frac{gq_0}{1+ H_1(g,0)} \cdot \left(-\frac{(1+H_1(g,0))q_0}{g}(\ln\bar{\rho}_-(\eta)- \ln \bar{\rho}_-(0))\right)\right)\\
  = &\left(\displaystyle\frac{\bar{\rho}_-}{p_0}\right)^{q_0^2}.\notag
\end{align}

We can now prove the following lemma.

\begin{lem}\label{posi}
Suppose that \eqref{PCI}-\eqref{PE} hold.
If\begin{equation}
  \mathcal{R}_{g\sigma}(\xi_1 )< \mathcal{P}_{g\sigma}< \mathcal{R}_{g\sigma}(0),
\end{equation}
 where
 \begin{align}
  \mathcal{R}_{g\sigma}(\xi)\defs & g \cdot{K}\int_0^1 \int_0^\xi \dot{\theta}_-^{(1)}(\tau,\eta)\dif \tau \dif \eta,\\
    \mathcal{P}_{g\sigma}\defs &  g\sigma\cdot \displaystyle\frac{q_0^2 -1}{p_0^2 q_0} \int_0^1 {P}_E(\eta)\dif \eta,
\end{align}
with ${K}: = - \displaystyle\frac{(q_0^2 -1)^2}{p_0 q_0} < 0$,
then there exists a $\bar{\xi}_*\in (0, \xi_1)$, such that
\begin{equation}
  \mathcal{R}_{g\sigma} (\bar{\xi}_*) = \mathcal{P}_{g\sigma}.
\end{equation}
Furthermore, for $0<{\sigma} \leq g^3$,
 there exists a unique $\dot{\xi}_*\in (0, \xi_1)$ such that
\begin{align}
\mathcal{R} (\dot{\xi}_*) =& \mathcal{P},
\end{align}
where
\begin{align}
  \mathcal{R}({\xi})\defs & g\cdot K\int_0^1 \Big(\displaystyle\frac{\bar{\rho}_-}{p_0}\Big)^{q_0^2-1}
  \int_0^{{\xi}} \dot{\theta}_-(\tau,\eta)\dif \tau \dif \eta,\\
  \mathcal{P}\defs & \displaystyle\frac{q_0^2 -1}{p_0^{q_0^2}q_0}\int_0^1 \bar{\rho}_-^{q_0^2-2}\Big(\frac{1}{q_0^2} \dot{P}_e^{\sharp}(\eta) -  {P}_I^*(\eta)\Big)\dif \eta.
\end{align}
In addition,
\begin{align}
   |\dot{\xi}_*  - \bar{\xi}_*|  \leq& \dot{C}_* g,
\end{align}
where the constant $\dot{C}_*$ depends on $p_0, q_0$, $P_I$, $P_E$ and  $\|{\theta}^\aleph\|_{\mathcal{C}^0((0,\bar{\xi}_*)\times (0,1))}$.

 \end{lem}

\begin{proof}
 Now we analyze the identity \eqref{eq117}.
By employing the equation \eqref{sf}, it holds that
\begin{align}\label{eq1017}
  &\int_0^{1} \mathcal{A}_+ \mathcal{A}_3 \dot{\pounds}_1 (\dot{\xi}_*,\eta) \dif \eta\notag\\
   = &\int_0^{1}\displaystyle\frac{1}{1+ H_1(g,0)}
  \Big(\displaystyle\frac{\bar{\rho}_-}{p_0}\Big)^{q_0^2}\displaystyle
  \frac{1}{\bar{\rho}_+\bar{q}_+}\displaystyle\frac{q_0^2 -1}{\bar{\rho}_-}{P}_I^*(\eta)\dif \eta\\
  & - \int_0^{1}\displaystyle\frac{1}{1+ H_1(g,0)}
  \Big(\displaystyle\frac{\bar{\rho}_-}{p_0}\Big)^{q_0^2}\displaystyle
  \frac{1}{\bar{\rho}_+\bar{q}_+} g \int_0^{\dot{\xi}_*} \dot{\theta}_-(\tau,\eta)\dif \tau \dif \eta\notag\\
  & + \int_0^{1}
  \Big(\displaystyle\frac{\bar{\rho}_-}{p_0}\Big)^{q_0^2}\Big(2-q_0^2\Big)
  \partial_\eta\int_0^{\dot{\xi}_*}\dot{\theta}_-(\tau,\eta)\dif \tau \dif\eta \notag\\
   : = & J_1 + J_2 + J_3,\notag
\end{align}
where
\begin{align*}
  &J_1\defs \int_0^{1}\displaystyle\frac{1}{1+ H_1(g,0)}
  \Big(\displaystyle\frac{\bar{\rho}_-}{p_0}\Big)^{q_0^2}\displaystyle
  \frac{1}{\bar{\rho}_+\bar{q}_+}\displaystyle\frac{q_0^2 -1}{\bar{\rho}_-}{P}_I^*(\eta)\dif \eta,\\
  &J_2\defs  - g \int_0^{1}\displaystyle\frac{1}{1+ H_1(g,0)}
  \Big(\displaystyle\frac{\bar{\rho}_-}{p_0}\Big)^{q_0^2}\displaystyle
  \frac{1}{\bar{\rho}_+\bar{q}_+} \int_0^{\dot{\xi}_*} \dot{\theta}_-(\tau,\eta)\dif \tau \dif \eta ,\\
  &J_3\defs \int_0^{1}
  \Big(\displaystyle\frac{\bar{\rho}_-}{p_0}\Big)^{q_0^2}\Big(2-q_0^2\Big)
  \partial_\eta\int_0^{\dot{\xi}_*}\dot{\theta}_-(\tau,\eta)\dif \tau \dif\eta.
\end{align*}
For the term $J_3$, integrating by parts, one has
\begin{align}\label{eq1018}
  J_3
   %=& -\int_0^1q_0^2 \Big(2-q_0^2\Big) \Big(p_0^{1-q_0^{2}}\bar{\rho}_-^{q_0^{2}-1}\Big)
%  \displaystyle\frac{\partial_\eta\bar{\rho}_-}{p_0}
%  \int_0^{\dot{\xi}_*} \dot{\theta}_-(\tau,\eta)\dif \tau \dif\eta\notag\\
  % =& g\int_0^1 \displaystyle\frac{1}{1+ H_1(g,0)} \Big(\displaystyle\frac{\bar{\rho}_-}{p_0}\Big)^{q_0^2-1} \Big(2-q_0^2\Big)\displaystyle\frac{q_0}{p_0}
%  \int_0^{\dot{\xi}_*}\dot{\theta}_-(\tau,\eta)\dif \tau \dif \eta\notag\\
  =& \displaystyle\frac{ g}{1+ H_1(g,0)}\displaystyle\frac{(2-q_0^2)q_0}{p_0} \int_0^1  \Big(\displaystyle\frac{\bar{\rho}_-}{p_0}\Big)^{q_0^2-1}
  \int_0^{\dot{\xi}_*}\dot{\theta}_-(\tau,\eta)\dif \tau \dif \eta.
\end{align}
Furthermore, one can obtain
\begin{align}\label{eq1019}
 & J_2 + J_3\notag\\
   =& \displaystyle\frac{g}{1+ H_1(g,0)} \Big(-\displaystyle\frac{1}{p_0 q_0}+ \displaystyle\frac{2-q_0^2}{p_0}q_0\Big)\int_0^1 \Big(\displaystyle\frac{\bar{\rho}_-}{p_0}\Big)^{q_0^2-1}
  \int_0^{\dot{\xi}_*} \dot{\theta}_-(\tau,\eta)\dif \tau \dif \eta\\
   =& -\displaystyle\frac{g}{1+ H_1(g,0)}  \displaystyle\frac{(q_0^2 -1)^2}{p_0 q_0}\int_0^1 \Big(\displaystyle\frac{\bar{\rho}_-}{p_0}\Big)^{q_0^2-1}
  \int_0^{\dot{\xi}_*} \dot{\theta}_-(\tau,\eta)\dif \tau \dif \eta.\notag
  \end{align}
Moreover,
\begin{align}
 &\int_0^{1}  \mathcal{A}_+ \mathcal{A}_3 \dot{P}_e^\sharp(\eta)\dif \eta - J_1\notag\\
   =& \displaystyle\frac{1}{1+ H_1(g,0)}\int_0^{1}
  \Big(\displaystyle\frac{\bar{\rho}_-}{p_0}\Big)^{q_0^2}\displaystyle
  \frac{1}{\bar{\rho}_+\bar{q}_+}
  \Big(\frac{1-\bar{M}_+^2}
  {\bar{\rho}_+\bar{q}_+^2} \dot{P}_e^\sharp(\eta) - \displaystyle\frac{q_0^2 -1}{\bar{\rho}_-}{P}_I^*(\eta)\Big)\dif \eta\\
   = &\displaystyle\frac{1}{1+ H_1(g,0)}\int_0^{1}
  \Big(\displaystyle\frac{\bar{\rho}_-}{p_0}\Big)^{q_0^2}\displaystyle
  \frac{1}{\bar{\rho}_+\bar{q}_+}
  \displaystyle\frac{q_0^2 -1}{\bar{\rho}_-}\Big(\frac{1}{q_0^2} \dot{P}_e^\sharp(\eta) -{P}_I^*(\eta)\Big)\dif \eta\notag\\
   =&\displaystyle\frac{1 }{1+ H_1(g,0)}\displaystyle\frac{q_0^2 -1}{p_0^{q_0^2}q_0}\int_0^1 \bar{\rho}_-^{q_0^2-2}\Big(\frac{1}{q_0^2} \dot{P}_e^\sharp(\eta) -  {P}_I^*(\eta)\Big)\dif \eta.\notag
\end{align}
Then the equation \eqref{eq117} yields that
\begin{align}
   J_2 + J_3 = \int_0^{1}  \mathcal{A}_+ \mathcal{A}_3 \dot{P}_e^\sharp(\eta)\dif \eta - J_1.
\end{align}
That is
\begin{equation}
\begin{split}
  &-g\displaystyle\frac{(q_0^2 -1)^2}{p_0 q_0}\int_0^1 \Big(\displaystyle\frac{\bar{\rho}_-}{p_0}\Big)^{q_0^2-1}
  \int_0^{\dot{\xi}_*} \dot{\theta}_-(\tau,\eta)\dif \tau \dif \eta\\
  = & \displaystyle\frac{q_0^2 -1}{p_0^{q_0^2}q_0}\int_0^1 \bar{\rho}_-^{q_0^2-2}\Big(\frac{1}{q_0^2} \dot{P}_e^\sharp(\eta) - {P}_I^*(\eta)\Big)\dif \eta.
\end{split}
\end{equation}
Let
\begin{align}
  \mathcal{R}({\xi})\defs & g\cdot K\int_0^1 \Big(\displaystyle\frac{\bar{\rho}_-}{p_0}\Big)^{q_0^2-1}
  \int_0^{{\xi}} \dot{\theta}_-(\tau,\eta)\dif \tau \dif \eta,\\
  \mathcal{P}\defs & \displaystyle\frac{q_0^2 -1}{p_0^{q_0^2}q_0}\int_0^1 \bar{\rho}_-^{q_0^2-2}\Big(\frac{1}{q_0^2} \dot{P}_e^\sharp(\eta) -  {P}_I^*(\eta)\Big)\dif \eta,
\end{align}
where ${K}: = - \displaystyle\frac{(q_0^2 -1)^2}{p_0 q_0} < 0$.
Applying \eqref{iniPE}, one has
\begin{align}
\frac{1}{q_0^2} \dot{P}_e^\sharp(\eta) - {P}_I^*(\eta)= {P}_E(\eta)g\sigma+ O(1)g^2 \sigma+ O(1)\sigma^2.
 \end{align}
In addition, by employing \eqref{rho-}, it follows that
 \begin{align}\label{rh-}
  \bar{\rho}_-^{q_0^2-2} = \Big(p_0 - \displaystyle\frac{g}{(1+ H_1(g,0))q_0}\eta \Big)^{q_0^2-2} = p_0^{q_0^2-2}- \frac{q_0^2-2}{q_0}p_0^{q_0^2-3}\eta g + O(1)g^2.
     \end{align}
     Therefore,
     \begin{align}
        \bar{\rho}_-^{q_0^2-2}\Big(\frac{1}{q_0^2} \dot{P}_e^\sharp(\eta) - {P}_I^*(\eta)\Big)
       % =& \Big(p_0^{q_0^2-2}- \frac{q_0^2-2}{q_0}p_0^{q_0^2-3}\eta g + O(g^2) \Big)\cdot \Big(\dot{P}_E(\eta)g\sigma+ O(1)g^2 \sigma \Big)\notag\\
        =p_0^{q_0^2-2} {P}_E(\eta)g\sigma + O(1)g^2 \sigma+ O(1)\sigma^2,
     \end{align}
    where $O(1)$ depends on $p_0, q_0$, $P_I$ and ${P}_E$.
  % Assume that
%   \begin{align}
%      \dot{P}_e(\eta;g,\sigma)\defs A(\eta)\sigma + B(\eta)g + b_{11}(\eta)\sigma^2 + b_{12}(\eta)g\sigma + b_{22}(\eta)g^2.
%   \end{align}
%    Now we consider
%    \begin{align}
%      &\Big(  p_0^{q_0^2-2}- \frac{q_0^2-2}{q_0}p_0^{q_0^2-3}\eta g + O(g^2)\Big)\cdot \Big(\frac{1}{q_0^2} \dot{P}_e(\eta) -  \sigma \dot{P}_I(\eta)\Big)\notag\\
%      =& \Big(  p_0^{q_0^2-2}- \frac{q_0^2-2}{q_0}p_0^{q_0^2-3}\eta g + O(g^2)\Big)\notag\\
%      &\cdot \Big(\frac{1}{q_0^2}\Big(A(\eta)\sigma + B(\eta)g + b_{11}(\eta)\sigma^2 + b_{12}(\eta)g\sigma + b_{22}(\eta)g^2 \Big) -  \sigma \dot{P}_I(\eta)\Big).
%    \end{align}
%    Let
%    \begin{align}
%     B(\eta) =& b_{11}(\eta)= b_{22}(\eta)=0,\\
%     A(\eta) =& q_0^2 \dot{P}_I(\eta),\\
%    b_12(\eta)=& P_E(\eta).
%    \end{align}
%    That is
%    \begin{align}
% A(\eta) =& q_0^2 \dot{P}_I(\eta),\\
%     b_{12}(\eta) = & \frac{q_0^2 -2}{p_0q_0} \eta A(\eta) + p_0^{2-q_0^2} q_0^2 P_E(\eta)\notag\\
%     =& \frac{q_0^2 -2}{p_0}q_0 \eta \dot{P}_I(\eta) + p_0^{2-q_0^2} q_0^2 P_E(\eta).
%    \end{align}
%    Therefore,
%      \begin{align}
%      \dot{P}_e(\eta)\defs q_0^2 \dot{P}_I(\eta)\sigma + \Big(\frac{q_0^2 -2}{p_0}q_0 \eta \dot{P}_I(\eta) + p_0^{2-q_0^2} q_0^2 P_E(\eta)\Big)g\sigma.
%   \end{align}
Denote
 \begin{align}
 \mathcal{R}_{g\sigma}(\xi)\defs & g\cdot {K}\int_0^1 \int_0^\xi \dot{\theta}_-^{(1)}(\tau,\eta)\dif \tau \dif \eta,\label{eq163}\\
    \mathcal{P}_{g\sigma}\defs & g\sigma\cdot  \displaystyle\frac{q_0^2 -1}{p_0^2 q_0}\int_0^1 {P}_E(\eta)\dif \eta.
\end{align}
By applying Lemma \ref{w0}, it is obvious that $\mathcal{R}_{g\sigma}^{'}(\xi)< 0$. Furthermore, applying \eqref{PCI}-\eqref{P1.23} and \eqref{expressedR}, one has
\begin{equation}\label{eq164}
\begin{split}
  &-\tilde{C}_0 g\sigma \geq \inf \limits_{\xi\in (0,\xi_1)}\mathcal{R}_{g\sigma}(\xi)= \mathcal{R}_{g\sigma}(\xi_1)= g\cdot {K}\int_{0}^{1}\int_{0}^{\xi_1} \dot{\theta}_-^{(1)}(\tau,\eta)\dif \tau \dif \eta\\
  =& g \sigma \cdot \mathcal{R}_{g\sigma}^\natural(\xi_1)
  \geq -\tilde{C}_1 g\sigma,
\end{split}
\end{equation}
where the constants $\tilde{C}_0, \tilde{C}_1>0$ and depend on $C_0$, $C_1$, $\xi_1$, $K$, $C_I$ and $C_{Ii}$, $(i=0,1,2)$. In addition,
\begin{equation}
  \begin{aligned}
  &\sup \limits_ {\xi\in (0,\xi_1)}\mathcal{R}_{g\sigma}(\xi) = \mathcal{R}_{g\sigma} (0) = 0.
  \end{aligned}
\end{equation}
 Obviously, there exists a $\bar{\xi}_*\in (0 , \xi_1)$ such that
\begin{equation}\label{eq270000}
  \mathcal{R}_{g\sigma} (\bar{\xi}_*) = \mathcal{P}_{g\sigma},
\end{equation}
if and only if
\begin{equation}\label{eq167}
  \mathcal{R}_{g\sigma}(\xi_1 )< \mathcal{P}_{g\sigma}< \mathcal{R}_{g\sigma}(0).
\end{equation}
Let $\dot{\xi}_{*} = \bar{\xi}_* + \delta\dot{\xi}_*$, denote
\begin{align}
  \dot{I}(\delta\dot{\xi}_*; \mathcal{R}(\dot{\xi}_{*}),\mathcal{P}; \dot{U}_-)
  \defs\mathcal{R}(\dot{\xi}_{*}) - \mathcal{P}.
\end{align}
Applying \eqref{eq270000}, one has
\begin{equation}
  \dot{I}(0; \mathcal{R}_{g\sigma}(\bar{\xi}_*),\mathcal{P}_{g\sigma} ;\dot{U}_-) = \mathcal{R}_{g\sigma}(\bar{\xi}_*) - \mathcal{P}_{g\sigma} = 0.
\end{equation}
Moreover,
\begin{equation}\label{zx}
\begin{split}
  \dot{I}(\delta\dot{\xi}_*; \mathcal{R}(\dot{\xi}_{*}),\mathcal{P}; \dot{U}_-) =& g\cdot {K}\int_0^{1} \dot{\theta}_-^{(1)}(\bar{\xi}_*,\eta)\dif \eta \cdot \delta \dot{\xi}_* + O(1)\big(g^2 \sigma + g\sigma^2 \big)\cdot \delta \dot{\xi}_* \\
  &+ O(1)g^2\sigma+ O(1)\sigma^2,
\end{split}
\end{equation}
where $O(1)$ depends on $p_0, q_0$, $P_I$, $P_E$ and $\bar{\xi}_*$.

The expansion \eqref{zx} yields that
\begin{align}
  \displaystyle\frac{\partial \dot{I}}{\partial \delta\dot{\xi}_*}(0; \mathcal{R}_{g\sigma}(\bar{\xi}_*),\mathcal{P}_{g\sigma} ;\dot{U}_-) =  g\cdot  {K}\int_0^{1} \dot{\theta}_-^{(1)}(\bar{\xi}_*,\eta)\dif \eta +O(1)\big(g^2 \sigma + g\sigma^2 \big)< 0.
  \end{align}
 Therefore, by applying the implicit function theorem and $0<{\sigma} \leq g^3$, there exists a unique $\delta\dot{\xi}_*$ such that
 \begin{align}
   |\delta\dot{\xi}_* |  \leq \bar{C}_* \Big(g+\frac{\sigma}{g}\Big)\leq \dot{C}_* g ,
 \end{align}
 where the constant $\dot{C}_*$ depends on $p_0, q_0$, $P_I$, $P_E$ and $\|{\theta}^\aleph\|_{\mathcal{C}^0((0,\bar{\xi}_*)\times (0,1))}$. Furthermore, for sufficiently small $g$,
 %if
% \begin{align}
%0< g<  \min\Big\{\frac{\bar{\xi}_*}{\dot{C}_*}, \, \frac{\xi_1 - \bar{\xi}_*}{\dot{C}_*} \Big\}.
% \end{align}
 there exists a $\dot{\xi}_*\in (0 , \xi_1)$ such that
  \begin{equation}
  \mathcal{R} (\dot{\xi}_{*}) = \mathcal{P}.
\end{equation}

\end{proof}

Once $\dot{\xi}_*$ is determined, then we can determine $\dot{U}_+$ in the domain $\dot{\Omega}_+$, as the following lemma shows:

 \begin{lem}\label{lem5}
 Let $\beta>2$.  Under the assumptions of Lemma \ref{posi}, there exists a unique solution $\dot{U}_+$ to the equations \eqref{eq66}-\eqref{eq68} with the boundary conditions \eqref{eq83}-\eqref{eq83a} and \eqref{eq109}-\eqref{eq111}. Moreover, the following estimate holds:
\begin{equation}\label{eq168}
\begin{aligned}
 & \| \dot{U}_+ \|_{(\dot{\Omega}_+;\dot{\Gamma}_s)} + \|\dot{\psi}'\|_{W_\beta^{1-\frac{1}{\beta}}(\dot{\Gamma}_s)}\\
  \leq& {C} \cdot\Big(\| {P}_I^*\|_{\mcc^{2,\alpha}(\Gamma_1)} + \|\dot{P}_e^\sharp\|_{\mcc^{2,\alpha}(\Gamma_3)} + H_{2_+}(g, \sigma) \Big) \leq \dot{C}_+ \sigma,
  \end{aligned}
\end{equation}
where the constant $\dot{C}_+$ depends only on $\bar{U}_\pm$, $L$, $P_I$, $P_E$ and $\beta$.
\end{lem}
\begin{proof}
   By applying \eqref{x20} in Theorem 3.1, it follows that
  \begin{equation}\label{x32}
  \begin{split}
  &\| \dot{p}_+ \|_{W_\beta^1(\dot{\Omega}_+)} +  \| \dot{\theta}_+\|_{W_\beta^1(\dot{\Omega}_+)} +  \| \dot{q}_+\|_{\mcc^0(\dot{\Omega}_+)}+  \|\dot{q}_+\|_{W_\beta^{1-\frac{1}{\beta}}(\dot{\Gamma}_s)}\\
   \leq& {C}\Big( \sum_{i=1}^{2} \|\dot{\pounds}_i\|_{W_\beta^{1-\frac{1}{\beta}}(\dot{\Gamma}_s)}+ \|\dot{P}_e^\sharp\|_{\mcc^{2,\alpha}(\Gamma_3)} +  H_{2_+}(g, \sigma) \Big).
   \end{split}
   \end{equation}
Moreover, according to $\eqref{eq111}$, one has
\begin{equation}\label{eq205}
   \|\dot{\psi}'\|_{W_\beta^{1-\frac{1}{\beta}}(\dot{\Gamma}_s)}\leq {C} \Big(\|\dot{\theta}_+\|_{W_\beta^{1-\frac{1}{\beta}}(\dot{\Gamma}_s)}
   + \|\dot{\theta}_-\|_{W_\beta^{1-\frac{1}{\beta}}(\dot{\Gamma}_s)}
   \Big).
\end{equation}
Therefore, employing Lemma \ref{w0} and Lemma \ref{e1}, one can obtain  the estimate $\eqref{eq168}$.

\end{proof}

\section{The nonlinear iteration scheme}

\subsection{The supersonic flow $U_-$ in $\Omega$}\

\begin{lem}\label{zc}
Suppose that \eqref{PCI}-\eqref{P1.23} hold, then there exists a positive constant $\sigma_L$ depending on $\bar{U}_-$ and $L$, such that for any $0<\sigma< \sigma_L$, there exists a unique solution $U_-\in {\mcc}^{2,\alpha}(\bar{\Omega})$ to the equations \eqref{eq260000}-\eqref{b1} with the initial-boundary conditions \eqref{eq44}-\eqref{eq45}. Moreover, denote $U_- : = \bar{U}_- + \delta U_-$, then the following estimates hold:
\begin{align}
\|U_- - \bar{U}_- \|_{\mcc^{2,\alpha}(\bar{\Omega})} \leq& C\| P_I^\sharp\|_{\mcc^{2,\alpha}(\Gamma_1)} \leq C_L \sigma,\label{eq174}\\
\|\delta U_- - \dot{U}_- \|_{{\mcc}^{1,\alpha}(\bar{\Omega})} \leq& C_L^{\natural} \sigma^2,\label{eq175}
\end{align}
where the constants $C_L$ and $C_L^{\natural}$ depend on $\bar{U}_-$, $P_I$, and $L$.
\end{lem}

\begin{proof}
The existence of the unique solution $U_-\in {\mcc}^{2,\alpha}(\bar{\Omega})$ can be obtained by employing the theory in the book \cite{LY}. Thus, it suffices to verify $\eqref{eq175}$.

The equations $\eqref{eq260000}$-\eqref{b1} can be rewritten as the following matrix form:
\begin{align}
&B(g,\sigma) \partial_\eta U + B_1 (U)\partial_\xi U+ g\cdot b(U)  = 0,\label{eq176}
\end{align}
where $U = (p, \theta, q)^\top$, $b(U) = \Big(\displaystyle\frac{\cos\theta}{q}, -\displaystyle\frac{\sin\theta}{\rho q^3}, \tan\theta\Big)^\top,$
\[ B(g,\sigma) = \begin{pmatrix}
   1+H_1(g,\sigma) & 0 & 0  \\
   0 & 1+H_1(g,\sigma) & 0  \\
   0 & 0 & 0
\end{pmatrix}, \, B_1(U)=
\begin{pmatrix}
   -\displaystyle\frac{\sin\theta}{\rho q } & q \cos\theta& 0  \\
   \displaystyle\frac{(M^2 -1)\cos\theta}{\rho^2 q^3}& -\displaystyle\frac{\sin \theta}{\rho q } & 0  \\
   \displaystyle\frac{1}{\rho}& 0 & q
\end{pmatrix}.\]
Therefore, $\delta U_- - \dot{U}_-$ satisfies the following problem:
 \begin{equation}\label{eq179}
 \begin{split}
   &B(g,0) \partial_\eta (\delta U_- - \dot{U}_-) + B_1 (\bar{U}_-)\partial_\xi (\delta U_- -\dot{U}_-) + g\cdot \nabla b(\bar{U}_-)(\delta U_- -\dot{U}_-)\\
   = & F_-(\delta U_-),\quad \text{in}\quad \Omega
   \end{split}
   \end{equation}
with the following initial-boundary conditions
  \begin{align}\label{eq179}
    \delta U_- - \dot{U}_- =&\big(P_I^\sharp(Y_0(\eta;g,\sigma))- {P}_I^*(\eta),0,0\big)^\top
    ,&\quad &\text{on} \quad \Gamma_1\\
    \delta \theta_- - \dot{\theta}_- = &0,&\quad &\text{on} \quad \Gamma_2\cup \Gamma_4
  \end{align}
 where
 \begin{align}
  F_-(\delta U_-):  =&-B_{\sigma} \partial_\eta \delta U_- + \Big( B_1(\bar{U}_-) - B_1({U}_-) \Big)\partial_\xi \delta U_- \notag\\
  &- g\cdot \Big(b(U_-) - b(\bar{U}_-)- \nabla b(\bar{U}_-) \delta U_-\Big),
\end{align}
with
\[ B_{\sigma} = \begin{pmatrix}
   H_1(g,\sigma)-H_1(g,0) & 0 & 0  \\
   0 &  H_1(g,\sigma)-H_1(g,0)  & 0  \\
   0 & 0 & 0
\end{pmatrix}.\]
Then, applying \eqref{eq174}, one can infer that
\begin{align}\label{eq180}
&\|\delta U_- - \dot{U}_- \|_{{\mcc}^{1,\alpha}(\bar{\Omega})}\notag\\
\leq& C \Big(\| F_-(\delta U_-)\|_{\mcc^{1,\alpha}(\bar{\Omega})} + \|P_I^\sharp(Y_0(\eta;g,\sigma))- {P}_I^*(\eta)\|_{\mcc^{2,\alpha}(\Gamma_1)} \Big)\\
 \leq & C\Big(\sigma \| \partial_\eta \delta U_- \|_{\mcc^{1,\alpha}(\bar{\Omega})}+ \| \partial_\xi \delta U_- \|_{\mcc^{1,\alpha}(\bar{\Omega})}\cdot \|\delta U_- \|_{\mcc^{1,\alpha}(\bar{\Omega})}+ g\|\delta U_- \|_{{\mcc}^{1,\alpha}(\bar{\Omega})}^2 + \sigma^2\Big)\notag\\
  \leq & C_L^{\natural} \sigma^2.\notag
\end{align}

\end{proof}

\subsection{The shock front and subsonic flow}
For the given quantities $P_I$, $P_E$ and $U_-$ in Lemma \ref{zc}, one needs to find the shock solution $(U_+;\psi)$. Here, $\psi$ is the shock front:
\begin{align}
  \Gamma_s\defs \{(\xi,\eta): \xi = \psi(\eta): = \dot{\xi}_* + \delta \psi(\eta), 0<\eta<1 \},
\end{align}
which close to the initial approximating shock front $\dot{\Gamma}_{\mr{s}}$.
Correspondingly, the subsonic region is
\begin{align}
   {\Omega}_+ = \{ (\xi, \eta)\in \mathbb{R}^2 : \psi(\eta) < \xi < L, 0 < \eta < 1\},
\end{align}
and the subsonic flow $U_+$ is supposed to be closed to $\bar{U}_+$.

Thus, The solution $U_+$ satisfies the following free boundary value problem
\begin{align}
&B(g,\sigma) \partial_\eta U_+ + B_1 (U_+) \partial_\xi U_+ + g\cdot b(U_+)  = 0,&\quad &\text{in}\quad \Omega_+\label{w1}\\
&\text{The R-H conditions} \quad\eqref{eq30}-\eqref{eq33},&\quad &\text{on}\quad \Gamma_{\mr{s}}\\
&\theta_+ = 0 ,&\quad &\text{on} \quad (\Gamma_{2}\cup\Gamma_{4})\cap\overline{\Omega_+}\\
&p_+ =  P_{\mr{out}}  (Y_L(\eta;g,\sigma);g,\sigma).&\quad  &\text{on}\quad \Gamma_3\label{w2}
\end{align}
The next step is to solve this free boundary value problem near $(\bar{U}_+; \dot{\psi})$. It should be pointed out that the free boundary $\psi$ will be determined by the shape of the shock front $\psi'$ and an exact point $\xi_* : = \psi(1)$ on the nozzle. That is, $\psi(\eta)$ will be rewritten as below:
\begin{equation}\label{bn}
  \psi(\eta ) = \xi_* - \int_{\eta}^1 \delta \psi'(\tau)\dif \tau,
\end{equation}
where $\xi_* := \dot{\xi}_* + \delta \xi_*$, and $\delta \xi_*$ will be determined by the solvability condition for the solution $U_+$.

Then the following transformation will be introduced to fix the
free boundary $\Gamma_s$:
\begin{align*}\label{eq229}
\mathcal{T} : \begin{cases}
\tilde{\xi} = L + \displaystyle\frac{L - \dot{\xi}_*}{L - {\psi}(\eta)}(\xi - L),\\
\tilde{\eta} = \eta,
\end{cases}
\end{align*}
with the inverse
\begin{align*}
\mathcal{T}^{-1} :
 \begin{cases}
\xi = L + \displaystyle\frac{L -  {\psi}(\tilde{\eta})}{L - \dot{\xi}_*}(\tilde{\xi} - L),\\
\eta = \tilde{\eta}.
\end{cases}
\end{align*}
Under this transformation, the domain $\Omega_+$ becomes
 \begin{equation}
   \dot{\Omega}_+ = \{ (\xi, \eta)\in \mathbb{R}^2 : \dot{\xi}_* < \xi < L, 0 < \eta < 1\},
 \end{equation}
 which is exactly the domain of initial approximating subsonic domain.

Let $\tilde{U}(\tilde{\xi}, \tilde{\eta})\defs U_+\circ \mathcal{T}^{-1}(\tilde{\xi}, \tilde{\eta})$. Direct calculations yield that $\tilde{U}$ satisfies the following equation in $\dot{\Omega}_+$:
\begin{equation}
\begin{split}
  &\widetilde{\mathcal{B}}\partial_{\tilde{\xi}}\tilde{U} + B(g,\sigma) \partial_{\tilde{\eta}}\tilde{U} + g\cdot b(\tilde{U}) = 0,\label{w}
\end{split}
\end{equation}
where
\begin{align*}
 \widetilde{\mathcal{B}} \defs \frac{(\tilde{\xi} - L)\cdot \psi'(\tilde{\eta})}{L- \psi (\tilde{\eta})}B(g,\sigma) + \frac{L - \dot{\xi}_*}{L - {\psi}(\tilde{\eta})}B_1(\tilde{U}).
\end{align*}
In addition, $\tilde{U}$ satisfies the following boundary conditions
\begin{align}
&\tilde{\theta} = 0 , &\quad &\text{on} \quad (\Gamma_{2}\cup\Gamma_{4})\cap\overline{\Omega_+}\\
&\tilde{p} = P_{\mr{out}} (Y_L(\tilde{\eta};g,\sigma);g,\sigma).&\quad  &\text{on}\quad \Gamma_3
\end{align}
Moreover, the R-H conditions \eqref{eq30}-\eqref{eq33} become
\begin{align}
  &G_i(\tilde{U}, U_-(\psi',\xi_*))=0, \quad i=1,2,\quad &\text{on}\quad \dot{\Gamma}_{{s}}\\
  &G_3(\tilde{U}, U_-(\psi',\xi_*);\psi)=0, \quad &\text{on}\quad \dot{\Gamma}_{{s}}\label{ww}
\end{align}
where $U_-(\psi',\xi_*)\defs U_-(\psi(\tilde{\eta}),\tilde{\eta})$.

Therefore, the free boundary problem \eqref{w1}-\eqref{w2} can be transformed into the fixed boundary problem \eqref{w}-\eqref{ww}. Then
an iteration scheme will be constructed to prove the existence of the solutions.

For simplicity of the notations, we shall drop `` $ \tilde{} $ '' in the sequel arguments.

\subsection{The linearized problem for the iteration}
This subsection is devoted to describe the linearized problem for the nonlinear iteration scheme, which will be used to prove the existence of solution to the problem \eqref{w}-\eqref{ww} in the next section.

Given approximating states $U=\bar{U}_+  + \delta U$ of the subsonic flow behind the shock front, as well as approximating shape of the shock front $\psi' = \delta\psi'$, then we update a new approximate state ${U}^{*} = \bar{U}_+ + \delta {U}^{*}$ of the subsonic flow and the shape of the shock front ${\psi^*}'= \delta{{\psi}^{*}}'$.
\begin{enumerate}
\item $\delta {U}^{*}:=(\delta p^*, \delta \theta^*, \delta q^*)^\top$ satisfies the following linearized equations in $\dot{\Omega}_+$
\begin{align}
  &\Big(1+H_1(g,0)\Big)\partial_\eta \delta p^* + \bar{q}_+ \partial_\xi \delta \theta^* - \frac{g}{\bar{q}_+^2} \delta q^* = f_1(\delta U, \delta \psi', \delta \xi_{*})+H_{2_+}(g, \sigma),\label{eq190}\\
   &\Big(1+H_1(g,0)\Big)\partial_\eta \delta {\theta}^* - \frac{1}{\bar{\rho}_+\bar{q}_+}\frac{1- {\bar{M}_+^2}}{\bar{\rho}_+{\bar{q}_+^2}}\partial_\xi \delta p^* - \frac{g}{\bar{\rho}_+{\bar{q}_+^3}}\delta {\theta}^* = f_2(\delta U, \delta \psi', \delta \xi_{*}),\label{eq192}\\
 & \partial_\xi\Big (\bar{q}_+\cdot \delta q^* + \frac{1}{\bar{\rho}_+} \delta p^*\Big )+ g\cdot \delta {\theta}^* = \partial_\xi f_3(\delta U) + f_4(\delta U,\delta \psi', \delta \xi_{*} ),\label{eq194}
\end{align}
where
\begin{align*}
  f_1(\delta U, \delta \psi', \delta \xi_{*}) : = &\Big(  \bar{q}_+\partial_\xi \delta \theta - \displaystyle\frac{g}{\bar{q}_+^2}\delta q + \frac{g}{\bar{q}_+}\Big)\\
  &- \Big(- \frac{\sin {\theta}}{{\rho}{q}} \partial_\xi  p+ {q} \cos \theta \partial_\xi \theta + g\displaystyle\frac{\cos\theta}{q} \Big)\\
  &+(1+H_1(g,\sigma)) \frac{ L-{\xi}}{L- \psi ({\eta})}\cdot \delta \psi'({\eta}) \partial_\xi  p \\
  &+\displaystyle\frac{\delta \xi_* - \int_{\eta}^1 \delta \psi'(\tau)\dif \tau}{L -\psi(\eta)} \Big(\frac{\sin \theta}{\rho q}\partial_\xi p- q \cos \theta \partial_\xi \theta\Big)\\
   &-\Big(H_1(g,\sigma) - H_1(g,0) \Big)\partial_\eta \delta p ,\\
   f_2(\delta U, \delta \psi', \delta \xi_{*}): = &\Big(  - \frac{1}{\bar{\rho}_+\bar{q}_+}\frac{1-{\bar{M}_+^2}}{\bar{\rho}_+{\bar{q}_+^2}} \partial_\xi \delta p - \frac{g}{\bar{\rho}_+\bar{q}_+^3} \delta \theta \Big)\\
    & - \Big( - \frac{\sin {\theta}}{{\rho}{q}}\partial_\xi \theta - \frac{\cos {\theta}}{{\rho}{q}}\frac{1-{M}^2}{{\rho}{{q}}^2} \partial_\xi  p -g\frac{\sin\theta}{\rho q^3}\Big)\\
    & + (1+H_1(g,\sigma))\frac{L-\xi}{L- \psi ({\eta})}\cdot \delta \psi'({\eta})\partial_\xi \theta\\
     &+\displaystyle\frac{\delta \xi_* - \int_{\eta}^1 \delta \psi'(\tau)\dif \tau}{L -\psi(\eta)}\Big(\frac{\sin \theta}{\rho q}\partial_\xi\theta +
      \frac{\cos \theta}{\rho q} \frac{1- {{M}}^2}{\rho q^2}\partial_\xi  p \Big)\\
      &- \Big(H_1(g,\sigma) - H_1(g,0) \Big)\partial_\eta \delta \theta ,\\
     f_3(\delta U): =&\Big(\bar{q}_+\cdot \delta q + \frac{1}{\bar{\rho}_+} \delta p\Big) - B(U),\\
     f_4(\delta U,\delta \psi', \delta \xi_{*} )\defs &g \cdot  \displaystyle\frac{\delta \xi_* - \int_{\eta}^1 \delta \psi'(\tau)\dif \tau}{L -\dot{\xi}_*}\tan \theta + g\cdot (\delta \theta - \tan \theta).
  \end{align*}
\item On the nozzle walls,
\begin{align}\label{eq198}
\delta \theta^* = 0 ,\quad \text{on}\quad (\Gamma_{2}\cup\Gamma_{4})\cap\overline{\dot\Omega_+}.
\end{align}
\item On the exit,
\begin{align}\label{eq199}
  &\delta p^* = P_e^*(\eta,\delta U) \defs P_e^\sharp(Y_L(\eta;g,\sigma;\delta U);g,\sigma),&\quad &\text{on}\quad\Gamma_3,
\end{align}
where
$$Y_L(\eta;g,\sigma;\delta U) = \displaystyle\frac{1}{1+H_1(g,\sigma)}\int_0^\eta \displaystyle\frac{1}{(\rho q \cos \theta)(L,s)} \dif s.$$

\item On the free boundary $\dot{\Gamma}_s$,
\begin{align}
&\alpha_{j+}\cdot \delta U^* = G_j^*(\delta U,\delta U_-, \delta \psi', \delta \xi_*),\quad j = 1,2,\label{eq200}\\
& \alpha_{3+} \cdot \delta U^*-(1+H_1(g,0))[\bar{p}] \delta {\psi^*}'=  G_3^*(\delta U,\delta U_-, \delta \psi', \delta \xi_*),\label{G3}
\end{align}
where
\begin{align}
 & G_j^*(\delta U,\delta U_-, \delta \psi', \delta \xi_*)\defs \alpha_{j+}\cdot \delta {U} - G_j(U, U_-(\psi',\xi_*)), \\
  & G_3^*(\delta U,\delta U_-, \delta \psi', \delta \xi_*) \defs \alpha_{3+} \cdot \delta U - (1+H_1(g,0))[\bar{p}]\delta \psi'\notag\\
  &\qquad \qquad \qquad \qquad \qquad \quad- G_3(U, U_-(\psi',\xi_*);\psi').
\end{align}
\end{enumerate}

\begin{rem}
The condition \eqref{eq200} can be rewritten as the following form:
\begin{align}
 A_*(\delta {p}^*,\delta {q}^*)^\top =  (G_1^*, G_2^*)^\top : = \mathbf{G}.
\end{align}
 Applying Lemma \ref{e1}, one has $\det A_* \neq 0$. Thus $(\delta {p}^*,\delta {q}^*)$ can be determined uniquely, that is
\begin{align}
  \delta p^* \defs& \pounds_1^*,\label{eq202}\\
  \delta q^*\defs & \pounds_2^*,\label{eq202+}
\end{align}
with
 \begin{equation}
  (\pounds_1^*, \pounds_2^*)^\top = A_{*}^{-1} \mathbf{G}.
 \end{equation}
Moreover, employing \eqref{G3}, one has
\begin{equation}\label{eq203}
    \delta {\psi^*}' = \displaystyle\frac{  \alpha_{3+} \cdot \delta U^*- G_3^*(\delta U,\delta U_-, \delta \psi', \delta \xi_*)}{(1+H_1(g,0))[\bar{p}]}\defs \pounds_3^*.
\end{equation}
\end{rem}

In order to find the solution $(U, \psi)$, one needs to construct a suitable function space for $(\delta U, \delta \psi')$ such that $\delta\xi_*$ can be determined, and the iteration mapping
\begin{align*}
  \mathbf{\Pi} : (\delta U; \delta \psi') \mapsto (\delta {U}^*; \delta {\psi^*}';\delta \xi_*)
\end{align*}
 is well defined and contractive.

For simplicity of notations, define the solution $(\delta {U}^*; \delta {\psi^*}';\delta \xi_*)$ to the linearized problem \eqref{eq190}-\eqref{G3} near $(\dot{U}_+; \dot{\psi}';0)$ as an operator:
\begin{align}\label{iteration}
  (\delta {U}^*; \delta {\psi^*}';\delta \xi_*) = \mathscr{T}_e(\mathscr{F}; \mathscr{G};  P_{{e}}^*),
\end{align}
where $\mathscr{F}\defs (f_1 + H_{2+}(g,\sigma), f_2, f_3,f_4)$, $\mathscr{G}\defs(\pounds_1^*, \pounds_2^*, \pounds_3^*) $.
In particular,
\begin{align}\label{initial}
   (\dot{U}_+; \dot{\psi}';0) = \mathscr{T}_e(\dot{\mathscr{F}}; \dot{\mathscr{G}}; \dot{P}_e^\sharp),
\end{align}
where $\dot{\mathscr{F}}\defs (H_{2+}(g,\sigma),0,  -B(\bar{U}_+),0)$, $\dot{\mathscr{G}}\defs (\dot{\pounds}_1,\dot{\pounds}_2, \dot{\pounds}_3)$.

When $\delta\xi_*$ is omitted, it will be denoted by
\begin{align}\label{iteration1}
  (\delta {U}^*; \delta {\psi^*}') = \mathscr{T}(\mathscr{F}; \mathscr{G}; P_{{e}}^*),
\end{align}
and
\begin{align}\label{initial1}
   (\dot{U}_+; \dot{\psi}') = \mathscr{T}(\dot{\mathscr{F}}; \dot{\mathscr{G}};\dot{P}_e^\sharp)
\end{align}
respectively.

\subsection{The solvability condition and a prior estimates of $\delta U^*$}\

By applying Theorem \ref{cv0} and taking
\begin{align*}
  &\mathcal{U}\defs \delta{p}^*, \,\mathcal{V} \defs\delta{\theta}^*,\, \mathcal{W} \defs\delta{q}^*, \,\mathcal{F}_1 \defs \displaystyle\frac{f_1 + H_{2+}(g,\sigma)}{ 1+H_1(g,0) },\, \mathcal{F}_2 \defs \displaystyle\frac{f_2}{ 1+H_1(g,0) },\\
  &\mathcal{F}_3 \defs \frac{f_3}{\bar{q}_+},\,\mathcal{F}_4 \defs \frac{f_4}{\bar{q}_+} ,\, \mathcal{U}_s\defs {\pounds}_1^*,\, \mathcal{W}_s\defs {\pounds}_2^*,\,\, \mathcal{U}_3\defs P_e^*,\\
  &\mathcal{A}_1 \defs \displaystyle\frac{\bar{q}_+}{ 1+H_1(g,0) }, \,\,\mathcal{A}_2 \defs \displaystyle\frac{1}{ 1+H_1(g,0) }\frac{g}{\bar{q}_+^2},\, \mathcal{A}_3 \defs \displaystyle\frac{1}{ 1+H_1(g,0) }\frac{1 - \bar{M}_+^2}{\bar{\rho}_+^2 \bar{q}_+^3},\\
  & \mathcal{A}_4 \defs \displaystyle\frac{1}{ 1+H_1(g,0) }\frac{g}{\bar{\rho}_+\bar{q}_+^3},\,
  \mathcal{A}_5 \defs \frac{1}{\bar{\rho}_+\bar{q}_+},\,\mathcal{A}_6 \defs\frac{g}{\bar{q}_+},
\end{align*}
then \eqref{x30} yields that
\begin{equation}\label{eq207}
\begin{aligned}
 \int\int_{\dot{\Omega}_+}   \displaystyle\frac{\mathcal{A}_+ f_2}{ 1+H_1(g,0) } \dif \xi \dif \eta
   = \int_0^{1} \mathcal{A}_+ \mathcal{A}_3 \Big( \pounds_1^* - P_e^*\Big)\dif \eta.
 \end{aligned}
\end{equation}

 Applying the arguments in Theorem 3.1, one can obtain the following estimates for $\delta U^*$ immediately.

\begin{lem}\label{cv}
Suppose that, for given $(\delta U; \delta \psi')$, there exists a $\delta\xi_*$ such that \eqref{eq207} holds. Then there exists a solution $(\delta U^*; \delta {\psi^*})$ to the linearized problem \eqref{eq190}-\eqref{eq194} with the boundary conditions \eqref{eq198}-\eqref{eq199} and \eqref{eq202}-\eqref{eq202+}, and satisfying the following estimates:
\begin{align}\label{eq218}
  &\| \delta p^* \|_{W_\beta^1(\dot{\Omega}_+)}  + \| \delta \theta^* \|_{W_\beta^1(\dot{\Omega}_+)}+\| \delta q^* \|_{\mcc^0(\dot{\Omega}_+)}+ \|\delta q^*\|_{W_\beta^{1-\frac{1}{\beta}}(\dot{\Gamma}_s)} \notag\\
  \leq&  C\Big(\sum_{i=1}^2 \| f_i \|_{L^\beta(\dot{\Omega}_+)} +  \|{f}_3 - {f}_3(\dot{\xi}_*,\eta) \|_{\mcc^0(\dot{\Omega}_+)}+ \| f_4 \|_{\mcc^0(\dot{\Omega}_+)} + H_{2+}(g,\sigma) \\
    & \qquad + \sum_{i=1}^2 \|{\pounds}_i^*\|_{W_\beta^{1-\frac{1}{\beta}}(\dot{\Gamma}_s)} + \| P_e^*\|_{W_\beta^{1-\frac{1}{\beta}}(\Gamma_3)}\Big),\notag\\
  &\| \delta {\psi^*}' \|_{W_\beta^{1-\frac{1}{\beta}}(\dot{\Gamma}_s)}
\leq C\Big(\| \delta U^*\|_{W_\beta^{1-\frac{1}{\beta}}(\dot{\Gamma}_s)}+\| G_3^* \|_{W_\beta^{1-\frac{1}{\beta}}(\dot{\Gamma}_s)} \Big),
\end{align}
where the constant $C$ depends on $\bar{U}_\pm$, $\dot{\xi}_*$, $L$ and $\beta$.
\end{lem}

\section{Well-posedness and contractiveness of the iteration scheme}
In order to carry out the iteration scheme, one needs to find a suitable function space for $(\delta U, \delta \psi')$ such that $\delta\xi_*$ can be determined, and the iteration mapping $\mathbf{\Pi}$ is well defined and contractive.

Let $\epsilon > 0$. Define the function space
\begin{equation}
  \textsl{N} (\epsilon): = \Big\{(\delta U, \delta \psi'): \|\delta U\|_{(\dot{\Omega}_+;\dot{\Gamma}_s)} +  \| \delta \psi' \|_{W_\beta^{1-\frac{1}{\beta}}(\dot{\Gamma}_s)} \leq \epsilon\Big\}.
\end{equation}
To begin with, one needs to show that there exists a $\delta \xi_*$ such that the solvability condition $\eqref{eq207}$ holds. In fact, the following lemma holds.
 \begin{lem}\label{vb}
  Let $0<\sigma\leq g^3$, and $(\delta U - \dot{U}_+;  \delta \psi' - \dot{\psi}')\in \textsl{N}(\frac12 \sigma g^{\frac{3}{2}})$, then there exists a solution $\delta \xi_{*}$ to the equation $\eqref{eq207}$ with the estimate:
\begin{equation}\label{eq100000}
  |\delta \xi_{*}|\leq C_s\frac{\sigma}{g},
\end{equation}
where the constant $C_s$ depends on $p_0, q_0$, $P_I$, $P_E$ and $\|{\theta}^\aleph\|_{\mathcal{C}^0((0,\dot{\xi}_*)\times (0,1))}$.
 \end{lem}

\begin{proof}
%Reference the calculations by Fang-Xin, we can define
Define
\begin{equation}\label{eq208}
\begin{split}
&I(\delta \xi_*, \delta \psi', \delta {U}; \pounds_1^*(\delta U_-))\\
=& - \int\int_{\dot{\Omega}_+} \displaystyle\frac{\mathcal{A}_+   f_2}{ 1+H_1(g,0) } \dif \xi \dif \eta
 +\int_0^{1} \mathcal{A}_+ \mathcal{A}_3 \Big( \pounds_1^* - P_e^*\Big)\dif \eta.
 \end{split}
\end{equation}
It is easy to verify that
\begin{equation}
 I(0, 0, 0;  \dot{\pounds}_1(\dot{U}_-)) = 0.
\end{equation}
We first claim that when $0<\sigma\leq g^3$ and $(\delta U - \dot{U}_+;  \delta \psi' - \dot{\psi}')\in \textsl{N}(\frac12 \sigma g^{\frac{3}{2}})$, one has
\begin{align}
  \frac{\partial I}{\partial(\delta \xi_*)}(0, 0, 0; \dot{\pounds}_1(\dot{U}_-)) \neq 0.
\end{align}
 Then by employing the implicit function theorem, there exists a $\delta \xi_{*}$ to the equation \eqref{eq208}.

To prove this, we show the expansion of $I$ near the state $(0, 0, 0; \dot{\pounds}_1(\dot{U}_-))$ and analyse each term of $I(\delta \xi_*, \delta \psi', \delta {U};  \pounds_1^*(\delta U_-))$.

Since $(\delta U - \dot{U}_+;  \delta \psi' - \dot{\psi}')\in \textsl{N}(\frac12 \sigma g^{\frac{3}{2}})$, by applying Lemma \ref{lem5}, it is easy to see that
\begin{align}\label{kk2}
  \|\delta U\|_{(\dot{\Omega}_+;\dot{\Gamma}_{\mr{s}})} +  \| \delta\psi'\|_{W_\beta^{1-\frac{1}{\beta}}(\dot{\Gamma}_s)}\leq C_* \sigma,
\end{align}
where the constant $C_*$ depends on $\dot{C}_+$.

Now, we consider the first term on the right hand side of \eqref{eq208}. After a routine calculation, one has
\begin{equation}\label{itepo}
\begin{split}
   &-\int\int_{\dot{\Omega}_+} \displaystyle\frac{\mathcal{A}_+   f_2}{ 1+H_1(g,0) } \dif\xi \dif\eta\\
 =& -\int\int_{\dot{\Omega}_+} \displaystyle\frac{\mathcal{A}_+}{ 1+H_1(g,0) }\displaystyle\frac{\delta \xi_*}{L-\psi}\Big(\displaystyle\frac{\cos\theta}{\rho q}\displaystyle\frac{1 - M^2}{\rho q^2}\partial_\xi \delta p\Big)\dif\xi\dif\eta\\
 & + O(1)\sigma^2\cdot \delta\xi_*+ O(1)\sigma^2,
  \end{split}
  \end{equation}
 where $O(1)$ depends only on $C_*$, $\bar{U}_+$ and $\dot{\xi}_*$.
Notice that
\begin{align*}
    &\displaystyle\frac{\delta \xi_*}{L -\psi}\displaystyle\frac{\cos\theta}{\rho q}\displaystyle\frac{1 - M^2}{\rho q^2}\partial_\xi \delta p
    \notag\\
   = & \displaystyle\frac{\delta \xi_*}{L -\psi}\Big(\displaystyle\frac{\cos\theta}{\rho q}\displaystyle\frac{1 - M^2}{\rho q^2} - \displaystyle\frac{1}{\bar{\rho}_+ \bar{q}_+}\cdot\displaystyle\frac{1-\bar{M}_+^2}{\bar{\rho}_+ \bar{q}_+^2}\Big)\partial_\xi \delta p
     \notag\\
    & + \displaystyle\frac{\delta \xi_*}{L -\psi} \displaystyle\frac{1}{\bar{\rho}_+ \bar{q}_+}\cdot\displaystyle\frac{1-\bar{M}_+^2}{\bar{\rho}_+ \bar{q}_+^2}\partial_\xi (\delta p - \dot{p}_+)
    \notag\\
    &- \displaystyle\frac{\delta \xi_*\cdot \int_{\eta}^{1}\delta \psi'(\tau)\dif \tau}{(L -\psi)(L - \xi_*)} \cdot \displaystyle\frac{1}{\bar{\rho}_+ \bar{q}_+}\cdot\displaystyle\frac{1-\bar{M}_+^2}{\bar{\rho}_+ \bar{q}_+^2}\partial_\xi \dot{p}_+
    + \displaystyle\frac{\delta \xi_*}{L -\xi_*} \displaystyle\frac{1}{\bar{\rho}_+ \bar{q}_+}\cdot\displaystyle\frac{1-\bar{M}_+^2}{\bar{\rho}_+ \bar{q}_+^2}\partial_\xi \dot{p}_+\notag\\
    =&  \displaystyle\frac{\delta \xi_*}{L -\xi_*} \displaystyle\frac{1}{\bar{\rho}_+ \bar{q}_+}\cdot\displaystyle\frac{1-\bar{M}_+^2}{\bar{\rho}_+ \bar{q}_+^2}\partial_\xi \dot{p}_+ + O(1)\sigma g^{\frac{3}{2}}\cdot  \delta\xi_* + O(1)\sigma^2 \cdot \delta\xi_*,
  \end{align*}
   where $O(1)$ depends only on $C_*$, $\bar{U}_+$ and $\dot{\xi}_*$.
Then, by applying \eqref{eq67}, it follows that
\begin{equation}\label{eq210}
\begin{split}
  &\int\int_{\dot{\Omega}_+} \displaystyle\frac{\mathcal{A}_+}{ 1+H_1(g,0) } \displaystyle\frac{1-\bar{M}_+^2}{\bar{\rho}_+^2\bar{q}_+^3}\partial_\xi\dot{p}_+ \dif \xi \dif \eta= \int\int_{\dot{\Omega}_+} \partial_\xi(\mathcal{A}_+ \mathcal{A}_3\dot{p}_+)\dif \xi \dif \eta\\
   = &\int\int_{\dot{\Omega}_+}\partial_\eta(\mathcal{A}_+ \dot{\theta}_+)\dif \xi \dif \eta = \int_{\dot{\xi}_*}^{L}\left(\mathcal{A}_+\dot{\theta}_+(\xi,1) - \mathcal{A}_+\dot{\theta}_+(\xi, 0)\right)\dif \xi =0.
  \end{split}
  \end{equation}
Therefore, \eqref{itepo} implies that
\begin{equation}\label{f22}
 - \int\int_{\dot{\Omega}_+}  \displaystyle\frac{\mathcal{A}_+f_2}{ 1+H_1(g,0) } \dif \xi \dif \eta = O(1)\sigma g^{\frac{3}{2}}\cdot  \delta\xi_* + O(1)\sigma^2 \cdot \delta\xi_* + O(1)\sigma^2,
\end{equation}
 where $O(1)$ depends only on $C_*$, $\bar{U}_+$ and $\dot{\xi}_*$.

 Then we estimate $\pounds_1^*$. Recalling \eqref{eq200}, for $j=1,2$, one has
\begin{equation}\label{eq0104}
\begin{split}
  G_j^* = &\alpha_{j+}\cdot \delta U - G_j \Big(U, U_-(\delta \psi' , \xi_*)\Big)\\
   =& \Big(\alpha_{j+} \cdot \delta U +\alpha_{j-} \cdot \delta U_-(\delta \psi',\xi_*) - G_j (U, U_-(\delta \psi' , \xi_*))\Big)\\
   &-\alpha_{j-} \cdot \Big(\delta U_-(\delta \psi',\xi_*) - \dot{U}_-({\xi}_*,\eta)\Big)- \alpha_{j-} \cdot \dot{U}_-({\xi}_*,\eta).
\end{split}
\end{equation}
Notice that
\begin{equation}
\begin{split}
  &\alpha_{j+}\cdot \delta U +\alpha_{j-}\cdot \delta U_-(\delta \psi',\xi_*) - G_j (U, U_-(\delta \psi' , \xi_*))\\
   = &\frac{1}{2} \int_{0}^{1} D^2 G_j (\bar{U}_+ + t \delta U, \bar{U}_- + t\delta U_-)\dif t \cdot(\delta U, \delta U_-)^2\\
   =& O(1)\sigma^2,
   \end{split}
  \end{equation}
  where $O(1)$ depends on $C_*$ and $C_L$. In addition, by employing Lemma \ref{w0} and Lemma \ref{zc}, it holds that
  \begin{equation}
  \begin{split}
   &\delta U_-(\delta \psi',\xi_*) - \dot{U}_-({\xi}_*,\eta)\\
    =& \Big(\delta U_-(\xi_* - \int_{\eta}^{1}\delta \psi'(\tau)\dif \tau,\eta  )- \dot{U}_-(\xi_* - \int_{\eta}^{1}\delta \psi'(\tau)\dif \tau,\eta)\Big)\\
    &+
     \Big(\dot{U}_-(\xi_* - \int_{\eta}^{1}\delta \psi'(\tau)\dif \tau
     ,\eta) - \dot{U}_-({\xi}_*,\eta)\Big)\\
      =& O(1) \sigma^2 ,
      \end{split}
      \end{equation}
  where $O(1)$ depends on $C_L^{\natural}$, $\dot{C}_-$ and $C_*$.
Therefore, \eqref{eq0104} yields that
\begin{align}
  G_j^* = - \alpha_{j-} \cdot \dot{U}_-(\dot{\xi}_*+ \delta \xi_*,\eta) +  O(1) \sigma^2.
\end{align}
Thus, one has
\begin{equation}\label{eq0105}
  \pounds_j^* = \dot{\pounds}_j (\dot{\xi}_*+ \delta \xi_*,\eta) + O(1)\sigma^2.
\end{equation}
Moreover, by applying \eqref{kk2} and \eqref{PE}, it holds that
\begin{equation}\label{q1}
\begin{split}
    P_e^*
    =& q_0^2\cdot\Big(P_I^\sharp( {Y}_0(\eta;g,\sigma)) - P_I^\sharp( {Y}_0(\eta;g,0))\Big)\\
    & + g\cdot \Big(\int_{0}^{Y_L(\eta;g,\sigma)}P_e(\tau;g,\sigma)\dif \tau  - \int_{0}^{Y_L(\eta;g,0)}P_e(t;g,\sigma)\dif t  \Big)\\
    & - g\sigma\cdot q_0^2 \Big( \int_{0}^{Y_0(\eta;g,\sigma)}P_I(\tau)\dif \tau  - \int_{0}^{Y_0(\eta;g,0)}P_I(t)\dif t  \Big)\\
    & + g\sigma\cdot q_0^2 \Big( P_E(Y_L(\eta;g,\sigma))- P_E(Y_L(\eta;g,0)) \Big)\\
   &+ \dot{P}_e^{\sharp}(\eta)\\
    = & \dot{P}_e^{\sharp}(\eta) + O(1)\sigma^2,
\end{split}
\end{equation}
where $O(1)$ depends on $p_0, q_0$, $P_I$, $P_E$, $P_I^{'}$, $P_E^{'}$ and $C_*$.

Furthermore, by applying \eqref{f22}, \eqref{q1} and Lemma \ref{posi},
then \eqref{eq208} implies that
\begin{align}\label{eq211}
    &I(\delta \xi_*, \delta \psi', \delta {U};\pounds_1^*(\delta U_-))\notag\\
    =& \int_0^{1}\mathcal{A}_+ \mathcal{A}_3\Big(\dot{\pounds}_1(\dot{\xi}_* + \delta\xi_* , \eta)- \dot{P}_e^{\sharp}(\eta)\Big)\dif \eta
    + O(1)\sigma g^{\frac{3}{2}}\cdot \delta\xi_* + O(1)\sigma^2\cdot\delta\xi_* +  O(1)\sigma^2\notag\\
    =&\frac{1}{1+H_1(g,0)}\big(\mathcal{R}(\dot{\xi}_* + \delta \xi_*) -\mathcal{P}\big) + O(1)\sigma g^{\frac{3}{2}}\cdot \delta\xi_* + O(1)\sigma^2 \cdot\delta\xi_* +  O(1)\sigma^2\notag\\
    =&\Big(g\cdot K \int_0^1 \Big(\displaystyle\frac{\bar{\rho}_-}{p_0}\Big)^{q_0^2-1}
  \dot{\theta}_-(\dot{\xi}_*,\eta)\dif \eta\Big)\cdot\delta \xi_* + O(1)g\sigma \cdot (\delta\xi_*)^2 + O(1)\sigma g^{\frac{3}{2}} \cdot \delta\xi_* \notag\\
  &+ O(1)\sigma^2 \cdot \delta\xi_*+ O(1)g^2\sigma  \cdot \delta\xi_*  +  O(1)\sigma^2\\
    =&  \Big(g\cdot{K} \int_0^{1} \dot{\theta}_-^{(1)}(\dot{\xi}_*,\eta)\dif \eta \Big)\cdot \delta \xi_* + O(1)g\sigma \cdot (\delta\xi_*)^2 + O(1)g^2\sigma  \cdot \delta\xi_* \notag\\
    &+ O(1)\sigma g^{\frac{3}{2}} \cdot \delta\xi_* + O(1)\sigma^2 \cdot \delta\xi_*+  O(1) \sigma^2,\notag
  \end{align}
which indicates that
\begin{equation}\label{eq214}
\begin{split}
  \displaystyle\frac{\partial I}{\partial \delta \xi_*}(0,0,0;\dot{\pounds}_1(\dot{U}_-)) =& g\cdot  K \int_{0}^{1} \dot{\theta}_-^{(1)}(\dot{\xi}_*,\eta)\dif \eta \\
  &+ O(1)\sigma g^2 + O(1)\sigma g^{\frac{3}{2}} + O(1)\sigma^2,
\end{split}
\end{equation}
where $O(1)$ depends on $p_0, q_0$, $P_I$, $P_E$, $P_I^{'}$, $P_E^{'}$, $C_*$, $C_L^{\natural}$, $C_L$, $\dot{C}_\pm$ and $\dot{\xi}_*$.
Applying $ 0< \sigma\leq g^3$ and $\dot{\theta}_-^{(1)}(\dot{\xi}_*,\eta) \geq  C_0 \sigma$, for sufficiently small $g$, one can deduce that
\begin{equation}\label{eq217}
  \displaystyle\frac{\partial I}{\partial \delta \xi_*}(0, 0,0;\dot{\pounds}_1(\dot{U}_-)) \neq 0.
\end{equation}
Moreover, the expansion $\eqref{eq211}$ implies the estimate $\eqref{eq100000}$.

\end{proof}

Lemma \ref{cv} and Lemma \ref{vb} yield that the existence of the subsonic solution $(\delta U^*; \delta \psi^{*'}; \delta \xi_*)$ to the linearized problem \eqref{eq190}-\eqref{G3} as $(\delta U - \dot{U}_+; \delta \psi'- \dot{\psi}')$ in the function space $\textsl{N}(\frac12 \sigma g^{\frac{3}{2}})$ with the sufficiently small $g$ and $\sigma$. In the sequel arguments, we will prove $(\delta U^*; \delta \psi^*)$ also satisfies $(\delta U^* - \dot{U}_+; \delta \psi^{*'}- \dot{\psi}')\in \textsl{N}(\frac12  \sigma g^{\frac{3}{2}})$ as long as $g$ and $\sigma$ are sufficiently small.
Define
\begin{align*}
  \mathscr{N}(\dot{U}_+; \dot{\psi}')\defs \Big\{(\delta U; \delta \psi'): (\delta U - \dot{U}_+;  \delta \psi' - \dot{\psi}')\in \textsl{N}\Big(\frac12  \sigma g^{\frac{3}{2}}\Big)\Big\}.
\end{align*}
Then the following lemma holds.
\begin{lem}\label{vn}
Let $0<\sigma \leq g^3$,
if $(\delta U ;  \delta \psi' )\in \mathscr{N}(\dot{U}_+; \dot{\psi}')$, then there exists a solution $(\delta {U}^*; \delta {\psi^*}')$ to the linearized problem \eqref{eq190}-\eqref{G3} and satisfies
 $(\delta {U}^* ; \delta {\psi^*}' )\in \mathscr{N}(\dot{U}_+; \dot{\psi}')$.
 \end{lem}

\begin{proof}
It suffices to verify that $(\delta {U}^* ; \delta {\psi^*}' )\in \mathscr{N}(\dot{U}_+; \dot{\psi}')$.

Recalling the definitions \eqref{iteration1}-\eqref{initial1}, one has
\begin{align}
  (\delta {U}^* - \dot{U}_+ ; \delta {\psi^*}' - \dot{\psi}' ) = \mathscr{T}(\mathscr{F} - \dot{\mathscr{F}};\mathscr{G} - \dot{\mathscr{G}}; P_{{e}}^* - \dot{P}_e^\sharp ),
\end{align}
where
\begin{align*}
  \mathscr{F}\defs& (f_1(\delta U, \delta \psi', \delta \xi_{*})+ H_{2+}(g,\sigma),
  f_2(\delta U, \delta \psi', \delta \xi_{*}), f_3(\delta U),f_4(\delta U, \delta \psi', \delta \xi_{*})),\\
  \dot{\mathscr{F}}\defs& (H_{2+}(g,\sigma),0,-B(\bar{U}_+),0),\\
  \mathscr{G}\defs &(\pounds_j^*(\delta U,\delta U_-, \delta \psi', \delta \xi_*); j=1,2,3),\\
  \dot{\mathscr{G}}\defs& (\dot{\pounds}_1, \dot{\pounds}_2, \dot{\pounds}_3).
\end{align*}
Similar as Lemma \ref{cv}, one has
\begin{equation}\label{l}
\begin{split}
   &\| \delta U^*- \dot{U}_+ \|_{(\dot{\Omega}_+;\dot{\Gamma}_s)} + \| \delta {\psi^*}' -\dot{\psi}' \|_{W_\beta^{1-\frac{1}{\beta}}(\dot{\Gamma}_s)}\\
    \leq & C\Big(\sum_{i=1}^2 \| f_i \|_{L^\beta(\dot{\Omega}_+)} +  \|{f}_3  +B(\bar{U}_+)
     \|_{\mcc^0(\dot{\Omega}_+)}+ \| f_4 \|_{\mcc^0(\dot{\Omega}_+)}\\
     &\quad \quad + \sum_{i=1}^3 \|{\pounds}_i^* - \dot{\pounds}_i\|_{W_\beta^{1-\frac{1}{\beta}}(\dot{\Gamma}_s)} + \| P_e^* - \dot{P}_e^\sharp\|_{W_\beta^{1-\frac{1}{\beta}}(\Gamma_3)}\Big).
\end{split}
\end{equation}
Now, we analyze the terms on the right hand side of \eqref{l}.
Applying the expression of $f_1$, one has
\begin{align}\label{eq293}
 &\|f_1(\delta U, \delta \psi', \delta \xi_{*})\|_{L^\beta(\dot{\Omega}_+)}\notag\\
 \leq&\Big\|\frac{\sin {\theta}}{{\rho}{q}}\partial_\xi \delta p -\Big( {q} \cos \theta -  \bar{q}_+ \Big)\partial_\xi \delta \theta - \Big(g\displaystyle\frac{\cos\theta}{q}- \frac{g}{\bar{q}_+}+ \displaystyle\frac{g}{\bar{q}_+^2}\delta q \Big)\Big\|_{L^\beta(\dot{\Omega}_+)}\notag\\
  &+ \Big\|(1+H_1(g,\sigma)) \frac{L - \xi}{L- \psi ({\eta})} \cdot \delta \psi'({\eta}) \partial_\xi \delta p  \Big\|_{L^\beta(\dot{\Omega}_+)}\\
  &+ \Big\|\displaystyle\frac{\delta \xi_* - \int_{\eta}^1 \delta \psi'(\tau)\dif \tau}{L -\psi(\eta)}\Big(\frac{\sin \theta}{\rho q} \partial_\xi \delta p -
  q \cos \theta\partial_\xi \delta \theta\Big) \Big\|_{L^\beta(\dot{\Omega}_+)}\notag\\
  & + \Big\| \Big(H_1(g,\sigma) - H_1(g,0) \Big)\partial_\eta \delta p\Big\|_{L^\beta(\dot{\Omega}_+)}\notag\\
  \leq& C\Big\{\|\delta U\|_{L^\infty(\dot{\Omega}_+)}\Big(\|(\delta p, \delta \theta) \|_{W_\beta^1(\dot{\Omega}_+)}+ |\delta \xi_*| \|\delta p \|_{W_\beta^1(\dot{\Omega}_+)}\Big) + g(\|\delta U\|_{L^\infty(\dot{\Omega}_+)})^2 \notag\\
  &\qquad +\|\delta \psi'\|_{L^\infty(\dot{\Gamma}_s)}\|(\delta p,\delta \theta) \|_{W_\beta^1(\dot{\Omega}_+)}
   +|\delta \xi_*| \|\delta \theta \|_{W_\beta^1(\dot{\Omega}_+)} + \sigma\|\delta p \|_{W_\beta^1(\dot{\Omega}_+)} \Big\}\notag\\
 \leq&  C \Big(\sigma^2 + C_s \sigma^2 \frac{\sigma}{g} + C_s \sigma   \frac{\sigma}{g} \Big)\notag\\
 \leq & C_{1}^{\flat}\frac{\sigma^2}{g}.\notag
\end{align}
Similarly, one has
\begin{align}
  \|f_2(\delta U, \delta \psi', \delta \xi_{*})\|_{L^\beta(\dot{\Omega}_+)}\leq  C_{2}^{\flat} \frac{\sigma^2}{g}.
\end{align}
Moreover, since
  \begin{equation}
  \begin{split}
    f_3(\delta U) + B(\bar{U}_+) & = -\left(B(U) - B(\bar{U}_+) - (\bar{q}_+ \delta q + \displaystyle\frac{1}{\bar{\rho}_+}\delta p)  \right)\\
    & = - \int_{0}^{1} D_{U}^2 B(\bar{U}_+ + t\delta U)\dif t \cdot (\delta U)^2,
  \end{split}
  \end{equation}
then one can obtain
\begin{equation}
  \|f_3 +B(\bar{U}_+)\|_{\mcc^0(\dot{\Omega}_+)}
  \leq C_{3}^{\flat}\sigma^2.
\end{equation}
Moreover,
\begin{align}
  \| f_4\|_{\mcc^0(\dot{\Omega}_+)}\leq C \Big(g  \sigma |\delta \xi_*| + g\sigma^2\Big) \leq C_{4}^{\flat} \sigma^2.
\end{align}
The constants $C_{i}^{\flat}$, $(i=1,2,3,4)$ depend on $\dot{\xi}_*$, $C_s$ and $C_*$.

On the exit $\Gamma_3$, recalling \eqref{q1}, one has
\begin{equation}
   \|P_e^* -\dot{P}_e^\sharp \|_{W_\beta^{1-\frac{1}{\beta}}(\Gamma_3)}\leq  C \sigma^2 .
\end{equation}
Finally, on the fixed boundary $\dot{\Gamma}_s$, for $j=1,2$, employing \eqref{eq0105}, one has
\begin{align}\label{eq535}
     \|{\pounds}_i^* - \dot{\pounds}_i\|_{W_\beta^{1-\frac{1}{\beta}}(\dot{\Gamma}_s)}
      \leq C\Big(\|\partial_\xi \dot{U}_-\|_{{\mcc}^{1,\alpha}(\bar{\Omega})}\cdot|\delta\xi_*|+ \sigma^2\Big)
       \leq C_{5}^{\flat} \frac{\sigma^2}{g}.
  \end{align}
In addition, a similar argument yields that
\begin{align}
     \|{\pounds}_3^* - \dot{\pounds}_3\|_{W_\beta^{1-\frac{1}{\beta}}(\dot{\Gamma}_s)}
       \leq C_{6}^{\flat} \frac{\sigma^2}{g}.
  \end{align}
  The constants $C_{i}^{\flat}$, $(i=5,6)$ depend on $\dot{C}_-^\natural$, $C_L^{\natural}$, $C_L$, $\dot{\xi}_*$, $C_s$ and $C_*$.

Therefore, for sufficiently small $g$, one has
\begin{equation}
\begin{split}
&\| \delta U^* - \dot{U}_+ \|_{(\dot{\Omega}_+;\dot{\Gamma}_s)}+ \| \delta {\psi^*}' -\dot{\psi}' \|_{W_\beta^{1-\frac{1}{\beta}}(\dot{\Gamma}_s)}\\
\leq& C \Big(\frac{\sigma^2}{g}+ \sigma^2\Big)
 \leq C \sigma  g^{\frac{3}{2}}\Big(g^{\frac{1}{2}} + g^{\frac{3}{2}}\Big) \leq \frac12 \sigma g^{\frac{3}{2}}.
\end{split}
\end{equation}

\end{proof}
The Theorem \ref{ppo} will be proved as long as we prove that the mapping $\mathbf{\Pi}$ is contractive. The following lemma completes the proof.
 \begin{lem}\label{thm63}
  Let $0<\sigma \leq g^3$, then the mapping $\mathbf{\Pi}$ is contractive.
 \end{lem}

\begin{proof}

Suppose that $(\delta U_k; \delta \psi_k^{'})\in \mathscr{N}(\dot{U}_+; \dot{\psi}')$, $(k=1,2)$, then by employing Lemma \ref{vb} and Lemma \ref{vn}, there exists $\delta \xi_{*k}$ satisfying the estimate \eqref{eq100000} and $(\delta {U}_k^* ; \delta \psi_k^{*'})\in \mathscr{N}(\dot{U}_+; \dot{\psi}')$ such that
\begin{align}\label{iteration2}
  (\delta {U}_k^*; \delta {\psi_k^*}';\delta \xi_{*k}) = \mathscr{T}_e(\mathscr{F}_k; \mathscr{G}_k; P_{{e}}^*(\eta; \delta U_k)),
\end{align}
where
\begin{align*}
  \mathscr{F}_k\defs& (f_1(\delta U_k, \delta \psi_k^{'}, \delta \xi_{*k})+ H_{2+}(g,\sigma), f_2(\delta U_k, \delta \psi_k^{'}, \delta \xi_{*k}), f_3(\delta U_k), f_4(\delta U_k, \delta \psi_k^{'}, \delta \xi_{*k})),\\
  \mathscr{G}_k\defs& (\pounds_j^*(\delta U_k,\delta U_-,  \delta \psi_k^{'}, \delta \xi_{*k}); j=1,2,3).
\end{align*}
  To prove the mapping $\mathbf{\Pi}$ is contractive, it suffices to verify that, for sufficiently small $g$ and $\sigma$,
\begin{equation}\label{ee}
\begin{split}
     &\| \delta U_2^* - \delta U_1^* \|_{(\dot{\Omega}_+;\dot{\Gamma}_s)} + \| \delta \psi_{2}^{*'} - \delta \psi_{1}^{*'}\|_{W_\beta^{1-\frac{1}{\beta}}(\dot{\Gamma}_s)}\\
     \leq& \frac12\Big(\| \delta U_2 - \delta U_1\|_{(\dot{\Omega}_+;\dot{\Gamma}_s)} + \| \delta {\psi_2}' - \delta {\psi_1}'\|_{W_\beta^{1-\frac{1}{\beta}}(\dot{\Gamma}_s)}\Big).
   \end{split}
   \end{equation}
Applying \eqref{iteration2}, one has
\begin{align}
  &(\delta U_2^* - \delta U_1^* ; \delta \psi_{2}^{*'} - \delta \psi_{1}^{*'})\label{iteration3}\\
  =&  \mathscr{T}_e(\mathscr{F}_2 - \mathscr{F}_1; \mathscr{G}_2 -\mathscr{G}_1 ; P_{{e}}^*(\eta; \delta U_2) - P_{{e}}^*(\eta; \delta U_1))\notag.
\end{align}
Since the right hand side of \eqref{iteration3} includes $\delta \xi_{*k}$, which is determined by Lemma \ref{vb} with given $(\delta U_k^*; \delta \psi_{k}^{*'})$, thus one has to estimate $|\delta \xi_{*2} - \delta \xi_{*1}|$ first.
%
%Since $(\delta U_{k}- \dot{U}_+; \delta \psi_{k}^{'} - \dot{\psi}')\in \textsl{N}(\frac12 \sigma g^{\frac43})$, $k =1,2$, then by Lemma \ref{vb} and Lemma \ref{vn}, there exists $\delta \xi_{*k}$ satisfying the estimate \eqref{eq100000} and $(\delta {U}_k^* - \dot{U}_+ ; \delta \psi_k^{*'} - \dot{\psi}' )\in \textsl{N}(\frac12 \sigma g^{\frac43})$.
%\begin{align}
% &\| \delta U_2^* - \delta U_1^* \|_{(\dot{\Omega}_+;\dot{\Gamma}_s)} + \| \delta \psi_{2}^{*'} - \delta \psi_{1}^{*'}\|_{W_\beta^{1-\frac{1}{\beta}}(\dot{\Gamma}_s)}\notag\\
%  \leq & C \Big(\sum_{i=1}^2 \| f_i^{(2)} - f_i^{(1)} \|_{L^\beta ( \dot{\Omega}_+)}+ \sum_{i=1}^{3}\| \pounds_i^{*2} - \pounds_i^{*1} \|_{W_\beta^{1-\frac{1}{\beta}}(\dot{\Gamma}_s)}+
%    \|f_3^{(2)} - f_3^{(1)} \|_{\mcc^0(\dot{\Omega}_+)}\notag\\
%    & +  \| f_4^{(2)} - f_4^{(1)}\|_{\mcc^0(\dot{\Omega}_+)}+\| \sigma P_e^{*2} - \sigma P_e^{*1}  \|_{W_\beta^{1-\frac{1}{\beta}}({\Gamma}_3)}\Big),
%  \end{align}
%  where $f_i^{(k)}\defs f_i(\delta U_k, \delta \psi_k^{'}, \delta \xi_{*k})$, $f_3^{(k)}: = f_3(\delta U_k)$, $\pounds_i^{*k}\defs \pounds_i^{*}(\delta U_k, \delta U_-, \delta \psi_k^{'}, \delta \xi_{*k})$,
%  $P_e^{*k} \defs P_e^{*}( Y_+(L,\eta;\delta U_k))$.
%
%For above equation, it is necessary to estimate $|\delta \xi_{*2} - \delta \xi_{*1}|$ first.
\begin{equation}\label{eq304}
\begin{split}
 0 =& I(\delta \xi_{*2},  \delta {U}_2,\delta {\psi_2}',  \pounds_1^*(\delta U_-)) - I(\delta \xi_{*1},  \delta {U}_1,\delta {\psi_1}',  \pounds_1^*(\delta U_-))
 \\
 = &I(\delta \xi_{*2}, \delta U_2 , \delta {\psi_2}' ,  \pounds_1^*(\delta U_-)) - I(\delta \xi_{*1}, \delta U_2 , \delta {\psi_2}' ,  \pounds_1^*(\delta U_-))
 \\
& + I(\delta \xi_{*1}, \delta U_2 , \delta {\psi_2}' ,  \pounds_1^*(\delta U_-)) -  I(\delta \xi_{*1}, \delta U_1 , \delta {\psi_1}',  \pounds_1^*(\delta U_-))
\\
= &\int_{0}^{1} \frac{\partial I }{\partial (\delta \xi_*)} (\delta \xi_{*t}, \delta U_2,\delta {\psi_2}',  \pounds_1^*(\delta U_-))\dif t \cdot (\delta \xi_{*2} - \delta \xi_{*1})
\\
& + \int_{0}^{1} \nabla_{(\delta U, \delta \psi')} I (\delta \xi_{*1} , \delta U_t, \delta {\psi_t}' , \pounds_1^*(\delta U_-))\dif t \cdot (\delta U_2 - \delta U_1, \delta {\psi_2}' - \delta {\psi_1}' ),
\end{split}
\end{equation}
where
\begin{align}
  &\delta \xi_{*t} : = t \delta \xi_{*2} + (1 - t) \delta \xi_{*1},\quad
    \delta U_t : = t \delta U_2 + (1 - t) \delta U_1,
    \notag \\
    & \delta {\psi_t}' : = t \delta {\psi_2}' + (1 - t)\delta {\psi_1}'.
\end{align}
Similar calculations as in Lemma \ref{vb}, one has
\begin{align}\label{eq305}
&\frac{\partial I }{\partial (\delta \xi_*)} (\delta \xi_{*t}, \delta U_2,\delta {\psi_2}',  \pounds_1^*(\delta U_-))\notag\\
 = &\frac{\partial I }{\partial (\delta \xi_*)}(0 , 0, 0;  \dot{\pounds}_1(\dot{U}_-))\\
 &+ \int_{0}^{1} \nabla_{(\delta \xi_*, \delta U, \delta \psi' , \pounds_1^*(\delta U_-))} \frac{\partial I}{\partial (\delta \xi_*)}(s\delta \xi_{*t}, s\delta U_{2}, s\delta {\psi_{2}}', s \pounds_1^*(\delta U_-) + (1-s)  \dot{\pounds}_1^*(\dot{U}_-))\dif s\notag\\
&\cdot (\delta \xi_{*t}, \delta U_2  , \delta {\psi_2}',  \pounds_1^*(\delta U_-) -  \dot{\pounds}_1^*(\dot{U}_-))\notag\\
 =& g\cdot K \int_{0}^{1} \dot{\theta}_-^{(1)}(\dot{\xi}_*,\eta)\dif \eta+ O(1)\sigma g^2 + O(1)\sigma g^{\frac{3}{2}} + O(1)\sigma^2.\notag
\end{align}
Moreover,
\begin{align}
  \nabla_{(\delta U, \delta \psi')} I (\delta \xi_{*1} , \delta U_t, \delta {\psi_t}' , \pounds_1^*(\delta U_-)) = O(1)\sigma.
\end{align}
Therefore, \eqref{eq304} implies that
\begin{equation}\label{eq306}
\begin{split}
|\delta \xi_{*2} - \delta \xi_{*1}| & \leq C\Big|\frac{\sigma \Big(\parallel \delta U_2 - \delta U_1\parallel_{W_\beta^{1-\frac{1}{\beta}}(\dot{\Gamma}_s)} + \parallel \delta {\psi_2}' - \delta {\psi_1}'\parallel_{W_\beta^{1-\frac{1}{\beta}}(\dot{\Gamma}_s)}\Big)} {Kg\cdot \int_{0}^{1} \dot{\theta}_-^{(1)}(\dot{\xi}_*,\eta)\dif \eta+ O(1)\sigma g^2 + O(1)\sigma g^{\frac{3}{2}} + O(1)\sigma^2}\Big|\\
&\leq C_{1}^{\sharp} g^{-1}\cdot \Big(\parallel \delta U_2 - \delta U_1\parallel_{(\dot{\Omega}_+;\dot{\Gamma}_s)} + \parallel \delta {\psi_2}' - \delta {\psi_1}'\parallel_{W_\beta^{1-\frac{1}{\beta}}(\dot{\Gamma}_s)}\Big).
\end{split}
\end{equation}
Applying \eqref{eq306}, one can deduce that
\begin{align}
&\| f_1(\delta U_2 , \delta {\psi_2}',\delta \xi_{*2}) -f_1(\delta U_1 , \delta {\psi_1}',\delta \xi_{*1})
{\|_{L^\beta({\dot{\Omega}}_+)}}\\
\leq& C\Big(\big(\|\delta U_1\|_{L^\infty(\dot{\Omega}_+)} +\|\delta U_2\|_{L^\infty(\dot{\Omega}_+)}+\|\delta \psi'\|_{L^\infty(\dot{\Gamma}_s)}\big)\notag\\
&\quad \cdot\big(\parallel \delta U_2 - \delta U_1\parallel_{(\dot{\Omega}_+;\dot{\Gamma}_s)} + \parallel \delta {\psi_2}' - \delta {\psi_1}'\parallel_{W_\beta^{1-\frac{1}{\beta}}(\dot{\Gamma}_s)}+|\delta \xi_{*2} -\delta \xi_{*1} |\big) \notag\\
&\quad + (|\delta \xi_{*1} | +|\delta \xi_{*2} |)\|\delta U_2 -\delta U_1 \|_{W_\beta^1(\dot{\Omega}_+;\dot{\Gamma}_s)}
 + \sigma \|\delta p_2 - \delta p_1 \|_{W_\beta^1(\dot{\Omega}_+)} \Big)\notag\\
 \leq&  C\big(  \sigma + \frac{\sigma}{g}\big)\cdot  \big(\parallel \delta U_2 - \delta U_1\parallel_{(\dot{\Omega}_+;\dot{\Gamma}_s)} + \parallel \delta {\psi_2}' - \delta {\psi_1}'\parallel_{W_\beta^{1-\frac{1}{\beta}}(\dot{\Gamma}_s)}\big) \notag\\
 \leq & C_{2}^{\sharp} \frac{\sigma}{g}\cdot \big(\parallel \delta U_2 - \delta U_1\parallel_{(\dot{\Omega}_+;\dot{\Gamma}_s)} + \parallel \delta {\psi_2}' - \delta {\psi_1}'\parallel_{W_\beta^{1-\frac{1}{\beta}}(\dot{\Gamma}_s)}\big).\notag
\end{align}
Similarly, one has
\begin{align}
&\| f_2(\delta U_2 , \delta {\psi_2}',\delta \xi_{*2}) -f_2(\delta U_1 , \delta {\psi_1}',\delta \xi_{*1})
{\|_{L^\beta({\dot{\Omega}}_+)}}\\
\leq&  C_{3}^{\sharp} \frac{\sigma}{g}\cdot \big(\parallel \delta U_2 - \delta U_1\parallel_{(\dot{\Omega}_+;\dot{\Gamma}_s)} + \parallel \delta {\psi_2}' - \delta {\psi_1}'\parallel_{W_\beta^{1-\frac{1}{\beta}}(\dot{\Gamma}_s)} \big),\notag\\
&\| f_3(\delta U_2 ) -f_3(\delta U_1 )
{\|_{\mcc^0({\dot{\Omega}}_+)}}\\
\leq&  C_{4}^{\sharp} \sigma\cdot  \| \delta U_2 - \delta U_1\|_{(\dot{\Omega}_+;\dot{\Gamma}_s)},\notag\\
&\| f_4(\delta U_2 ) -f_4(\delta U_1 )
{\|_{\mcc^0({\dot{\Omega}}_+)}}\\
\leq& C_{5}^{\sharp}\sigma\cdot \big(\parallel \delta U_2 - \delta U_1\parallel_{(\dot{\Omega}_+;\dot{\Gamma}_s)} + \parallel \delta {\psi_2}' - \delta {\psi_1}'\parallel_{W_\beta^{1-\frac{1}{\beta}}(\dot{\Gamma}_s)}\big).\notag
\end{align}
On the free boundary $\dot{\Gamma}_s$, for $j=1,2$, one has
\begin{equation}
\begin{split}
    &\| \pounds_i^{*2} - \pounds_i^{*1} \|_{W_\beta^{1-\frac{1}{\beta}}(\dot{\Gamma}_s)}\\
      \leq& C\|\partial_\xi \dot{U}_-\|_{{\mcc}^{1,\alpha}(\bar{\Omega})}\cdot|\delta\xi_{*2} - \delta \xi_{*1}|\\
       \leq &C_{6}^{\sharp} \frac{\sigma}{g}\cdot \big(\parallel \delta U_2 - \delta U_1\parallel_{(\dot{\Omega}_+;\dot{\Gamma}_s)} + \parallel \delta {\psi_2}' - \delta {\psi_1}'\parallel_{W_\beta^{1-\frac{1}{\beta}}(\dot{\Gamma}_s)}  \big).
       %\Notag\\
%       \Leq & C G^{\Frac32}\Big(\Parallel \Delta U_2 - \Delta U_1\Parallel_{(\Dot{\Omega}_+;\Dot{\Gamma}_S)} + \Parallel \Delta {\Psi_2}' - \Delta {\Psi_1}'\Parallel_{W_\Beta^{1-\Frac{1}{\Beta}}(\Dot{\Gamma}_S)}  \Big),
  \end{split}
  \end{equation}
Similarly,
\begin{equation}
\begin{split}
    &\| \pounds_3^{*2} - \pounds_3^{*1} \|_{W_\beta^{1-\frac{1}{\beta}}(\dot{\Gamma}_s)}\\
    \leq& C_{7}^{\sharp} \frac{\sigma}{g} \cdot \big(\parallel \delta U_2 - \delta U_1\parallel_{(\dot{\Omega}_+;\dot{\Gamma}_s)} + \parallel \delta {\psi_2}' - \delta {\psi_1}'\parallel_{W_\beta^{1-\frac{1}{\beta}}(\dot{\Gamma}_s)}  \big).
\end{split}
\end{equation}
Moreover,
\begin{align}\label{differP}
  \|  P_e^{*2} -  P_e^{*1}  \|_{W_\beta^{1-\frac{1}{\beta}}({\Gamma}_3)}\leq C_{8}^{\sharp}\cdot \sigma  \| \delta U_2 - \delta U_1\|_{(\dot{\Omega}_+;\dot{\Gamma}_s)}.
\end{align}
By employing \eqref{eq306}-\eqref{differP}, it holds that
   \begin{equation}
   \begin{split}
     &\| \delta U_2^* - \delta U_1^* \|_{(\dot{\Omega}_+;\dot{\Gamma}_s)} + \| \delta \psi_{2}^{*'} - \delta \psi_{1}^{*'}\|_{W_\beta^{1-\frac{1}{\beta}}(\dot{\Gamma}_s)}\\
     \leq & C \frac{\sigma}{g} \cdot  \big(\| \delta U_2 - \delta U_1\|_{(\dot{\Omega}_+;\dot{\Gamma}_s)} + \| \delta {\psi_2}' - \delta {\psi_1}'\|_{W_\beta^{1-\frac{1}{\beta}}(\dot{\Gamma}_s)}\big)\\
     \leq & C_{9}^{\sharp}g^2\cdot  \big(\| \delta U_2 - \delta U_1\|_{(\dot{\Omega}_+;\dot{\Gamma}_s)} + \| \delta {\psi_2}' - \delta {\psi_1}'\|_{W_\beta^{1-\frac{1}{\beta}}(\dot{\Gamma}_s)}\big),
     \end{split}
   \end{equation}
   The constants $C_{j}^{\sharp}$, $(j=1,\cdots 9)$ depend on $p_0, q_0$, $P_I$, $P_E$, $\beta$, $L$ and $\dot{\xi}_*$. Therefore, we can choose $ C_{9}^{\sharp}g^2 = \frac12$ for sufficiently small $g$.
 Thus, the proof of Lemma \ref{thm63} is completed.
\end{proof}

%\appendix
%\section{The higher regularities of the subsonic solution}

\section*{Acknowlegements}
The research of Beixiang Fang was supported in part by Natural Science
Foundation of China under Grant Nos. 11971308, and 11631008. The research of Xin Gao was supported in part by China Scholarship Council (No.201906230072).

%????????????

\end{document}